\documentclass{article}[9pt]
\usepackage[]{hyperref}
\usepackage{mathptmx, mathptmx, enumerate}
\usepackage{amsmath,amssymb,amsthm,amsbsy} %
\usepackage{graphicx}
\usepackage{epstopdf}
\usepackage{mathrsfs}
\usepackage{mathtools}
\usepackage{bm}
\usepackage{yhmath}
\usepackage{color}
\usepackage{arydshln}
\usepackage[hang,small]{caption}
\usepackage[subrefformat=parens]{subcaption}
\usepackage{comment}

\usepackage{algorithm}
\usepackage{algorithmicx}
\usepackage{algpseudocode}

\algdef{SE}[DOWHILE]{Do}{doWhile}{\algorithmicdo}[1]{\algorithmicwhile\ #1}

\mathtoolsset{showonlyrefs=true}

\newtheorem{theorem}{Theorem}
\newtheorem{corollary}{Corollary}
\newtheorem{lemma}{Lemma}
\newtheorem{proposition}{Proposition}
\theoremstyle{definition}
\newtheorem{definition}{Definition}
\newtheorem{example}{Example}
\newtheorem{remark}{Remark}

\newtheorem{assumption}{Assumption}

\newtheorem{fact}{Fact}
\newtheorem{problem}{Problem}

\numberwithin{equation}{section}

\newcommand{\setN}{\mathbb{N}}
\newcommand{\setR}{\mathbb{R}}
\newcommand{\setPR}{\mathbb{R}_{++}}
\newcommand{\setNNR}{\mathbb{R}_{+}}

\newcommand{\ip}[2]{\left\langle #1 , #2 \right\rangle}
\newcommand{\norm}[1]{\left\| #1 \right\|}

\newcommand{\ran}{\operatorname{ran}}

\newcommand{\nullsp}{\operatorname{null}}
\newcommand{\setLO}[2]{\mathcal{B}\left( #1 , #2 \right)}

\newcommand{\Id}{\operatorname{Id}}
\newcommand{\zeroMatrix}{\mathrm{O}}

\newcommand{\lipconst}[1]{\beta_{#1}}

\newcommand{\ri}{\operatorname{ri}}

\DeclareFontFamily{U}{mathx}{\hyphenchar\font45}
\DeclareFontShape{U}{mathx}{m}{n}{
      <5> <6> <7> <8> <9> <10>
      <10.95> <12> <14.4> <17.28> <20.74> <24.88>
      mathx10
      }{}
\DeclareSymbolFont{mathx}{U}{mathx}{m}{n}
\DeclareMathSymbol{\bigtimes}{1}{mathx}{"91}

\newcommand{\dom}{\operatorname{dom}}

\newcommand{\prox}{\operatorname{Prox}}
\newcommand{\fix}{\operatorname{Fix}}
\newcommand{\rec}{\operatorname{rec}}
\newcommand{\gra}{\operatorname{gra}}

\newcommand{\spH}{\mathcal{H}}
\newcommand{\spK}{\mathcal{K}}

\newcommand{\spX}{\mathcal{X}}
\newcommand{\spY}{\mathcal{Y}}
\newcommand{\spZ}{\mathcal{Z}}
\newcommand{\sptildeZ}{\widetilde{\mathcal{Z}}}

\newcommand{\spfrakZ}{\mathfrak{Z}}

\newcommand{\opP}{\mathfrak{P}}
\newcommand{\opL}{\mathfrak{L}}

\newcommand{\opC}{\mathfrak{C}}

\newcommand{\setS}{\mathcal{S}}
\newcommand{\minimize}[0]{\operatornamewithlimits{minimize\ }}
\newcommand{\argmin}{\operatornamewithlimits{argmin\ }}

\newcommand{\xstar}{{x^\star}}

\newcommand{\xdiamond}{x^\diamond}
\newcommand{\vdiamond}{v^\diamond}
\newcommand{\wdiamond}{w^\diamond}
\newcommand{\zdiamond}{z^\diamond}

\newcommand{\lrangle}[1]{\left\langle #1 \right\rangle}

\DeclareMathAlphabet\mathbfcal{OMS}{cmsy}{b}{n}

\newcommand{\Cz}{\Delta}

\newcommand{\fBeforeExtension}{\mathfrak{f}}

\begin{document}

\title{A Linearly involved Generalized Moreau Enhanced Model with Non-quadratic Smooth Convex Data Fidelity Functions}
\author{Wataru Yata, Keita Kume and Isao Yamada}
\maketitle
\renewcommand{\thefootnote}{\fnsymbol{footnote}}
\footnotetext[0]{This work was supported by JSPS Grants-in-Aid (19H04134, 23KJ0945, 24K23885).}
\renewcommand{\thefootnote}{\arabic{footnote}}

\begin{abstract}
      In this paper, we introduce an overall convex model incorporating a nonconvex regularizer.
      The proposed model is designed by extending the least squares term in the constrained LiGME model [Yata Yamagishi Yamada 2022] to  fairly general smooth convex functions for flexible utilization of non-quadratic data fidelity functions.
      Under an overall convexity condition for the proposed model, we present sufficient conditions for the existence of a minimizer of the proposed model and an inner-loop free algorithm with guaranteed convergence to a global minimizer of the proposed model.
      To demonstrate the effectiveness of the proposed model and algorithm, we conduct numerical experiments in scenarios of \textit{Poisson denoising problem} and \textit{simultaneous declipping and denoising problem}.
\end{abstract}

\section{Introduction} 
Estimation of a target signal $x^\star\in\spX$ from its noisy observation $y\in\spY$ following
\begin{align}
  \label{eq:linear-regression}
  y = Ax^\star +\varepsilon\in\spY
\end{align} 
is a central goal in inverse problems and signal processing
(see, e.g., \cite{nashed2002,theodoridis2020}), where $\spX$ and $\spY$ are finite-dimensional real Hilbert spaces (all Hilbert spaces in this paper should be understood as finite-dimensional real Hilbert spaces, i.e., Euclidean spaces), $A\in\setLO{\spX}{\spY}$ is a known linear operator and $\varepsilon\in\spY$ is noise.
A standard approach formulates the following convex optimization model~\eqref{eq:convex-opt} (see Section~\ref{sec:notation} and \ref{sec:Preliminaries-convex} for notation and technical terms) and assigns its minimizer to the estimate of~$\xstar$.
\begin{problem}
  \label{prob:convex-opt}
  For finite-dimensional real Hilbert spaces $\spX,\spY,\spZ$ and $\spfrakZ$, let 
  (a) $A\in \setLO{\spX}{\spY}$, and $f\in\Gamma_0(\spY)$ be differentiable over $\spY$ such that $\nabla f$ is $\lipconst{\nabla f}$-Lipschitz continuous over $\spY$; 
  (b) $\Psi\in\Gamma_0(\spZ)$ be \textit{coercive} and \textit{prox-friendly} and $(\mu, \opL)\in\setPR\times\setLO{\spX}{\spZ}$; (c) $\Cz\subset \spfrakZ$ be a \textit{simple} closed convex set and $\opC\in\setLO{\spX}{\spfrakZ}$.
  Assume the existence of $x\in\spX$ satisfying  $\opC x\in \Cz$ and $\Psi\circ\opL(x)< \infty$, i.e., $\dom(\Psi\circ \opL)\cap \opC^{-1}(\Cz)\neq\emptyset$. Then consider 
  \begin{equation}
    \label{eq:convex-opt}
    \minimize_{ \opC x \in\Cz} J_{\Psi\circ\opL}(x) \coloneqq f\circ A(x) + \mu \Psi\circ\opL (x).
  \end{equation}
\end{problem}
In the cost function $J_{\Psi\circ\opL}$ in \eqref{eq:convex-opt}, $f\circ A$ is a data fidelity function designed with $f$, called in this paper \emph{observation loss function}, according to the observation model \eqref{eq:linear-regression}.
A typical example of observation loss functions is the least squares function $f\coloneqq\frac{1}{2}\norm{y - \cdot}^2_{\spY}$
which is mainly motivated by the fact that this fidelity function $f$ is the negative log-likelihood function of the observation model \eqref{eq:linear-regression} under the assumption that $\varepsilon$ is the additive white Gaussian noise.
For non-Gaussian noise such as Poisson noise, many data fidelity functions have been proposed (see, e.g., \cite{byrne1993,chouzenoux2015,banerjee2024}). 

The second term $\Psi \circ \opL$ in \eqref{eq:convex-opt} is a regularizer for promoting prior knowledge on the target signal $x^\star$. 
In modern signal processing applications, many regularizers $\Psi\circ\opL$ have been designed based on a certain sparsity prior on $\xstar$. A standard convex regularizer for sparsity is the $\ell_1$ norm $\Psi\coloneqq \norm{\cdot}_1$ which is the convex envelope of $\ell_0$ pseudo-norm $\norm{\cdot}_0$ in the vicinity of zero vector.
The convexity of $\Psi\coloneqq\norm{\cdot}_1$ ensures the convexity of $J_{\Psi\circ\opL}$, which is a key for finding a global minimizer. 
The $\ell_1$ norm has been utilized widely, e.g., in the Lasso model \cite{tibshirani1996} and the total variation (TV) model \cite{rudin1992}.
However, the gap between $\norm{\cdot}_0$ and $\norm{\cdot}_1$ causes some underestimation effect which degrades the estimation accuracy \cite{candes2008}.

To suppress such an underestimation effect caused by $\norm{\cdot}_1$, nonconvex regularizers such as $\ell_p$-quasinorm ($0<p<1$) have been utilized as a sparsity promoting regularizer~$\Psi$, e.g., in \cite{chartrand2007}.
However, the severe nonconvexity of $\ell_p$-quasinorm makes the cost function nonconvex, and thus makes it difficult to design algorithms of guaranteed convergence to a global minimizer.
Indeed, even in the standard case of $f\coloneqq \frac{1}{2}\norm{y-\cdot}^2_{\spY}$, existing algorithms typically have a risk to be trapped into a local minimizer.

As a remedy of this dilemma, in a case of $f\coloneqq \frac{1}{2}\norm{y-\cdot}_{\spY}^2$,
the \textit{constrained LiGME (cLiGME) model}:
\begin{equation}
  \label{eq:cLiGME}
  \minimize_{\opC x\in\Cz} \frac{1}{2}\norm{y - Ax}^2_{\spY} + \mu \Psi_B \circ\opL (x)
\end{equation}
has been proposed \cite{cLiGME}, where
\begin{equation}
  \label{eq:GME}
  \Psi_B(\cdot) \coloneqq \Psi(\cdot) - \min_{v\in\spZ}\left[
    \Psi(v) + \frac{1}{2}\norm{B(\cdot - v)}^2_{\sptildeZ}
  \right]
\end{equation} 
is\footnote{
  The existence of a minimizer in the definition of $\Psi_B$ in \eqref{eq:GME} is guaranteed by the coercivity of $\Psi$ and the lower boundedness of $\frac{1}{2}\norm{B\cdot}^2_{\sptildeZ}$ \cite{LiGME} (see also Fact \ref{fact:existence-minimizer}).
} the \textit{generalized Moreau enhanced (GME) function} of convex $\Psi$ with a \textit{GME matrix} $B\in\setLO{\spZ}{\sptildeZ}$.
The GME regularizer $\Psi_B$ was introduced originally in the \textit{Linearly involved Generalized Moreau Enhanced (LiGME) model} \cite{LiGME} as an extension of the \textit{generalized minimax concave (GMC) penalty} \cite{selesnick2017}.
Indeed, the model \eqref{eq:cLiGME} reproduces the LiGME model \cite{LiGME} by setting $(\opC,\Cz)\coloneqq (\Id,\spX)$ and the GMC model by 
$(\Psi,\opL,\opC,\Cz) \coloneqq (\norm{\cdot}_1, \Id,\Id,\spX)$.
A remarkable feature of GMC and GME regularizers is that we can balance the nonconvexity of $\Psi_B$ and the convexity of the cost function in \eqref{eq:cLiGME} with a strategical tuning of $B$ \cite{Chen2023}.
Thanks to the convexity of the cost function, iterative algorithms for LiGME and cLiGME models have been established in \cite[Thoerem 1]{LiGME} and \cite[Theorem 3.1]{cLiGME} with guaranteed convergence to a global minimizer under its existence. 
These algorithms are designed based on proximal splitting techniques without using inner loops.
Applications of GMC, LiGME and cLiGME models have been found, e.g., in transient artifact suppression \cite{feng2020},
bearing fault diagnosis \cite{hou2021}, 
damage identification \cite{sun2022}, 
magnetic resonance imaging \cite{kitahara2021}, grouped variable selection \cite{liu2023}, discrete-valued signal estimation
\cite{shoji2025} and nonconvex enhancement of multi-layered regularizer \cite{katsuma2025}.
We refer readers to \cite{shabili2021,kuroda2024,heng2025} for further developments of GMC and GME strategies using inner-loop free algorithms in a case of\footnote{
  An extension of the data fidelity function to Bregman divergence is found in \cite[Section IV.A]{shabili2021}, where an iterative algorithm was given only in a case of $f = \frac{1}{2}\norm{y-\cdot}^2_{\spY}$.
} $f\coloneqq \frac{1}{2}\norm{y-\cdot}^2_{\spY}$.
The reference \cite{zhang2023} extended the data fidelity function in the LiGME model to a general convex function, where a DC type iterative algorithm, involving inner loops, has been utilized (see Remark~\ref{remark:research-question}).
Another research direction from GMC \cite{selesnick2017} is found in \cite{zhang2025} where the least angle regression (LARS) algorithm \cite{efron2004,tibishirani2013} established originally for the Lasso model has been extended for the GMC model with a specially designed GME matrix $B$.

In this paper, we address the following research questions:
\begin{enumerate}[(Q1)]
  \item Can we extend the cLiGME model \eqref{eq:cLiGME} and its inner-loop free algorithm for applications to non-quadratic data fidelity cases according to non-Gaussian noise~$\varepsilon$~?
  \item What conditions are sufficient for an extended model to have a minimizer~?
\end{enumerate}

\begin{remark}[Related works on research questions]
  \label{remark:research-question}
  \quad
  \begin{enumerate}[(a)]
    \item Extension of the data fidelity function allows us to utilize non-quadratic data fidelity functions for non-Gaussian noise cases.
    In \cite{zhang2023}, 
    for utilization of a non-quadratic data fidelity function, LiGME model was extended as an instance of \textit{difference-of-convex (DC) optimization problems} \cite{tao1997,thi2018} under an \textit{overall convexity condition} of the cost function.
    \cite{zhang2023} also proposed a DC type iterative algorithm
    based on the so-called "simplified DCA"
    \cite{tao1997}
    with use of inner loops, where every cluster point of the sequence generated by the algorithm in \cite{zhang2023} is guaranteed to be a global minimizer.
    On the other hand, in a case of $f\coloneqq \frac{1}{2}\norm{y - \cdot}^2_{\spY}$, \cite{LiGME,cLiGME} proposed inner-loop free algorithms based on proximal splitting techniques with guaranteed convergence of the whole sequence to a global minimizer.

    \item 
    The existence of a minimizer is certainly a key for successful applications of optimization models because we hope to use a minimizer as an estimate of the target signal. 
    Indeed, in \cite{LiGME,cLiGME,zhang2023}, the existence of a minimizer is assumed implicitly \cite{LiGME,cLiGME} or explicitly \cite{zhang2023}.
    So far, the existence of a minimizer for the cLiGME model is guaranteed based on standard conditions such as the nonsingularity of $A^*A$ or the compactness of $\opC^{-1}(\Cz)$ \cite[Footnote 3]{Chen2023}.
    In this paper, by exploiting more special properties of the cost function and the constraint in \eqref{eq:cLiGME}, we derive widely applicable sufficient conditions to ensure the existence of a minimizer of the model \eqref{eq:cLiGME}.
  \end{enumerate}
\end{remark}
In this paper, we consider the following model as a nonconvex enhancement of the model \eqref{eq:convex-opt}.
\begin{problem}[constrained LiGME model for non-quadratic Smooth Convex data Fidelities (cLiGME-SCF)]
  \label{prob:cLiGME-w-general-fidelity}
  In the settings of Problem \ref{prob:convex-opt}, consider the following optimization model:
  \begin{equation}
    \label{eq:cLiGME-w-general-fidelity}
    \minimize_{\substack{\opC x \in\Cz}} J_{\Psi_B\circ\opL}(x)\coloneqq f\circ A(x) + \mu \Psi_B\circ\opL (x)
  \end{equation}
  with the GME function $\Psi_B$ of $\Psi$ in \eqref{eq:GME}
  under
  \begin{equation}
    \label{eq:assumption-proper}
    \dom (\Psi \circ \opL) \cap \opC^{-1}(\Cz)\neq\emptyset
  \end{equation}
  (Note: $\dom (J_{\Psi_B\circ \opL}) = \dom(\Psi\circ \opL)$ will be shown in Proposition \ref{prop:overall-convexity}(b)).
\end{problem}

\begin{remark}[The model \eqref{eq:cLiGME-w-general-fidelity} is a generalization of the models \eqref{eq:convex-opt} and \eqref{eq:cLiGME}]
  The model \eqref{eq:cLiGME-w-general-fidelity} reproduces the naive convex model \eqref{eq:convex-opt} by letting $B=\zeroMatrix_{\spZ,\sptildeZ}\in\setLO{\spZ}{\sptildeZ}$. 
  By letting $f\coloneqq \frac{1}{2}\norm{y-\cdot}^2_{\spY}$, the model~\eqref{eq:cLiGME-w-general-fidelity} reproduces the cLiGME model~\eqref{eq:cLiGME}.
\end{remark}

Regarding the research questions (Q1) and (Q2), we present in 
Theorem \ref{thm:existence-solution} sufficient conditions for the existence of a minimizer of the model \eqref{eq:cLiGME-w-general-fidelity} and in Theorem \ref{thm:convergence-analysis} and Algorithm \ref{alg:proposed} an inner-loop free algorithm for the model~\eqref{eq:cLiGME-w-general-fidelity}.

The remainder of this paper is organized as follows\footnote{A preliminary version of this paper will be presented in \cite{yata2025}.}: after introducing notation and necessary notions of convex analysis, monotone operator theory and fixed point theory in Section \ref{sec:Preliminaries} (see also Appendix \ref{appendix:known-facts} for helpful facts for mathematical discussions in this paper), 
we present in Section \ref{sec:convexity-cond} a sufficient condition, called \textit{the overall convexity condition}, for the convexity of $J_{\Psi_B\circ\opL}$ (see Proposition \ref{prop:overall-convexity}).
Under the overall convexity condition, we present sufficient conditions for the existence of a minimizer of the model \eqref{eq:cLiGME-w-general-fidelity} (see Theorem \ref{thm:existence-solution}).
Even in a case where $f= \frac{1}{2}\norm{y - A\cdot}_{\spY}^2$ (in this case, the model~\eqref{eq:cLiGME-w-general-fidelity} reproduces the cLiGME model~\eqref{eq:cLiGME}), the proposed sufficient conditions for the existence of a minimizer are novel and reproduce the previously known conditions \cite[Footnote 3]{Chen2023}.
Under the overall convexity condition and the existence of a minimizer, in Section \ref{sec:algorithm}, we present an inner-loop free algorithm that approximates iteratively a global minimizer of the proposed model (see Theorem~\ref{thm:convergence-analysis}).
In Section \ref{sec:fidelity}, we discuss favorable properties of the observation loss function for the proposed model.
Even in a case where the observation loss function enjoys such conditions only over the constraint set, we propose a reformulation of the observation loss function for achieving favorable conditions over the entire space (see Proposition \ref{prop:separable-extension} and Corollary \ref{corollary:existence-minimizer-extension}).
In Section \ref{sec:numerical}, we demonstrate the effectiveness of the proposed methods by numerical experiments in scenarios of Poisson denoising problem (see Section~\ref{sec:poisson-denosing}) and simultaneous denoising and declipping problem (see Section~\ref{sec:declipping}).

\section{Preliminaries}
\label{sec:Preliminaries}
\subsection{Notation}
\label{sec:notation}
{
Symbols $\setN_0$, $\setR$, $\setNNR$ and $\setPR$ denote respectively all nonnegative integers, all real numbers, all nonnegative real numbers and all positive real numbers.
Let $\spH$ and $\spK$ be finite-dimensional real Hilbert spaces.
A Hilbert space $\spH$ is equipped with an inner product $\ip{\cdot}{\cdot}_\spH$ and its induced norm $\norm{\cdot}_\spH\coloneqq \sqrt{\ip{\cdot}{\cdot}_{\spH}}$.
The zero element of $\spH$ is  $0_{\spH}$.
For Hilbert spaces $\spH$ and $\spK$, $\spH\times \spK$ defines a new Hilbert space equipped with 
{
\thickmuskip=0.5\thickmuskip
\medmuskip=0.5\medmuskip
\thinmuskip=0.5\thinmuskip
\arraycolsep=0.5\arraycolsep
\begin{align*}
  \mbox{vector addition: }
  &(\spH\times\spK)\times(\spH\times\spK)\to(\spH\times\spK):((x_1,y_1),(x_2,y_2))\mapsto (x_1+x_2, y_1+y_2),\\
  \mbox{scalar multiplication: } 
  &\setR\times (\spH\times\spK)\to (\spH\times\spK) : (\alpha, (x,y))\mapsto (\alpha x, \alpha y),\\
  \mbox{inner product: } &\ip{\cdot}{\cdot}_{\spH\times \spK}: (\spH\times\spK) \times (\spH\times\spK)\to\setR:((x_1,y_1),(x_2,y_2))\mapsto\ip{x_1}{x_2}_\spH + \ip{y_1}{y_2}_\spK,\\
  \mbox{induced norm: } &\norm{\cdot}_{\spH\times \spK}: (\spH\times \spK)\to\setR:(x,y)\mapsto\sqrt{\ip{(x,y)}{(x,y)}_{\spH\times \spK}}.
\end{align*}
}%

$\setLO{\spH}{\spK}$ denotes the set of all linear operators from $\spH$ to $\spK$.
For a linear operator $L\in \setLO{\spH}{\spK}$, $\norm{L}_{\mathrm{op}}$ denotes the operator norm of $L$ (i.e., $\norm{L}_{\mathrm{op}} \coloneqq\sup_{x\in\spH, \norm{x}_{\spH}\leq 1} \norm{Lx}_\spK$),
$L^*\in\setLO{\spK}{\spH}$ denotes the adjoint operator of $L$ 
(i.e., $(\forall x\in \spH)(\forall y \in\spK)\, 
\ip{Lx}{y}_\spK = \ip{x}{L^*y}_\spH$), and $L^\dagger\in\setLO{\spK}{\spH}$ denotes the Moore-Penrose pseudo inverse of $L$ (see, e.g., \cite[Definition 3.28]{CAaMOTiH}).
The identity operator is denoted by
$\Id$ and the zero operator from $\spH$ to $\spK$ by $\zeroMatrix_{\spH,\spK}$. In particular, we use $\zeroMatrix_{\spH}$ for the zero operator from $\spH$ to $\spH$.
We express the positive definiteness and the positive semidefiniteness of
a self-adjoint operator $L\in \setLO{\spH}{\spH}$ 
as $L\succ \zeroMatrix_\spH$ and $L\succeq  \zeroMatrix_\spH$, respectively. We also use $L\succneqq \zeroMatrix_{\spH}$ to express the positive semidefiniteness of nonzero linear operator $L$.
Any $L\succ\zeroMatrix_{\spH}$ can be used to define a new Hilbert space $(\spH,\ip{\cdot}{\cdot}_{L},\norm{\cdot}_L)$ with the inner product $\ip{\cdot}{\cdot}_{L}\coloneqq \ip{L \cdot}{\cdot}_{\spH}$ and its induced norm $\norm{\cdot}_L\coloneqq \sqrt{\ip{\cdot}{\cdot}_{L}}$. 

For a set $S\subset\spH$,
$\operatorname{aff}(S)\coloneqq\{ \sum_{i=1}^k \lambda_i x_i\in\spH \mid k\in\setN, (\forall i \in\{1,2,\ldots,k\})\ \lambda_i\in\setR,x_i\in S, \sum_{i=1}^k\lambda_i=1\}$ is the affine hull of $S$, and $\ri (S)\coloneqq \{x\in\spH\mid (\exists \alpha\in\setPR) \ {B}(x;\alpha)\cap \operatorname{aff}(S)\subset S \}$ is the relative interior of $S$, where ${B}(x;\alpha)\coloneqq\{y\in\spH\mid \norm{x-y}_{\spH}<\alpha\}$ is an open ball of radius $\alpha>0$ centered at $x\in\spH$.
For a linear operator $L\in\setLO{\spH}{\spK}$ and sets $S_{\spH}\subset\spH$ and $S_{\spK}\subset\spK$, $L(S_{\spH})\coloneqq \{Lx \in\spK\mid x\in S_{\spH}\}$ is the image of $S_{\spH}$ under $L$, and $L^{-1}(S_{\spK})\coloneqq \{x\in\spH \mid Lx\in S_{\spK}\}$ is the preimage of $S_{\spK}$ under $L$.
In particular, $\nullsp L$ denotes the null space of $L$ (i.e., $\nullsp L\coloneqq L^{-1}(\{0_{\spK}\})$) and $\ran L$ denotes the range space of $L$ (i.e., $\ran L\coloneqq L(\spH)$).

For a vector $x\in\setR^n$, the $i$-th component of $x$ is denoted by $[x]_i\in\setR$.
Likewise, for a matrix $A \in \setR^{m\times n}$, $(i,j)$-th entry of $A$ is denoted by $[A]_{i,j}\in\setR$.
}
\subsection{Selected elements in convex analysis, monotone operator theory and fixed point theory}
\label{sec:Preliminaries-convex}
A set $C\subset\spH$ is said to be convex if $(\forall x,y\in C, \alpha\in (0,1))\ \alpha x + (1-\alpha) y \in C$.
A function $\varphi:\spH\to (-\infty,\infty]$ is said to be
(a) proper if $\dom \varphi \coloneqq \{x\in\spH \mid \varphi(x)<\infty\}\neq \emptyset$, 
(b) lower semicontinuous
if $ \operatorname{lev}_{\leq\alpha} \varphi \coloneqq \{ x\in \spH | \varphi(x)\leq \alpha\}$ 
is closed for every $\alpha\in \setR$,
(c) convex if 
$\varphi(\alpha x + (1-\alpha) y)\leq \alpha \varphi(x)+(1-\alpha) \varphi(y) $
for every $x,y \in \dom \varphi$ and $\alpha\in(0,1)$.
The set of all proper lower semicontinuous convex functions defined on $\spH$ is denoted by $\Gamma_0 (\spH)$.

We use the following functions/operators associated with $\varphi\in\Gamma_0(\spH)$.

\noindent
\textbf{(Legendre-Fenchel conjugate)} 
The \textit{Legendre-Fenchel conjugate} of $\varphi\in\Gamma_0(\spH)$ is defined as 
\begin{equation}
  \varphi^*:\spH\to(-\infty, \infty]: y \mapsto\sup_{x\in\spH}\left( \ip{x}{y}_{\spH} - \varphi(x)\right).
\end{equation}
For any $\varphi\in\Gamma_0(\spH)$, $\varphi^*\in\Gamma_0(\spH)$ is guaranteed.

\noindent \textbf{(Subdifferential)}
For a function $\varphi\in\Gamma_0(\spH)$, the \textit{subdifferential} of $\varphi$ is defined as
\begin{equation}
  \partial \varphi:\spH\to 2^{\spH}:x\mapsto \{u\in\spH\mid  (\forall y\in\spH)\ \varphi(x)+ \ip{y-x}{u}_{\spH}\leq \varphi(y)\}.
\end{equation}

\noindent \textbf{(Proximity operator)}
The \textit{proximity operator} $\prox_{\varphi}:\spH\to\spH$ of $\varphi\in\Gamma_0(\spH)$ is defined as 
\begin{equation}
  \prox_{\varphi}:\spH\to\spH:x\mapsto \argmin_{v\in\spH}\left[
    \varphi(v) + \frac{1}{2}\norm{v-x}^2_\spH
  \right].
\end{equation}
$\varphi\in\Gamma_0(\spH)$ is said to be \textit{prox-friendly} if
$\prox_{\gamma\varphi}$ is available as a computable operator for every $\gamma>0$. For $\varphi\in\Gamma_0(\spH)$, $\prox_{\varphi^*} = \Id - \prox_{\varphi}$ holds \cite[Proposition 24.8(ix)]{CAaMOTiH}.
For a nonempty closed convex set $C\subset \spH$, the proximity operator $\prox_{\iota_C}$ of the \textit{indicator function} of $C$:
\begin{equation}
  \iota_C:\spH\to(-\infty,\infty]:x\mapsto \begin{cases}
    0 & x\in C\\
    \infty & x\notin C
  \end{cases}
\end{equation} 
coincides with the \textit{metric projection} $P_C$ onto $C$:
\begin{equation}
  P_C:\spH\to\spH:x\mapsto \argmin_{y\in C}\norm{x-y}_{\spH}.
\end{equation}
A closed convex set $C$ is said to be \textit{simple} if $P_C$ is available as a computable operator.

\noindent \textbf{(Recession function)}
For $\varphi\in\Gamma_0(\spH)$, its \textit{recession function} $\operatorname{rec}(\varphi)\in\Gamma_0(\spH)$ is defined as
\begin{equation}
  \label{eq:properties-recession}
  \operatorname{rec}(\varphi):\spH\to (-\infty,\infty]:x \mapsto \sup_{v\in\dom \varphi} (\varphi(v+x) - \varphi(v)) = \lim_{\alpha\to+\infty} \frac{\varphi(u + \alpha x)}{\alpha},
\end{equation}
where the last equation holds independently of the choice $u\in\dom\varphi$ \cite[Proposition 9.30(iii)]{CAaMOTiH}. From the definition, $\operatorname{rec}(\varphi)(0_{\spH}) =0$ holds.
For $\psi\in\Gamma_0(\spK)$ and $L\in\setLO{\spH}{\spK}$ satisfying $\dom \psi\cap \ran L\neq \emptyset$, $\operatorname{rec} (\psi\circ L) = \operatorname{rec}(\psi)\circ L$ holds \cite[Proposition 9.30(vii)]{CAaMOTiH}.

By using the recession function, we introduce the concept of coercivities that have been exploited for ensuring the existence of a minimizer of optimization models (see Fact \ref{fact:existence-minimizer} below).

\noindent
\textbf{(Coercivity)} 
A function $\varphi\in\Gamma_0(\spH)$ is said to be \textit{coercive} \cite[Definition 2.1]{Auslender1996} if 
\begin{equation}
  \label{eq:def-coercive}
  (\forall x \in\spH\setminus \{0_{\spH}\})\quad \operatorname{rec} (\varphi) (x) > 0.
\end{equation}
The condition \eqref{eq:def-coercive} is equivalent\footnote{From \cite[Theorem 3.26(a) and Corollary 3.27]{rockafellar2009variational}, $\varphi\in\Gamma_0(\spH)$ satisfies \eqref{eq:def-coercive} if and only if $\operatorname{lev}_{\leq\alpha}\varphi$ is bounded for every $\alpha\in\setR$.
Furthermore, \cite[Proposition 11.12]{CAaMOTiH} verifies that $\operatorname{lev}_{\leq\alpha}\varphi$ is bounded for every $\alpha\in \setR$ if and only if $\varphi$ satisfies \eqref{eq:def2-coercive}.} to the condition
\begin{equation}
  \label{eq:def2-coercive}
  \lim_{\norm{x}_{\spH}\to +\infty} \varphi(x) = \infty
\end{equation}
used in \cite{CAaMOTiH,LiGME,cLiGME} for the definition of the coercivity.

\noindent
\textbf{(Weak coercivity)}
A function $\varphi\in\Gamma_0(\spH)$ is said to be \textit{weakly coercive} \cite[Definition~2.2]{Auslender1996} \cite[Definition 3.2.1]{auslender2003} if $\varphi$ satisfies the following conditions
\begin{equation}
  \label{eq:recession-nonnegative}
 (\forall x \in\spH)\quad \operatorname{rec} (\varphi) (x) \geq 0
\end{equation}
and 
\begin{equation}
  \label{eq:constant-on-recession-direction}
  \operatorname{rec}  (\varphi) (x) = 0 
  \implies
  (\forall \alpha \in\setR)
  (\forall u\in\dom\varphi)\quad
  \varphi(u+\alpha x) = \varphi(u).
\end{equation}

\begin{fact}
  \label{fact:existence-minimizer}
  For $\varphi\in\Gamma_0(\spH)$ and a closed convex set $C\subset\spH$ satisfying $\dom \varphi \cap C \neq \emptyset$,
  consider the following minimization problem:
  \begin{equation}
    \label{eq:fact-opt-prob}
    \minimize_{x \in C} \varphi(x).
  \end{equation}
  Then the following hold.
  \begin{enumerate}[(a)]
    \item (\cite[Proposition 11.15]{CAaMOTiH}). Assume that $\varphi$ is coercive or $C$ is bounded. Then the minimization problem~\eqref{eq:fact-opt-prob}
    has a minimizer in $C$.
    \item (\cite[Theorem 2.4]{Auslender1996}). Assume that $\varphi$ is weakly coercive and $C$ is \textit{asymptotically multipolyhedral} \cite[Definition 2.3]{Auslender1996}, i.e., $C$ can be decomposed as $C=S+K$ with a compact set $S\subset\spH$ and a polyhedral cone\footnote{
      A set $K$ is a cone if $(\forall x \in K)(\forall \alpha\in\setPR)\ \alpha x \in K$.
      A set $K$ is polyhedral if it is a finite intersection of closed half-spaces.
    } $K\subset\spH$.
    Then the minimization problem \eqref{eq:fact-opt-prob} has a minimizer in $C$.
  \end{enumerate}
\end{fact}

We also use the following notions and properties of operators.

\noindent
\textbf{(Nonexpansive operator)}
An operator $T:\spH\to\spH$ is said to be \textit{Lipschitz continuous} with a Lipschitz constant $\lipconst{T}\in\setNNR$ if 
\begin{equation}
   (\forall (x,y)\in\spH\times \spH)\ 
  \norm{Tx-Ty}_{\spH} \leq \lipconst{T}\norm{x-y}_{\spH}.
\end{equation}
An operator $T:\spH\to\spH$ is said to be 
\textit{nonexpansive} if $T$ is $1$-Lipschitz continuous.
In particular, $T$ is \textit{$\alpha$-averaged nonexpansive} with $\alpha\in (0,1)$ if there exists a nonexpansive operator $R:\spH\to\spH$ such that $T= (1-\alpha)\Id + \alpha R$.
A fixed point of a nonexpansive operator can be approximated iteratively via the Krasnosel'ski\u{\i}-Mann iteration.
\begin{fact}[Krasnosel'ski\u{\i}-Mann iteration (See, e.g., {\cite{groetsch1972}})]
  \label{fact:KM-itr}
  Let $T:\spH\to\spH$ be a nonexpansive operator such that $\fix(T) \coloneqq \{x \in \spH\mid x = T(x)\}\neq \emptyset$, where $\fix(T)\subset\spH$ is called the \textit{fixed point set of $T$} and known to be a closed convex set.
    For any initial point $x_0\in\spH$,
    set $(x_k)_{k\in\setN_0}$ by 
    \begin{equation}
      (\forall k\in \setN_0)\  x_{k+1} = [(1-\lambda_k)\Id + \lambda_k T](x_k)
    \end{equation}
    with a sequence $(\lambda_k)_{k\in\setN_0} \subset [0,1]$ 
    satisfying
    $\sum_{k\in\setN_0}\lambda_k(1- \lambda_k)=+\infty$.
    Then $(x_k)_{k\in\setN_0}$ converges to a point in $\fix(T)$.
    In particular, if $T= (1-\alpha)\Id + \alpha R$ is $\alpha$-averaged nonexpansive with $\alpha\in(0,1)$ and a nonexpansive operator $R:\spH\to\spH$,
    the sequence generated by 
    $
      (\forall n \in \setN_0)\ x_{k+1} =T(x_k)
    $
    converges to a point in $\fix(T)=\fix(R)$.
\end{fact}

\noindent
\textbf{(Cocoercive operator)}
An operator $T:\spH\to\spH$ is said to be \textit{$\alpha$-cocoercive} with $\alpha\in\setPR$ if
\begin{equation}
  (\forall (x, y)\in\spH\times\spH)\ \ip{x-y}{Tx - Ty}_{\spH}\geq \alpha \norm{Tx-Ty}^2_{\spH}.
\end{equation} 

\noindent
\textbf{(Monotone operator)} A set-valued operator $A:\spH\to 2^{\spH}$ is \textit{monotone} if
\begin{equation}
  (\forall (x,u)\in \gra(A))(\forall (y,v) \in \gra(A))\ 
  \ip{x-y}{u-v}_{\spH}\geq 0,
\end{equation}
where $\gra(A) \coloneqq \left\{
  (x,u)\in \spH\times\spH |u\in A(x)
\right\}\subset \spH\times\spH$. 
In particular,
$A$ is \textit{maximally monotone} if
\begin{equation}
  (x,u)\in \gra(A) \iff (\forall (y, v)\in \gra(A))\
  \ip{x-y}{u-v}_{\spH}\geq 0
\end{equation} 
holds for every $(x,u)\in \spH\times\spH$.
For $\varphi\in\Gamma_0(\spH)$, its subdifferential $\partial \varphi:\spH\to2^{\spH}$ is maximally monotone.
$A:\spH\to 2^\spH$ is maximally monotone if and only if 
the \textit{resolvent} $(\Id +A)^{-1}:$ $u\in\spH\mapsto 
\left\{ x\in \spH \middle| u\in x+A(x)\right\}$ is single-valued and
$\frac{1}{2}$-averaged nonexpansive \cite[Corollary 23.9]{CAaMOTiH}.

\subsection{Relative strong convexity and relative weak convexity}
The \textit{relative strong/weak convexity} has been utilized, e.g., in \cite{haihao2018,davis2018}, as a generalization of the classical strong/weak convexity.
As will be shown in Proposition \ref{prop:overall-convexity}, mechanism to achieve the convexity of the proposed model \eqref{eq:cLiGME-w-general-fidelity} can be interpreted through {the strong/weak convexity relative to a quadratic function $q_M$}
(Note: The strong convexity relative to a reference function other than $q_M$ is also considered in \cite{haihao2018,davis2018}).

\begin{definition}[Strong convexity and weak convexity relative to $q_M$]
  \label{def:strong-convexity}
  Let
  $\varphi:\spH\to(-\infty,\infty]$ be proper and lower semicontinuous, $C\subset \spH$ be a closed convex set such that $\dom \varphi \cap C \neq \emptyset$. 
  Define $q_M$ with $M\succeq \zeroMatrix_{\spH}$ by
  \begin{equation}
    q_M:\spH\to\setR: x\mapsto \frac{1}{2}\ip{Mx}{x}_{\spH}.
  \end{equation}
  \begin{enumerate}[(a)]
    \item With $\alpha>0$, $\varphi$ is said to be \textit{$\alpha$-strongly convex relative to the reference function $q_M$ on $C$} if 
      $\varphi - \alpha q_M
    $ is convex over $C$. In particular, if $C\coloneqq \spH$, we say $\varphi$ is $\alpha$-strongly convex relative to $q_M$. (Note: Every convex function is $\alpha$-strongly convex relative to $q_{\zeroMatrix_{\spH}}$.)
    \item With $\alpha>0$, $\varphi$ is said to be \textit{$\alpha$-weakly convex relative to the reference function $q_M$ on $C$} if 
      $\varphi + \alpha q_M
    $ is convex over $C$. In particular, if $C\coloneqq \spH$, we say $\varphi$ is $\alpha$-weakly convex relative to $q_M$.
  \end{enumerate}
  In a case of $q_M = q_{\Id}$, the relative strong/weak convexity reduces to the classical strong/weak convexity.
\end{definition}

We present basic properties of the strong convexity and the weak convexity relative to $q_M$.
\begin{proposition}[Properties of strongly/weakly convex function relative to $q_M$]
  \label{prop:properties-strong-convexity}
  \quad
  \begin{enumerate}[\upshape (a)]
    \item Let $\varphi_1:\spH\to(-\infty,\infty]$
    and $\varphi_2:\spH\to(-\infty,\infty]$ be proper lower semicontinuous. 
    Assume that $\varphi_1$ is $\alpha_1$-strongly convex relative to $q_{M_1}$ and $\varphi_2$ is $\alpha_2$-weakly convex relative to $q_{M_2}$.
    Then $\varphi_1+\varphi_2\in\Gamma_0(\spH)$ if $\alpha_1 M_1 - \alpha_2M_2\succeq \zeroMatrix_{\spH}$ and $\dom(\varphi_1 +\varphi_2)\neq \emptyset$. 

    \item Let $\psi\in\Gamma_0(\spK)$ be $\alpha$-strongly convex relative to $q_M$ and $L\in\setLO{\spH}{\spK}$ satisfy $\dom  \psi\cap\ran L\neq\emptyset$. Then $\psi\circ L$ is $\alpha$-strongly convex relative to $q_{L^*ML}$.
  \end{enumerate}
\end{proposition}
\begin{proof}
\noindent
\textbf{(Proof of (a))}
We have the expression
$
  \varphi_1 + \varphi_2
  = \left(\varphi_1 - \alpha_1q_{M_1}\right) + \left(\varphi_2 + \alpha_2q_{M_2}\right)
  + q_{(\alpha_1M_1- \alpha_2M_2)}.
$
From $\alpha_1M_1-\alpha_2M_2\succeq \zeroMatrix_{\spH}$, we observe that $\varphi_1 - \alpha_1q_{M_1}$, $\varphi_2 + \alpha_2q_{M_2}$ and $q_{(\alpha_1M_1- \alpha_2M_2)}$ in RHS of the last equation belong to $\Gamma_0(\spH)$. 
Hence, $\varphi_1+ \varphi_2\in\Gamma_0(\spH)$ is verified by \cite[Corollary 9.4]{CAaMOTiH} and $\dom (\varphi_1+ \varphi_2) \neq \emptyset$.

\noindent
{
\textbf{(Proof of (b))}
We observe 
$
  \psi \circ L - \alpha q_{L^*ML}= \psi \circ L - \alpha q_{M}\circ L
  = \left(\psi - \alpha q_{M}\right) \circ L.
$
Then $\left(\psi - \alpha q_{M} \right) \circ L\in\Gamma_0(\spH)$ is verified by \cite[Proposition
9.5]{CAaMOTiH} with $\psi -  \alpha q_{M}\in\Gamma_0(\spK)$, where $\dom\left(\psi - \alpha q_{M}\right)\cap \ran L\neq \emptyset$ follows from $\dom\left(\psi -  \alpha q_{M}\right)\cap \ran L = \dom \psi \cap\ran L$.
}

\end{proof}

\section{Existence of minimizer of cLiGME-SCF and its proximal splitting type algorithm}
Section \ref{sec:convexity-cond} presents a sufficient condition for the convexity of the cost function of Problem \ref{prob:cLiGME-w-general-fidelity} and sufficient conditions for the existence of a minimizer of Problem \ref{prob:cLiGME-w-general-fidelity}. 
Section~\ref{sec:algorithm} presents a proximal splitting type algorithm for Problem \ref{prob:cLiGME-w-general-fidelity}. 

\subsection{Overall convexity condition and existence of minimizer}
\label{sec:convexity-cond}
In this section\footnote{
  We remark that all statements in Proposition \ref{prop:overall-convexity}, Proposition \ref{prop:coercivity-J} and Theorem \ref{thm:existence-solution} in this section hold under $\dom f = \spY$, even in the absence of the differentiability of $f$ over $\spY$ which is assumed in Problem \ref{prob:convex-opt} and Problem \ref{prob:cLiGME-w-general-fidelity}.
}, we present a sufficient condition for the overall convexity of the cost function $J_{\Psi_B\circ \opL}$ and sufficient conditions for the existence of a minimizer of \eqref{eq:cLiGME-w-general-fidelity}.
The following proposition presents basic properties of the LiGME regularizer $\Psi_B\circ \opL$ and a sufficient condition for the convexity of $J_{\Psi_B\circ\opL}$.
\begin{proposition}[Overall convexity condition for $J_{\Psi_B\circ\opL}$ in Problem \ref{prob:cLiGME-w-general-fidelity}]
  \label{prop:overall-convexity}
  In Problem \ref{prob:cLiGME-w-general-fidelity}, the following hold.
  \begin{enumerate}[\upshape (a)]
    \item The LiGME regularizer $\Psi_B\circ \opL$ can be decomposed as 
    \begin{equation}
      \label{eq:LiGME-reg-decompose}
      \Psi_B\circ \opL = \Psi\circ \opL -
      \frac{1}{2} \norm{ B\opL \cdot}^2_{\sptildeZ} + \left(\Psi + \frac{1}{2}\norm{B \cdot}^2_{\sptildeZ}\right)^*\circ B^*B\opL,
    \end{equation}
    where  $\left(\Psi + \frac{1}{2}\norm{B \cdot}^2_{\sptildeZ}\right)^*\circ B^*B\opL \in \Gamma_0(\spX)$ enjoys
     \begin{equation}
      \label{eq:dom-of-conjugate}
      \dom \left[ \left(\Psi + \frac{1}{2}\norm{B \cdot}^2_{\sptildeZ}\right)^*\circ B^*B\opL\right]= \spX.
    \end{equation}
    Moreover, $\Psi_B\circ \opL$ is proper lower semicontinuous 
    and 1-weakly convex relative to $q_{\opL^*B^*B\opL}$
    with
    \begin{align}
      \label{eq:dom-LiGME-reg}
      \dom (\Psi_B\circ \opL) &\overset{\eqref{eq:LiGME-reg-decompose}}{=}
      \dom(\Psi\circ\opL )\cap \dom\left( -
      \frac{1}{2} \norm{ B\opL \cdot}^2_{\sptildeZ}\right) \cap \dom \left[\left(\Psi + \frac{1}{2}\norm{B \cdot}^2_{\sptildeZ}\right)^*\circ B^*B\opL\right]\\
      &\overset{\eqref{eq:dom-of-conjugate}}{=}\dom (\Psi\circ \opL) \cap \spX\cap \spX
      = \dom (\Psi\circ \opL) \overset{\eqref{eq:assumption-proper}}{\neq}\emptyset.
    \end{align}
    \item $J_{\Psi_B\circ\opL}$ is proper lower semicontinuous with 
    \begin{equation}
      \label{eq:dom-of-J}
      \dom (J_{\Psi_B\circ\opL}) =\dom(\Psi\circ \opL) \overset{\eqref{eq:assumption-proper}}{\neq}\emptyset.
    \end{equation}
    \item Let\footnote{
    At least, we can choose $\Lambda = \zeroMatrix_{\spY}$ from the convexity of $f$. However, as mentioned in Remark \ref{remark:oc-cond}(c) below, we recommend choosing $\Lambda \succneqq \zeroMatrix_{\spY}$.
  } $f$ be 1-strongly convex relative to $q_\Lambda$  with $\Lambda \succeq\zeroMatrix_{\spY}$.
    The following relation holds:
    \begin{align}
      \label{eq:convexity-cond-C0}
      &(C_0): 
      A^*\Lambda A-\mu\opL^*B^*B\opL\succeq \zeroMatrix_{\spX}\\
      \label{eq:convexity-cond}
      \implies &(C_1): f\circ A \mbox{ is $\mu$-strongly convex relative to $q_{\opL^*B^*B\opL}$, or equivalently }\\
      \label{eq:def-d}
      &\quad \quad \quad\quad \mathfrak{d}\coloneqq f\circ A - \frac{\mu}{2}\norm{B\opL \cdot}_{\sptildeZ}^2\in\Gamma_0(\spX)
      \\
      \implies &(C_2):  J_{\Psi_B\circ \opL} \in\Gamma_0(\spX).
    \end{align}
    We call $(C_1)$ \emph{the overall convexity condition} for $J_{\Psi_B\circ \opL}$.
    \item Under the overall convexity condition $(C_1)$, the cost function $J_{\Psi_B\circ\opL}$ in \eqref{eq:cLiGME-w-general-fidelity}
    can be decomposed into three convex functions:
    \begin{equation}
      \label{eq:J-sum-of-convex}
      J_{\Psi_B\circ\opL}= \mathfrak{d} + \mu \Psi\circ\opL 
      +\mu \left(\Psi + \frac{1}{2}\norm{B \cdot}^2_{\sptildeZ}\right)^*\circ B^*B\opL.
    \end{equation}
  \end{enumerate}
\end{proposition}
\begin{proof}
  See Appendix \ref{sec:a1}.
\end{proof}

\begin{remark}
  \label{remark:oc-cond}
  \quad
  \begin{enumerate}[(a)]
    \item (Comparison with existing conditions). 
    In the case of $f\coloneqq\frac{1}{2}\norm{y - \cdot}^2_{\spY}$, $f$ is $1$-strongly convex relative to $q_{\Id}$, and the relation $(C_0)\iff(C_1)$ holds. 
    In this case, the condition $(C_0)$ reproduces the overall convexity condition $
       A^*A - \mu\opL^*B^*B\opL\succeq\zeroMatrix_{\spX}
    $ for the LiGME model \cite[Proposition 1]{LiGME}.
    \item (An algebraic design of $B$ enjoying $(C_0)$).
    Assume that $\Lambda$ in Proposition \ref{prop:overall-convexity}(c) is available for designing $B$.
    From $\Lambda\succeq\zeroMatrix_{\spY}$, 
    $\Lambda$ can be decomposed as $\Lambda=V^* V$ with $V\in\setLO{\spY}{\spY}$.
    With such a linear operator $V$, $(C_0)$ can be rewritten as 
    \begin{equation}
      \label{eq:convexity-cond-remark}
      (VA)^*(VA)- \mu \opL^*B^*B\opL\succeq\zeroMatrix_{\spX}.
    \end{equation}
    GME matrix $B$ enjoying \eqref{eq:convexity-cond-remark} can be designed with $VA$ and the LDU decomposition of the matrix expression of $\opL$ \cite[Theorem 1]{Chen2023}.
    For other strategies for designing $B$, see, e.g., \cite{liu2022}.
    \item (Recommended choice of $\Lambda$ in $(C_0)$). 
    Although it seems to be difficult to design $B$ satisfying $(C_1)$ directly, we can design $B$ satisfying $(C_0)$ with $\Lambda \succeq \zeroMatrix_{\spY}$ as in (b).
    In designing $B$ satisfying $(C_0)$, we recommend avoiding the choice $\Lambda = \zeroMatrix_{\spY}$ if possible because (I) the condition $(C_0)$ with $\Lambda = \zeroMatrix_{\spY}$ imposes $B\opL = \zeroMatrix_{\spX,\sptildeZ}$ on $B$; (II) $\Psi_B \circ \opL $ with such a $B$ reproduces the convex regularizer $\Psi\circ \opL$ by
    \begin{equation}
      \Psi_B\circ\opL = \Psi\circ\opL - \min_{v\in\spZ}\left[
      \Psi(v) + \frac{1}{2}\norm{B(\opL(\cdot) - v)}^2_{\sptildeZ}
      \right]
      =\Psi\circ\opL - \underbrace{\min_{v\in\spZ}\left[
        \Psi(v) + \frac{1}{2}\norm{Bv}^2_{\sptildeZ}
      \right]}_{\mbox{constant}}
      \in\Gamma_0(\spX).
    \end{equation}
  \end{enumerate}
\end{remark}

Under the overall convexity condition $(C_1)$ in \eqref{eq:convexity-cond}, 
the following proposition presents sufficient conditions for the cost function $J_{\Psi_B\circ\opL}$ in \eqref{eq:cLiGME-w-general-fidelity} to enjoy the (weak) coercivity (see Section \ref{sec:Preliminaries-convex} for definitions), which is a key ingredient to show the existence of a minimizer.

\begin{proposition}[Sufficient conditions for (weak) coercivity of $J_{\Psi_B\circ\opL}$]
  \label{prop:coercivity-J}
  Consider Problem \ref{prob:cLiGME-w-general-fidelity} under $(C_1)$ in \eqref{eq:convexity-cond}.
  Then the following hold.
  \begin{enumerate}[\upshape (a)]
    \item Assume that $f\circ A$ is bounded below, i.e., $\inf_{x\in\spX} f\circ A (x)>-\infty$.
    Then
    \begin{equation}
      \label{eq:rec-J-geq-f-A-geq-0}
      (\forall x\in\spX)\quad
      \rec (J_{\Psi_B\circ\opL})(x)\geq \rec (f\circ A)(x)\geq 0
    \end{equation}
    holds.
    Moreover, for $x\in\spX$, the following relations hold:
    \begin{align}
      \label{eq:implication-rec-positive}
      "\rec(f\circ A)(x) > 0 \mbox{ or } x\notin \nullsp \opL" &\iff \rec (J_{\Psi_B\circ\opL})(x) > 0,\\
      \label{eq:implication-rec-0}
      "\rec(f\circ A)(x) = 0 \mbox{ and } x\in\nullsp \opL" &\iff \rec (J_{\Psi_B\circ\opL})(x) = 0.
    \end{align}
    \item If $f$ is coercive, then, for $x\in\spX$, we have 
    \begin{align}
    \label{eq:rec-J-positive-under-coercivity-of-f}
    &\rec (J_{\Psi_B\circ\opL}) (x) \geq 0,\\
    \label{eq:rec-J-0-and-null}
    \rec (J_{\Psi_B\circ\opL}) (x) &= 0
  \iff x\in\nullsp A \cap\nullsp\opL.
\end{align}
    \item The cost function $J_{\Psi_B\circ\opL}$ in \eqref{eq:cLiGME-w-general-fidelity} is weakly coercive if $f$ is coercive. 
    \item The cost function $J_{\Psi_B\circ\opL}$ in \eqref{eq:cLiGME-w-general-fidelity} is coercive if one of the following holds:
    \begin{enumerate}[\upshape (i)]
      \item $f$ is coercive and $\nullsp A\cap \nullsp\opL= \{0_{\spX}\}$;
      \item $f$ is bounded below and $\nullsp\opL =\{0_{\spX}\}$.
    \end{enumerate}
  \end{enumerate}
\end{proposition}
\begin{proof}
  See Appendix \ref{appendix:coercivity-J}.
\end{proof}

Based on Proposition \ref{prop:coercivity-J} and Fact \ref{fact:existence-minimizer}, the following theorem presents sufficient conditions to guarantee the existence of a minimizer of the model \eqref{eq:cLiGME-w-general-fidelity}.
\begin{theorem}[Sufficient conditions for existence of minimizer of the model \eqref{eq:cLiGME-w-general-fidelity}]
  \label{thm:existence-solution}
  Consider Problem \ref{prob:cLiGME-w-general-fidelity} under $(C_1)$ in \eqref{eq:convexity-cond}.
  Then the model \eqref{eq:cLiGME-w-general-fidelity} has a minimizer if one of the following holds:
  \begin{enumerate}[\upshape (i)]
  \item $f$ is coercive and $\opC^{-1}(\Cz)=\{x\in \spX \mid \opC x\in\Cz\}$ is asymptotically multipolyhedral;
  \item $f$ is coercive and $\ \nullsp A\cap \nullsp \opL= \{0_{\spX}\}$;
  \item $f$ is bounded below and $\ \nullsp \opL = \{0_{\spX}\}$;
  \item $\opC^{-1}(\Cz)$ is bounded.
  \end{enumerate} 
\end{theorem}
\begin{proof}
  \textbf{(Proof under the condition (i))} By Proposition \ref{prop:coercivity-J}(c), $J_{\Psi_B\circ\opL}$ is weakly coercive. Therefore, the existence of a minimizer is verified by Fact \ref{fact:existence-minimizer}(b).

  \noindent
  \textbf{(Proof under the condition (ii) or (iii))}
  By Proposition \ref{prop:coercivity-J}(d), $J_{\Psi_B\circ\opL}$ is coercive. Therefore, the existence of a minimizer is verified by Fact \ref{fact:existence-minimizer}(a).

  \noindent
  \textbf{(Proof under the condition (iv))}
  The existence of a minimizer is verified by Fact~\ref{fact:existence-minimizer}(a).
\end{proof}

The asymptotically multipolyhedral convex set in the condition (i) of Theorem~\ref{thm:existence-solution} covers a wide range of convex sets. 
We list typical examples of asymptotically multipolyhedral sets in Example \ref{example:asymptotically-multipolyhedral}.
For others, see, e.g., \cite{Goberna2010}.
\begin{example}[Asymptotically multipolyhedral sets]
  \label{example:asymptotically-multipolyhedral}
  \quad
  \begin{enumerate}[(a)]
    \item (Polyhedral set). Any polyhedral set can be decomposed as the sum of a compact polyhedral set and a polyhedral convex cone (see, e.g., \cite[Theorem 2.62]{royset2021}). Therefore, any polyhedral convex set is asymptotically multipolyhedral.
    \item\sloppy (Entire space). The entire space $\spX$ is a typical example of polyhedral sets and thus asymptotically multipolyhedral as well. 
    For the LiGME model, i.e., the model~\eqref{eq:cLiGME-w-general-fidelity} with $f=\frac{1}{2}\norm{y- \cdot}_{\spY}^2$ and $\opC^{-1}(\Cz)=\spX$, the existence of its minimizer is automatically guaranteed by the condition (i) in Theorem~\ref{thm:existence-solution}, while such an existence is assumed implicitly in \cite{LiGME}.
    \item (Linearly involved compact set). Let $\Cz$ be a compact convex set.
    Then $\opC^{-1}(\Cz)$ can be decomposed as 
    $\opC^{-1}(\Cz) =  \opC^\dag (\Cz) + \ker \opC$, where $\opC^\dagger (\Cz)$ is compact. Hence, $\opC^{-1}(\Cz)$ is asymptotically multipolyhedral.
  \end{enumerate}
\end{example}

\subsection{Proximal splitting type algorithm for proposed model}
\label{sec:algorithm}
In this section, we present a proximal splitting type algorithm (see \eqref{eq:K-M} in Theorem \ref{thm:convergence-analysis}(d) and Algorithm \ref{alg:proposed} below) which approximates iteratively a global minimizer of the model \eqref{eq:cLiGME-w-general-fidelity} under the following assumption.
\begin{assumption}[Assumption on the model \eqref{eq:cLiGME-w-general-fidelity} for algorithm]
  \label{assumption:alg}
  In Problem \ref{prob:cLiGME-w-general-fidelity}, assume the following\footnote{Assumption \ref{assumption:alg}(i) is reproduced here just for recall (see Problem \ref{prob:convex-opt} and Problem \ref{prob:cLiGME-w-general-fidelity}).}:
  \begin{enumerate}[(i)]
    \item $f$ is differentiable over $\spY$ and  $\nabla f$ is $\lipconst{\nabla f}$-Lipschitz continuous over $\spY$;
    \item the overall convexity condition $(C_1)$ in \eqref{eq:convexity-cond} is achieved; 
    \item the solution set $\setS\subset\spX$ of the model \eqref{eq:cLiGME-w-general-fidelity} is nonempty, i.e., a minimizer of \eqref{eq:cLiGME-w-general-fidelity} exists (see Theorem~\ref{thm:existence-solution} for sufficient conditions);
    \item the following sum and chain rules hold:
    \begin{align}
      \label{eq:sum-rule-J-and-constraint}
      \partial(J_{\Psi_B\circ \opL} + \iota_{\Cz}\circ \opC)&= \partial J_{\Psi_B\circ \opL} +\partial(\iota_{\iota_{\Cz}}\circ\opC),\\
      \label{eq:chain-rule-of-regularizer}
      \partial (\Psi\circ \opL)&= \opL^*\circ \partial \Psi\circ \opL,\\
      \label{eq:chain-rule-of-indicator}
      \partial(\iota_{\Cz}\circ \opC) &= \opC^*\circ \partial \iota_{\Cz} \circ \opC.
    \end{align}
  \end{enumerate}
\end{assumption} 

\begin{remark}[Sufficient conditions for Assumption \ref{assumption:alg}(iv)]
  The sum rule and chain rule in Fact \ref{fact:subdifferential} with $\dom (J_{\Psi_B\circ \opL}) = \dom (\Psi\circ \opL)$ in \eqref{eq:dom-of-J} yields the following sufficient conditions for Assumption \ref{assumption:alg}(iv):
    \begin{align}
      &\eqref{eq:sum-rule-J-and-constraint} \impliedby \emptyset  \neq 
      \ri ( \dom (\Psi\circ \opL) - {\opC^{-1}} (\Cz)), \quad
      \eqref{eq:chain-rule-of-regularizer} \impliedby0_{\spZ} \in \ri(\dom\Psi -\ran \opL),\\
      &\eqref{eq:chain-rule-of-indicator}\impliedby 0_{\spfrakZ} \in \ri (\dom \iota_{\Cz} - \ran \opC).
    \end{align}
\end{remark}

The following lemma gives a cocoercive constant of an operator which is used in convergence analysis.
\begin{lemma}
  \label{lemma:coercive-constant}
  Consider Problem \ref{prob:cLiGME-w-general-fidelity} under Assumption \ref{assumption:alg}.
  The following hold.
  \begin{enumerate}[\upshape (a)]
    \item Let $\lipconst{\nabla(f \circ A)}\in\setPR$ be a Lipschitz constant of $\nabla (f \circ A)$ (Note: we can choose $\lipconst{\nabla f}\norm{A}^2_{\mathrm{op}}$ as $\lipconst{\nabla(f\circ A)}$). Then $\nabla \mathfrak{d}= \nabla(f\circ A) - \mu\opL^*B^*B\opL$ is $\lipconst{\nabla (f\circ A)}$-Lipschitz continuous over $\spX$, where $\mathfrak{d}$ is defined in \eqref{eq:def-d}.
    \item Let $\lipconst{\nabla\mathfrak{d}}\in\setPR$ be a Lipschitz constant of $\nabla\mathfrak{d}$ (Note: we can choose $\lipconst{\nabla f}\norm{A}^2_{\mathrm{op}}$ as $\lipconst{\nabla \mathfrak{d}}$ by (a)).
  Set 
  \begin{equation}
    \label{eq:def-rho}
    \rho \coloneqq \frac{1}{\max\left\{\lipconst{\nabla\mathfrak{d}}, \mu\norm{B}_{\mathrm{op}}^2\right\}}>0.
  \end{equation}
  Then a mapping
  $
  \spX\times\spZ\to \spX\times\spZ:(x,v)\mapsto(\nabla \mathfrak{d} (x), \mu B^*B v)
  $
  is $\rho$-cocoercive over the Hilbert space $(\spX\times \spZ, \ip{\cdot}{\cdot}_{\spX\times \spZ}, \norm{\cdot}_{\spX\times \spZ})$.
  \end{enumerate}
\end{lemma}
\begin{proof}
  See Appendix \ref{appendix:a5}.
\end{proof}

The following theorem shows that the solution set of the model \eqref{eq:cLiGME-w-general-fidelity} can be expressed with the fixed point set of a certain nonexpansive operator under Assumption~\ref{assumption:alg}, and thus a minimizer of \eqref{eq:cLiGME-w-general-fidelity} can be approximated iteratively by Krasnosel'ski\u{\i}-Mann iteration (see Fact \ref{fact:KM-itr}).

\begin{theorem}
  \label{thm:convergence-analysis}
  Consider Problem \ref{prob:cLiGME-w-general-fidelity} under Assumption \ref{assumption:alg}. Let $\setS$ be the solution set of~\eqref{eq:cLiGME-w-general-fidelity} and $\lipconst{\nabla \mathfrak{d}}\in\setPR$ be a Lipschitz constant\footnote{
    From Lemma \ref{lemma:coercive-constant}(a), we can choose $\lipconst{\nabla f}\norm{A}^2_{\mathrm{op}}$ as $\lipconst{\nabla \mathfrak{d}}$.
  } of $\nabla\mathfrak{d}=\nabla (f\circ A) - \mu\opL^* B^*B\opL$, where $\mathfrak{d}$ is defined in \eqref{eq:def-d}.
  For a product space $\spH\coloneqq \spX\times \spZ\times\spZ\times \spfrakZ$, define $T:\spH\to\spH: (x,v,w,z)\mapsto (\xi,\zeta,\eta,\varsigma)$ with $(\sigma,\tau)\in \setPR\times\setPR$ by
  \begin{equation}
    \label{eq:def-T}
    \begin{split}
      \xi &\coloneqq\left(\Id - \frac{1}{\sigma} \nabla \mathfrak{d}\right) (x) - \frac{\mu}{\sigma}\opL^*B^*B v -\frac{\mu}{\sigma}\opL^*w - \frac{\mu}{\sigma}\opC^*z\\
    \zeta &\coloneqq \prox_{\frac{\mu}{\tau}\Psi}\left[
      \frac{2\mu}{\tau}B^*B\opL\xi
      -\frac{\mu}{\tau} B^*B\opL x + \left(\Id -\frac{\mu}{\tau}B^*B \right)(v)
    \right]\\
    \eta&\coloneqq \prox_{\Psi^*}(2\opL \xi -\opL x +w)\\
    \varsigma &\coloneqq (\Id - P_{\Cz})(2\opC\xi - \opC x +z).
    \end{split}
  \end{equation}
  Then the following hold.
  \begin{enumerate}[\upshape (a)]
    \item 
    $x^\diamond\in \setS$ holds if and only if
    there exists $(v^\diamond,w^\diamond,z^\diamond)\in\spZ\times\spZ\times \spfrakZ$ s.t.
    \begin{align}
      (0_{\spX}, 0_{\spZ}, 0_{\spZ}, 0_{\spfrakZ}) \in (F+G)(x^\diamond, v^\diamond, w^\diamond, z^\diamond),
    \end{align}
    where $F:\spH\to\spH$ and $G:\spH\to 2^{\spH}$ are defined by
    \begin{align}
      F(x,v,w,z) &= (\nabla \mathfrak{d}(x), \mu B^* B v, 0_{\spZ}, 0_{\spfrakZ}),\\
      G(x,v,w,z) &= \{\mu \opL B^*B v + \mu\opL^*w + \mu\opC^*z\}\times ( - \mu B^*B\opL x + \mu\partial\Psi(v)) \\
      &\quad\times
      (-\mu\opL x + \mu\partial \Psi^*(w))\times (-\mu\opC x + \mu\partial\iota_{\Cz}^* (z) ).
    \end{align}
    \item With $\Xi:\spH\to\spX:(x,v,w,z)\mapsto x$, the solution set $\mathcal{S}$ of \eqref{eq:cLiGME-w-general-fidelity} can be expressed as 
    \begin{equation}
      \label{eq:fixed-point-encode}
      \mathcal{S}=\Xi(\fix T)= \{
        \Xi
    (x^\diamond,v^\diamond,w^\diamond,z^\diamond)\in\spX\mid 
    (x^\diamond,v^\diamond,w^\diamond,z^\diamond)\in\fix T
    \}.
    \end{equation}
    \item With $\rho\in\setPR$ defined in \eqref{eq:def-rho}, choose $(\sigma, \tau)\in(0,\infty)\times (\frac{1}{2\rho},\infty)$ satisfying
    \begin{align}
      \label{eq:cond-sigma-tau}
      \sigma > \mu \norm{\opL^*\opL + \opC^*\opC}_{\mathrm{op}}+\frac{2\rho \mu^2\norm{B^*B\opL}^2_{\mathrm{op}} + \tau}{2\rho\tau - 1}
    \end{align}
    (see Remark \ref{remark:choice-sigma-tau} below for an example of $(\sigma,\tau)$). Then 
    \begin{equation}
      \label{eq:def-P}
      \opP \coloneqq \begin{bmatrix}
        \sigma \Id & -\mu\opL^*B^*B & -\mu \opL^* & -\mu\opC^*\\
        -\mu B^* B \opL & \tau \Id & \zeroMatrix_{\spZ, \spZ} & \zeroMatrix_{\spfrakZ,\spZ}\\
        -\mu \opL & \zeroMatrix_{\spZ,\spZ} & \mu \Id & \zeroMatrix_{\spfrakZ,\spZ}\\
        -\mu\opC & \zeroMatrix_{\spZ,\spfrakZ} & \zeroMatrix_{\spZ,\spfrakZ} & \mu\Id
      \end{bmatrix}\succ \zeroMatrix_{\spH}
    \end{equation}
    holds.
    Furthermore,
    $T$ is $\frac{2}{4-\theta}$-averaged nonexpansive over $(\spH, \ip{\cdot}{\cdot}_{\opP}, \norm{\cdot}_{\opP})$ with
    \begin{equation}
      \theta\coloneqq \frac{ \sigma  + \tau- \mu\norm{\opL^*\opL + \opC^*\opC }_{\mathrm{op}}}{\rho(\sigma \tau - \tau\mu\norm{\opL^*\opL + \opC^*\opC }_{\mathrm{op}} - \mu^2 \norm{B^* B \opL}_{\mathrm{op}}^2)} < 2.
    \end{equation}
    \item  For any initial point $h_0\coloneqq(x_0,v_0,w_0,z_0)\in\spX\times\spZ\times\spZ\times\spfrakZ = \spH$, generate the sequence $(h_k)_{k\in\setN_0} \subset \spH$ by the Krasnosel'ski\u{\i}-Mann iteration 
    \begin{equation}
      \label{eq:K-M}
      (\forall k \in\setN_0)\quad 
      h_{k+1} \coloneqq T(h_k).
    \end{equation}
    Then the sequence $(h_k)_{k\in\setN_0}$ converges to a point in $\fix T$, which implies that the sequence $x_k\coloneqq \Xi(h_k)$ converges to a point in the solution set $\setS$.
  \end{enumerate}
\end{theorem}
\begin{proof}
  See Appendix \ref{appendix:a6}.
\end{proof}

Algorithm \ref{alg:proposed} illustrates a concrete expression of the Krasnosel'ski\u{\i}-Mann iteration \eqref{eq:K-M} of the proposed nonexpansive operator $T$.

\begin{algorithm}[h]
  \caption{Proposed algorithm for Problem \ref{prob:cLiGME-w-general-fidelity}}
  \label{alg:proposed}
  \begin{algorithmic}[1]
    \State Choose $h_0\coloneqq\left(x_0, v_0, w_0, z_0\right)\in \spX\times\spZ\times\spZ\times\spfrakZ$.
    \State Choose $(\sigma,\tau)\in(0,\infty)\times(\frac{1}{2\rho},\infty)$ 
    satisfying \eqref{eq:cond-sigma-tau}.
    \For{$k=0,1,2,\cdots$ }
    \State $x_{k+1} \coloneqq\left(\Id - \frac{1}{\sigma} \nabla \mathfrak{d}\right) (x_k) - \frac{\mu}{\sigma}\opL^*B^*B v_k -\frac{\mu}{\sigma}\opL^*w_k - \frac{\mu}{\sigma}\opC^*z_k$
    \State $v_{k+1}\coloneqq \prox_{\frac{\mu}{\tau}\Psi}\left[
      \frac{2\mu}{\tau}B^*B\opL x_{k+1}
      -\frac{\mu}{\tau} B^*B\opL x_k + \left(\Id -\frac{\mu}{\tau}B^*B \right)(v_k)
    \right]$
    \State $w_{k+1}\coloneqq \prox_{\Psi^*}(2\opL x_{k+1} -\opL x_k +w_k)$
    \State $z_{k+1}\coloneqq(\Id - P_{\Cz})(2\opC x_{k+1} - \opC x_k +z_k)$
     \State $h_{k+1} \coloneqq (x_{k+1},v_{k+1},w_{k+1},z_{k+1} )$
    \EndFor
  \end{algorithmic}
\end{algorithm}
\begin{remark}[Choice of $(\sigma,\tau)$ satisfying \eqref{eq:cond-sigma-tau}]
  \label{remark:choice-sigma-tau}
  In numerical experiments in Section \ref{sec:numerical}, we used $\tau= \frac{3}{2\rho}$ and
  \begin{equation}
    \sigma \coloneqq 1.001 \times\left(\mu \norm{\opL^*\opL + \opC^*\opC}_{\mathrm{op}}+\frac{2\rho \mu^2\norm{B^*B\opL}^2_{\mathrm{op}} + \tau}{2\rho\tau - 1}\right).
  \end{equation}
\end{remark}
\begin{remark}[Comparison with existing algorithms]
    The proposed algorithm can be seen as an extension of \cite[Algorithm 1]{LiGME} and \cite[Algorithm 1]{cLiGME} which are proposed for the quadratic data fidelity case, i.e., $f\coloneqq \frac{1}{2}\norm{y-\cdot}^2_{\spY}$.
    Moreover, the proposed algorithm is applicable without the conditions $\dom\Psi = \spZ$ and $\Psi\circ(-\Id) = \Psi$, imposed in \cite[Algorithm 1]{LiGME} and \cite[Algorithm 1]{cLiGME}.
\end{remark}

\section{Quadratic extrapolation for broader applicability of cLiGME-SCF}
\label{sec:fidelity}
In this section, we introduce a reformulation technique of an observation loss function to enjoy favorable properties for the proposed model.
Before introducing a reformulation technique, we discuss favorable properties of the observation loss function $f\in\Gamma_0(\spY)$ for the proposed model. 
As mentioned in Remark~\ref{remark:oc-cond}(c), in a case where we design $B$ achieving $(C_0)$, $f$ is desired to be $1$-strongly convex relative to $q_\Lambda$ with $\Lambda\succneqq \zeroMatrix_{\spY}$ over $\spY$.
Moreover, the Lipschitz differentiability of $f$ over $\spY$ is required for convergence guarantee of the proposed algorithm (see Problem \ref{prob:cLiGME-w-general-fidelity} and Assumption~\ref{assumption:alg}).
Therefore, favorable properties of $f$ for the proposed model are summarized as follows:
\begin{equation}
  \label{eq:favorable-properties}
  \begin{cases}
    \mbox{$f$ is strongly convex relative to $q_{\Lambda}$ with $\Lambda\succneqq\zeroMatrix_{\spY}$ over $\spY$},\\
    \mbox{$f$ is differentiable over $\spY$ and $\nabla f$ is Lipschitz continuous over $\spY$}.
  \end{cases}
\end{equation}

To broaden applicability of the proposed model \eqref{eq:cLiGME-w-general-fidelity} and Algorithm \ref{alg:proposed}, in this section, we consider a case where an original observation loss function \footnote{
    We use symbol $\fBeforeExtension$ for the original observation loss function instead of $f$ because $\fBeforeExtension$ does not enjoy the properties in \eqref{eq:favorable-properties}.
  }
  $\fBeforeExtension\in\Gamma_0(\spY)$ satisfies favorable properties in \eqref{eq:favorable-properties} with $f\coloneqq \fBeforeExtension$ over $A(\opC^{-1}(\Cz))$ but not necessarily over $\spY$.
By assuming a separable structure of $\fBeforeExtension$, we introduce a reformulation of an optimization problem with $\fBeforeExtension$ (see Problem \ref{prob:relaxed-fidelity} below) into Problem \ref{prob:cLiGME-w-general-fidelity} with $f\coloneqq \widetilde{\fBeforeExtension}$ enjoying \eqref{eq:favorable-properties}.
See Section \ref{sec:numerical} for examples of Problem \ref{prob:relaxed-fidelity} below.

\begin{problem}
  \label{prob:relaxed-fidelity}
  In Problem \ref{prob:cLiGME-w-general-fidelity} with $\spY\coloneqq \setR^m$, define closed convex sets
  $\Pi_i \subset \setR \ (i\in\{1,2\dots, m\})$ by
  \begin{equation}
    \Pi_i \coloneqq \left\{t\in\setR \mid \inf_{\opC x\in\Cz} [Ax]_i \leq t \leq \sup_{\opC x\in\Cz} [Ax]_i\right\}\subset \setR,
  \end{equation}
  and set $\Pi\coloneqq \bigtimes_{i=1}^m\Pi_i(\subset \spY)$.
  (Note: From the definition of $\Pi$, the inclusion $A^{-1}(\Pi)\supset\opC^{-1}(\Cz) $ holds.)
  Assume that the observation loss function $\fBeforeExtension$ enjoys the following conditions:
  \begin{enumerate}[(i)]
    \item $\fBeforeExtension\in\Gamma_0(\spY)$ is the separable sum 
    $\fBeforeExtension(u) \coloneqq \sum_{i=1}^m \fBeforeExtension_i([u]_i)
    $ of univariate functions $\fBeforeExtension_i$, where each $\fBeforeExtension_i\in\Gamma_0(\setR) \ (i\in\{1,2,\ldots,m\})$ is twice continuously differentiable over $\Pi_i$ and
   \begin{equation}
    \label{eq:assumption-hessian-sup}
    \sup_{t\in\Pi_i} \fBeforeExtension_i''(t)<\infty;
   \end{equation}
    \item a diagonal matrix $\Lambda \in\setR^{m\times m}$ with 
    \begin{equation}
      \label{eq:def-lambda}
      [\Lambda]_{i,i}\coloneqq
      \inf_{t\in\Pi_i} \fBeforeExtension_i''(t)\geq 0
    \end{equation}
    satisfies $\Lambda\succneqq \zeroMatrix_{\spY}$, i.e.,
    $\fBeforeExtension$ is 1-strongly convex relative to $q_{\Lambda}$ over $\Pi$.
  \end{enumerate}
  Then consider the optimization model 
  \begin{equation}
    \label{eq:before-reformulated-model}
     \minimize_{\opC x\in\Cz} \fBeforeExtension\circ A(x) + \mu \Psi_B\circ\opL (x).
  \end{equation}
\end{problem}

The observation loss function $\fBeforeExtension$ in Problem \ref{prob:relaxed-fidelity} does not necessarily enjoy properties in \eqref{eq:favorable-properties} with $f\coloneqq \fBeforeExtension$.
However, by constructing a smooth convex extension $\widetilde{\fBeforeExtension}(\cdot;r)$ of $\fBeforeExtension |_{\opC^{-1}(\Cz)}$ as in Proposition \ref{prop:separable-extension} below, we can reformulate Problem \ref{prob:relaxed-fidelity} into an instance of Problem~\ref{prob:cLiGME-w-general-fidelity} enjoying \eqref{eq:favorable-properties} with $f\coloneqq \widetilde{\fBeforeExtension}(\cdot;r)$, where $r$ is a univariate convex function. A constant function $(\forall t\in\setR) \ r(t)=0$ enjoys conditions for $r$ in Proposition~\ref{prop:separable-extension} below, and we mainly use this choice for reformulation of the model (see Remark \ref{remark:choice-r} below).
However, in Corollary~\ref{corollary:existence-minimizer-extension}(d), $\widetilde{\fBeforeExtension}(\cdot;r)$ with special nonconstant $r$ helps us to derive sufficient conditions for the existence of a minimizer. For this reason, we present Proposition \ref{prop:separable-extension} with possibly nonconstant function $r$.

\begin{proposition}[Construction of smooth convex extension of $\fBeforeExtension|_{\opC^{-1}(\Cz)}$]
  \label{prop:separable-extension}
In Problem \ref{prob:relaxed-fidelity}, let $r:\setR\to\setR$ be a twice continuously differentiable convex function such that
  \begin{equation}
    \label{eq:cond-r}
      \inf_{t\in\setR} r''(t) \geq 0,\quad \sup_{t\in\setR} r''(t) < \infty,\quad
      r(0) = r'(0)=r''(0) = 0.
  \end{equation}
  For every $i\in \{1,2,\ldots, m\}$, define functions $(1\leq i \leq m)\ \widetilde{\fBeforeExtension}_i(\cdot;r):\setR\to(-\infty,\infty)$ by
    \begin{equation}
      \label{eq:def-ftilde-i-r}
      (\forall t\in\setR) \ \
      \widetilde{\fBeforeExtension}_i(t;r) \coloneqq \begin{cases}
        \frac{\fBeforeExtension_i''(l_i)}{2} (t -l_i)^2 + \fBeforeExtension_i'(l_i) (t-l_i) + \fBeforeExtension_i(l_i) +r(-t+l_i), & t\leq l_i\\
        \fBeforeExtension_i(t), & l_i <t < h_i\\
        \frac{\fBeforeExtension_i''(h_i)}{2} (t -h_i)^2 + \fBeforeExtension_i'(h_i) (t-h_i) + \fBeforeExtension_i(h_i) +r(t-h_i), & t\geq h_i,
      \end{cases}
    \end{equation}
    where $l_i\coloneqq \inf (\Pi_i)\in[-\infty,\infty)$ and $h_i\coloneqq \sup (\Pi_i)\in(-\infty,\infty]$.
    Then, for 
    \begin{equation}
      \label{eq:def-ftilde-r}
      \widetilde{\fBeforeExtension}(\cdot; r):\spY\to\setR: u\mapsto \sum_{i=1}^m \widetilde{\fBeforeExtension}_i([u]_i;r),
    \end{equation}
    the following hold.
    \begin{enumerate}[\upshape (a)]
      \item $\widetilde{\fBeforeExtension}(\cdot; r)$ is twice continuously differentiable over $\spY$ and $\nabla \widetilde{\fBeforeExtension}(\cdot; r)$ is Lipschitz continuous over $\spY$ with a constant
      \begin{equation}
        \label{eq:lip-const-with-r}
        \lipconst{\nabla \widetilde{\fBeforeExtension}(\cdot;r)} \coloneqq \max\left\{
          \sup_{t\in \Pi_1} \fBeforeExtension_1''(t),\sup_{t\in \Pi_2} \fBeforeExtension_2''(t), \ldots , \sup_{t\in \Pi_m} \fBeforeExtension_m''(t)
        \right\} + \sup_{t\in\setR} r''(t)<\infty;
      \end{equation}
      \item $\widetilde{\fBeforeExtension}(\cdot;r)$ is $1$-strongly convex relative to $q_{\Lambda}$ over $\spY$ with $\Lambda\succneqq\zeroMatrix_{\spY}$ defined in \eqref{eq:def-lambda};
      \item $\fBeforeExtension \circ A(x) = \widetilde{\fBeforeExtension}(A x;r)$ holds for every $ x\in\opC^{-1}(\Cz)$;
      \item Under the condition\footnote{
      \label{footnote:existence-r}
      E.g.,
      set 
   $
    r(t) = \begin{cases}
      0, & (t\leq 0)\\
      \frac{1}{6}t^3, &(0<t<1)\\
      \frac{1}{2}t^2 - \frac{1}{2}t +\frac{1}{6}, & (1\leq t)
    \end{cases}
  $ whose Hessian is 
  $
    r''(t) = \begin{cases}
      0, & (t\leq 0)\\
      t, &(0<t<1)\\
      1, & (1\leq t).
    \end{cases}
  $
    } $\lim_{t\to +\infty}r(t)/t = \infty$, we have \begin{align}
      &\mbox{$\widetilde{\fBeforeExtension}(\cdot;r)$ is coercive if $\fBeforeExtension$ is coercive, }\\
      &\mbox{$\widetilde{\fBeforeExtension}(\cdot;r)$ is bounded below if $\fBeforeExtension$ is bounded below.}
    \end{align}
    \end{enumerate}
    \end{proposition}
\begin{proof}
  See Appendix \ref{appendix:proof-separable-extension}.
\end{proof}

\begin{remark}[Recommended choice of $r$]
  \label{remark:choice-r}
  In practice, we recommend employing $(\forall t \in\setR) \ r(t)= 0$ for constructing a smooth convex extension $\widetilde{\fBeforeExtension}(\cdot;r)$ in \eqref{eq:def-ftilde-i-r} and \eqref{eq:def-ftilde-r} because $\lipconst{\nabla \widetilde{\fBeforeExtension}(\cdot;r)}$ in \eqref{eq:lip-const-with-r} achieves the smallest value with this choice.
\end{remark}

With the convex extension $\widetilde{\fBeforeExtension}(\cdot;r)$ of $\fBeforeExtension |_{\opC^{-1} (\Cz)}$ in Proposition \ref{prop:separable-extension},
we can reformulate Problem \ref{prob:relaxed-fidelity} into Problem \ref{prob:cLiGME-w-general-fidelity} satisfying \eqref{eq:favorable-properties} as follows.

\begin{corollary}[Reformulation of Problem \ref{prob:relaxed-fidelity} into Problem \ref{prob:cLiGME-w-general-fidelity}]
  \label{corollary:existence-minimizer-extension}
  For Problem \ref{prob:relaxed-fidelity}, define $\widetilde{\fBeforeExtension}\coloneqq \widetilde{\fBeforeExtension}(\cdot;r)$ as in Proposition \ref{prop:separable-extension} with $(\forall t \in\setR) \ r(t)=0$, and consider the following optimization model:
  \begin{equation}
    \label{eq:reformulated-model}
    \minimize_{\opC x\in\Cz} \widetilde{\fBeforeExtension}\circ A(x) + \mu \Psi_B\circ\opL (x).
  \end{equation}
  Then the following hold.
  \begin{enumerate}[\upshape (a)]
    \item $\widetilde{\fBeforeExtension}$ is twice continuously differentiable over $\spY$ and $\nabla \widetilde{\fBeforeExtension}$ is $\lipconst{\nabla \widetilde{\fBeforeExtension}}$-Lipschitz continuous over $\spY$ with 
    $
        \lipconst{\nabla \widetilde{\fBeforeExtension}}=\max\left\{
          \sup_{t\in \Pi_1} \fBeforeExtension_1''(t),\sup_{t\in \Pi_2} \fBeforeExtension_2''(t), \ldots , \sup_{t\in \Pi_m} \fBeforeExtension_m''(t)
        \right\} <\infty.
    $
    \item The solution set of \eqref{eq:reformulated-model} coincides with that of \eqref{eq:before-reformulated-model}, and the model \eqref{eq:reformulated-model} is an instance of Problem~\ref{prob:cLiGME-w-general-fidelity}.
    \item If $A^*\Lambda A - \mu\opL^*B^*B\opL\succeq \zeroMatrix_{\spX}$ holds with $\Lambda$ defined in \eqref{eq:def-lambda}, then $\widetilde{\fBeforeExtension}\circ A$ is $\mu$-strongly convex relative to $q_{\opL^*B^*B\opL}$ and $\widetilde{\fBeforeExtension}\circ A + \mu \Psi_B\circ\opL\in\Gamma_0(\spX)$ holds.
    \item Assume that $A^*\Lambda A - \mu\opL^*B^*B\opL\succeq \zeroMatrix_{\spX}$ holds. Then both models~\eqref{eq:before-reformulated-model} and \eqref{eq:reformulated-model} have a minimizer if one of the following holds:
    \begin{enumerate}[\upshape (i)]
      \item $\fBeforeExtension$ is coercive and $\opC^{-1}(\Cz)=\{x\in \spX \mid \opC x\in\Cz\}$ is asymptotically multipolyhedral;
      \item $\fBeforeExtension$ is coercive and $\ \nullsp A\cap \nullsp \opL= \{0_{\spX}\}$;
      \item $\fBeforeExtension$ is bounded below and $\ \nullsp \opL = \{0_{\spX}\}$;
      \item $\opC^{-1}(\Cz)$ is bounded.
    \end{enumerate} 
    \item Assume that (i) $A^*\Lambda A - \mu\opL^*B^*B\opL\succeq \zeroMatrix_{\spX}$ holds, (ii) there exists a minimizer of \eqref{eq:reformulated-model}, and (iii) the sum and chain rules in \eqref{eq:sum-rule-J-and-constraint}, \eqref{eq:chain-rule-of-regularizer} and \eqref{eq:chain-rule-of-indicator} hold. 
    Then the sequence $(x_k)_{k\in\setN_0}$ generated by applying Algorithm \ref{alg:proposed} to the model \eqref{eq:reformulated-model} converges to a global minimizer of the original model \eqref{eq:before-reformulated-model}.
  \end{enumerate}
\end{corollary}
\begin{proof}
  See Appendix \ref{appendix:proof-reformulation}.
\end{proof}

\begin{remark}[Key for the existence of a minimizer of \eqref{eq:reformulated-model} is the coercivity or the lower boundedness of $\fBeforeExtension$]
  The reformulated observation loss function
  $\widetilde{\fBeforeExtension}$ in Corollary \ref{corollary:existence-minimizer-extension} is not necessarily coercive even if the original loss function $\fBeforeExtension$ is coercive. Likewise, $\widetilde{\fBeforeExtension}$ is not necessarily bounded below if $\fBeforeExtension$ is bounded below. This situation seems to imply that the coercivity or lower boundedness of $\fBeforeExtension$ does not guarantee the conditions (i)-(iii) in Theorem \ref{thm:existence-solution} for the existence of a minimizer of the model \eqref{eq:reformulated-model}. However, Corollary \ref{corollary:existence-minimizer-extension}(d) shows that the existence of a minimizer of the model \eqref{eq:reformulated-model} can be guaranteed by checking the coercivity (or lower boundedness) of $\fBeforeExtension$ instead of $\widetilde{\fBeforeExtension}$.
\end{remark}

\section{Applications and numerical experiments of proposed model and algorithm}
\label{sec:numerical}
\subsection{Poisson denoising for piecewise constant signal estimation}
\label{sec:poisson-denosing}
\subsubsection{Conventional convex model for Poisson denoising}
\sloppy 
We consider an estimation problem from the observation corrupted by Poisson noise.
The task is to estimate the target 
piecewise constant 
signal 
$\xstar\in\Cz(\subset\setPR^n)$ ($\Cz$: compact and closed convex subset of $\setPR^n$)
from
  $y \in(\setN_0)^n$
corrupted by Poisson noise $\varepsilon$, i.e.,
each $y_i$ follows the Poisson distribution
\begin{equation}
  \label{eq:poisson-dist}
  P([y]_i = k | [\xstar]_i) = \frac{([\xstar]_i)^k e^{- [\xstar]_i}}{k !}
\end{equation}
with parameter $[\xstar]_i$. 
The negative log-likelihood function associated with \eqref{eq:poisson-dist} is 
\begin{equation}
  -\log\left( \prod_{i=1}^{n} \frac{([x]_i)^{[y]_i} e^{- [x]_i}}{ [y]_i!}\right)
  =  \sum_{i=1}^{n}\left( - [y]_i\log([x]_i) + [x]_i - \log([y]_i!) \right).
\end{equation}
Therefore, the observation loss function $\fBeforeExtension^{\lrangle{\mathrm{poisson}}}:\setR^m\to(-\infty,\infty]: u\mapsto \sum_{i=1}^m \fBeforeExtension^{\lrangle{\mathrm{poisson}}}_i([u]_i)$
with 
\begin{equation}
  \label{eq:poisson-fidelity}
  \left\{
  \begin{aligned}[1]
    &(\forall i \in \mathbb{S}_{>0}\coloneqq \{\widetilde{i}\in \{ 1,2, \ldots m\} \mid y_{\widetilde{i}}>0\})\quad
  \fBeforeExtension^{\lrangle{\mathrm{poisson}}}_i(t) \coloneqq \begin{cases}
    t - [y]_i \log (t), & t>0\\
    \infty, & t\leq 0
  \end{cases}\\
   &(\forall i \in \mathbb{S}_{=0}\coloneqq\{\widetilde{i}\in \{ 1,2, \ldots m\} \mid y_{\widetilde{i}}=0\}) \quad
   \fBeforeExtension^{\lrangle{\mathrm{poisson}}}_i(t) \coloneqq \begin{cases}
    t ,  & t\geq 0\\
    \infty, & t< 0
  \end{cases}
  \end{aligned}
  \right.
\end{equation}
has been utilized for Poisson denoising, e.g., in \cite{byrne1993,triet2007,zanella2009,caroline2009}.

From the piecewise constancy of $\xstar$, $D \xstar$ is sparse with the first-order difference matrix 
\begin{equation}
      D \coloneqq \begin{bmatrix}
        -1 & 1      &        & \\
        & \ddots & \ddots  & \\
        &        & -1      & 1
      \end{bmatrix}\in \setR^{(n-1)\times n}.
\end{equation}
To promote the sparsity of $D x$, the total variation (TV) regularizer \cite{rudin1992} $\norm{\cdot}_1\circ D$ has been utilized as a convex regularizer where $\norm{\cdot}_{1}\in\Gamma_0(\setR^{n-1})$ is the $\ell_1$ norm\footnote{$\norm{\cdot}_{1}$ is prox-friendly. See, e.g., \cite[Example 24.22]{CAaMOTiH}.}. 
With $\fBeforeExtension^{\lrangle{\mathrm{poisson}}}$ defined in \eqref{eq:poisson-fidelity} and the TV regularizer $\norm{\cdot}_1\circ D$, a convex model for estimation of $\xstar$ has been formulated as \cite{triet2007}
\begin{equation}
  \label{eq:convex-poisson}
  \minimize_{x\in\Cz} \fBeforeExtension^{\lrangle{\mathrm{poisson}}}(x) + \mu \norm{\cdot}_1\circ D(x).
\end{equation}

\subsubsection{Nonconvex enhancement of convex model for Poisson denoising and reformulation into Problem \ref{prob:cLiGME-w-general-fidelity}}
As a nonconvex enhancement of the convex model \eqref{eq:convex-poisson}, we propose 
\begin{equation}
  \label{eq:proposed-poisson}
  \minimize_{ x\in\Cz} \fBeforeExtension^{\lrangle{\mathrm{poisson}}}(x) + \mu (\norm{\cdot}_1)_B\circ D(x).
\end{equation}
For every $i\in\{1,2,\ldots,m\}$, ${\fBeforeExtension_i^{\lrangle{\mathrm{poisson}}}}$ is twice continuously differentiable over $\setPR$ with 
$
  (\forall t\in\setPR) \ {\fBeforeExtension^{\lrangle{\mathrm{poisson}}}_i}''(t) = \begin{cases}
    \frac{[y]_i}{t^2}, & i\in\mathbb{S}_{>0}\\
    0, & i\in\mathbb{S}_{=0}
  \end{cases}.
$
Since $\Cz$ is compact and satisfies $\Cz\subset\setPR^n$, the inequality
\begin{equation}
  (\forall i\in\{1,2,\ldots,m\})\quad 
  \sup_{t \in\Pi_i} {\fBeforeExtension^{\lrangle{\mathrm{poisson}}}_i}''(t) =\begin{cases}
    \frac{[y]_i}{l_i^2}, & i\in\mathbb{S}_{>0}\\
    0, & i\in\mathbb{S}_{=0} 
  \end{cases}
\end{equation}
holds with $l_i \coloneqq \inf(\Pi_i)(>0)$. Moreover, $\fBeforeExtension^{\lrangle{\mathrm{poisson}}}$ is $1$-strongly convex relative to $q_{\Lambda}$ over $\Pi$, where $\Pi_i$ and $\Pi$ are defined as in Problem \ref{prob:relaxed-fidelity} and a diagonal matrix $\Lambda$ in \eqref{eq:def-lambda} is given with $h_i \coloneqq \sup(\Pi_i)(<\infty)$ by
\begin{equation}
  \label{eq:lambda-poisson}
  [\Lambda]_{i,i}
  \coloneqq 
  \begin{cases}
    \frac{[y]_i}{h_i^2}, & i\in\mathbb{S}_{>0}\\
    0, & i\in\mathbb{S}_{=0}.
  \end{cases}
\end{equation}
Hence, the model \eqref{eq:proposed-poisson} is an instance of Problem \ref{prob:relaxed-fidelity}.
By Corollary \ref{corollary:existence-minimizer-extension},
the model \eqref{eq:proposed-poisson} can be reformulated into 
\begin{equation}
  \label{eq:proposed-poisson-reformulated}
  \minimize_{ x\in\Cz} f^{\lrangle{\mathrm{poisson}}}(x) + \mu (\norm{\cdot}_1)_B\circ D(x)
\end{equation}
which is an instance of Problem \ref{prob:cLiGME-w-general-fidelity}.
With a GME matrix $B$ satisfying
\begin{equation}
  \label{eq:convexity-cond-poisson}
  \Lambda - \mu D^* B^*B D\succeq \zeroMatrix_{\spX},
\end{equation}
the cost function of \eqref{eq:proposed-poisson-reformulated} becomes convex by Corollary \ref{corollary:existence-minimizer-extension}(c). Moreover, the existence of a minimizer of the model \eqref{eq:proposed-poisson-reformulated} is verified by Corollary \ref{corollary:existence-minimizer-extension}(d) with the compactness of $\Cz$.

\subsubsection{Numerical experiments of Poisson denoising}
We conducted numerical experiments on estimation of $\xstar\in\Cz \coloneqq[5,40]^n  \ (n=150)$ from its noisy observation $y$ in \eqref{eq:poisson-dist}, where $\xstar$ is a piecewise constant signal shown in Figure \ref{fig:reconst-examples-poisson}(a) below.
We compared the conventional model \eqref{eq:convex-poisson} and the proposed model \eqref{eq:proposed-poisson-reformulated}.
For the proposed model \eqref{eq:proposed-poisson-reformulated}, we designed $B$ by \cite[Theorem 1]{Chen2023} with 
$(A,\opL,\mu,\theta) = (\Lambda^{\frac{1}{2}}, D, \mu, 0.99)$ to enjoy \eqref{eq:convexity-cond-poisson}. 
For minimizations of \eqref{eq:convex-poisson} and \eqref{eq:proposed-poisson-reformulated}, we applied \footnote{By setting $B=\zeroMatrix$, the model \eqref{eq:proposed-poisson} reproduces the model \eqref{eq:convex-poisson}. Thus, we can obtain a minimizer of the model \eqref{eq:convex-poisson} by solving the model \eqref{eq:proposed-poisson-reformulated} with $B=\zeroMatrix$.} the proposed algorithm \eqref{eq:K-M}, where $(\sigma,\tau)$ is set by Remark \ref{remark:choice-sigma-tau}. The proposed algorithm was terminated after the residual achieved $\norm{h_k-h_{k-1}}_{\spH} < 10^{-6}$.

Figure \ref{fig:mu-vs-MAE}(a)-(c) show respectively the  averages\footnote{Following \cite{dupe2009}, we use absolute error as a criterion.} of
\begin{align}
  &\mbox{(a) absolute error (AE): } \norm{\bar{x} - \xstar}_1,\quad \quad \quad \quad
  \mbox{(b) squared error (SE): } \norm{\bar{x} -\xstar}^2_{\spX},\\
  &\mbox{(c) number of entries such that $[D\bar{x}]_i>10^{-4}$}
\end{align}
over 1000 independent observations.
From Figure \ref{fig:mu-vs-MAE}(a) and (b), we see that the conventional model~\eqref{eq:convex-poisson} and the proposed model~\eqref{eq:proposed-poisson-reformulated} achieve the lowest MAE and MSE respectively on $\mu = 0.6$ and $\mu=1$. From these figures, we see that the proposed model achieves lower MAE and MSE than conventional model.
As can be seen from Figure~\ref{fig:mu-vs-MAE}(c), the average number of nonzero entries for $D\bar{x}$ by the model~\eqref{eq:convex-poisson} with $\mu=0.6$ is $22.9$, whereas the average number of nonzero entries for $D\bar{x}$ by the model~\eqref{eq:proposed-poisson-reformulated} with $\mu=1$ is $18.1$.
Hence, we conclude that the proposed model~\eqref{eq:proposed-poisson-reformulated} outperforms the conventional model \eqref{eq:convex-poisson} in a view of MAE and MSE, with promoting the sparsity of $D\bar{x}$ more effectively.

\begin{figure}[t]
  \centering
  \begin{minipage}[b]{0.3\linewidth}
    \centering
    \includegraphics[keepaspectratio, scale=0.1]
    {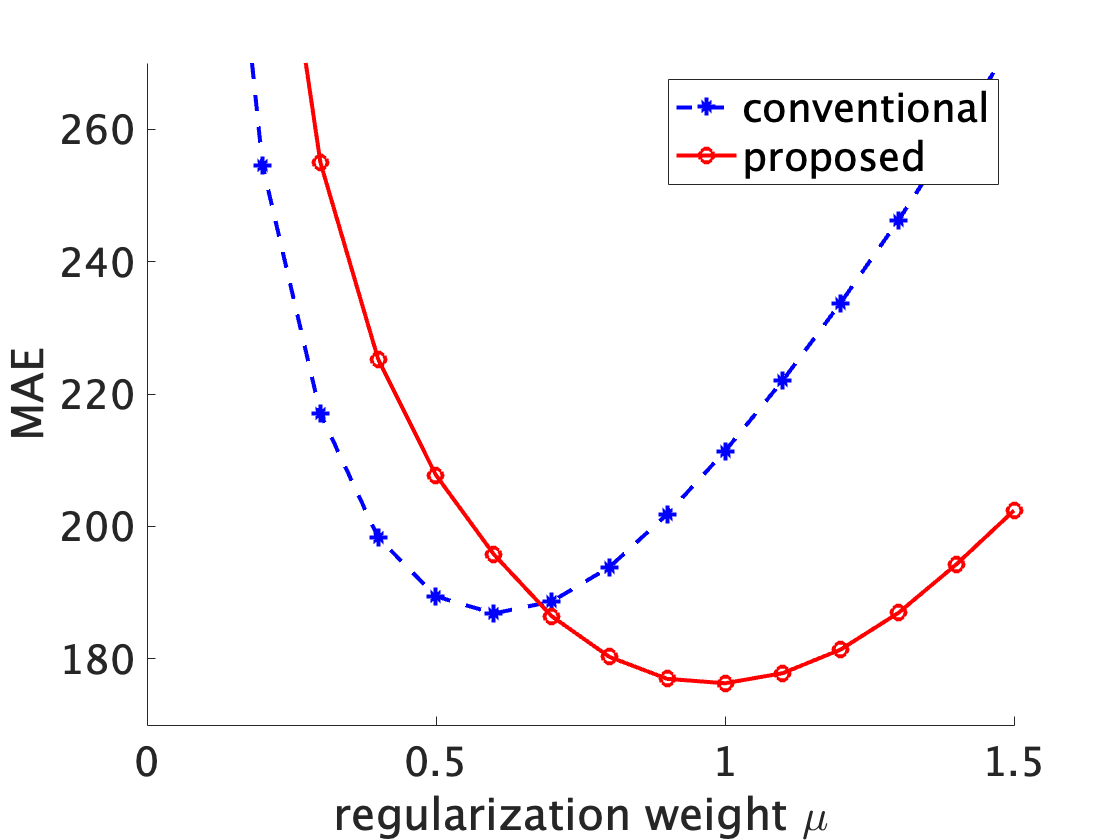}
    \subcaption{}
  \end{minipage}
  \begin{minipage}[b]{0.3\linewidth}
    \centering
    \includegraphics[keepaspectratio,scale=0.1]
    {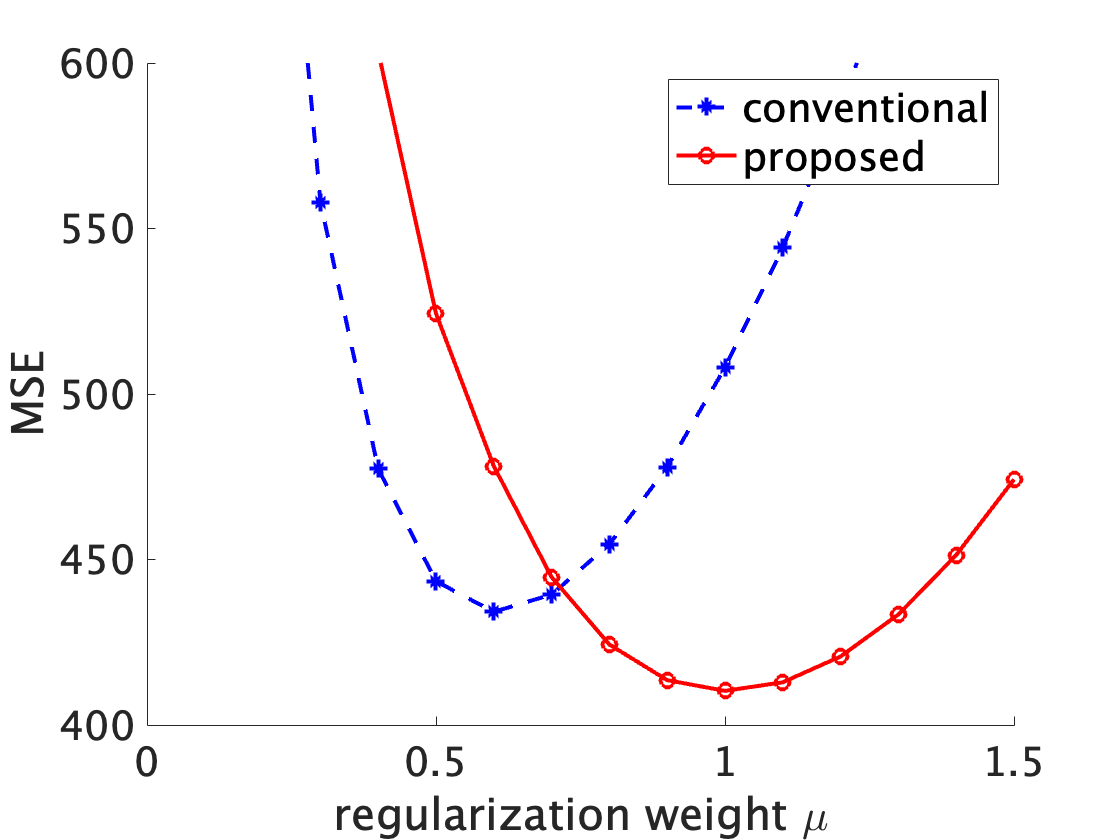}
    \subcaption{}
  \end{minipage}
  \begin{minipage}[b]{0.3\linewidth}
    \centering
    \includegraphics[keepaspectratio,scale=0.1]
    {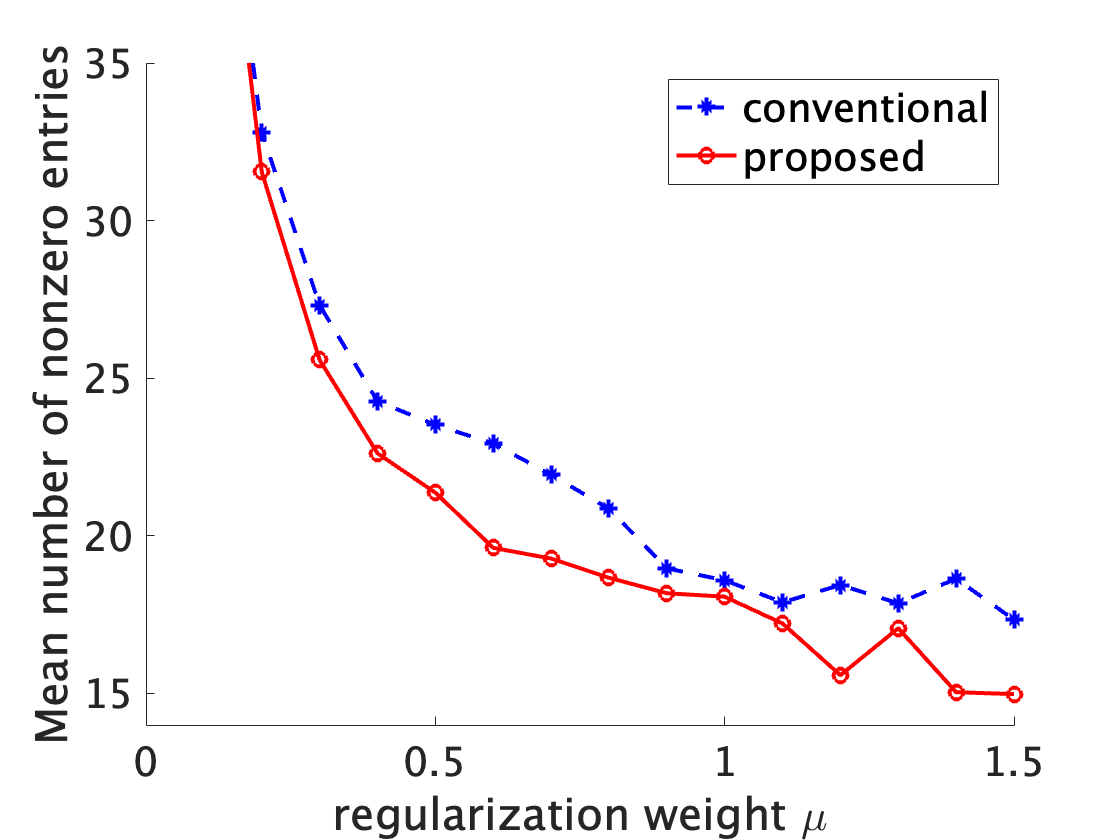}
    \subcaption{}
  \end{minipage}
  \caption{
   regularization weight $\mu$ versus (a) mean absolute error (MAE), (b) mean squared error (MSE), and (c)  mean of number of entries in estimate $\bar{x}$ such that $[D\bar{x}]_i>10^{-4}$
   (Note: $\norm{D \xstar}_0 = 5$)
  }
  \label{fig:mu-vs-MAE}
\end{figure}

Figure \ref{fig:reconst-examples-poisson}(a) shows the target signal $\xstar$ and (b-d) show reconstruction examples by the models \eqref{eq:convex-poisson} and \eqref{eq:proposed-poisson-reformulated}.

\begin{figure}[ht]
  \centering
  \begin{minipage}[b]{0.45\linewidth}
    \centering
    \includegraphics[keepaspectratio, scale=0.15]
    {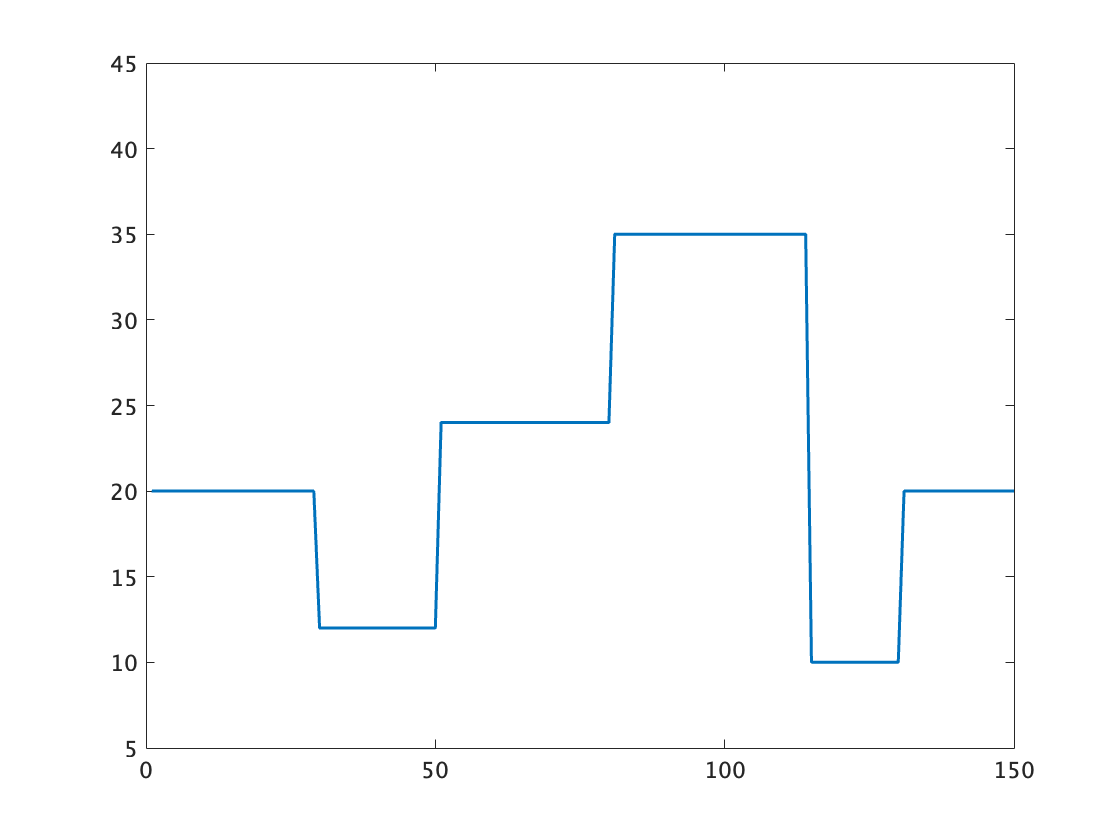}
    \subcaption{}
  \end{minipage}
  \begin{minipage}[b]{0.45\linewidth}
    \centering
    \includegraphics[keepaspectratio, scale=0.15]
    {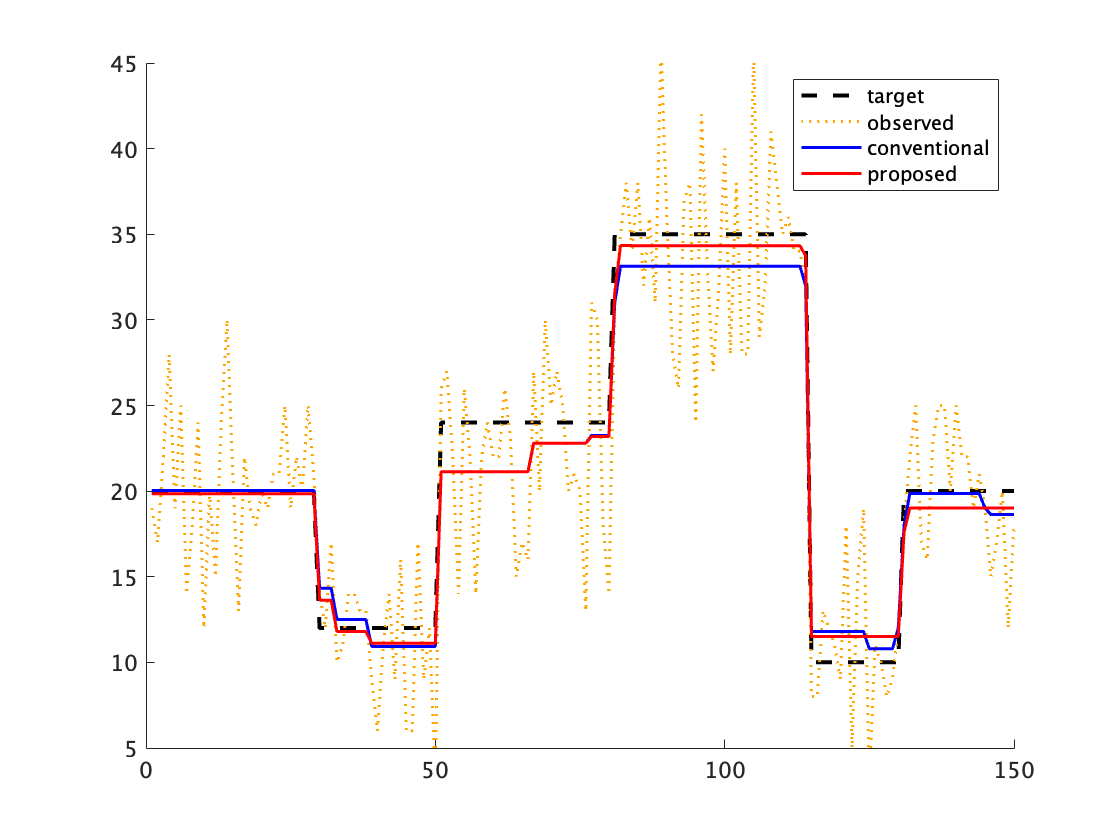}
    \subcaption{}
  \end{minipage}\\
   \begin{minipage}[b]{0.45\linewidth}
    \centering
    \includegraphics[keepaspectratio, scale=0.15]
    {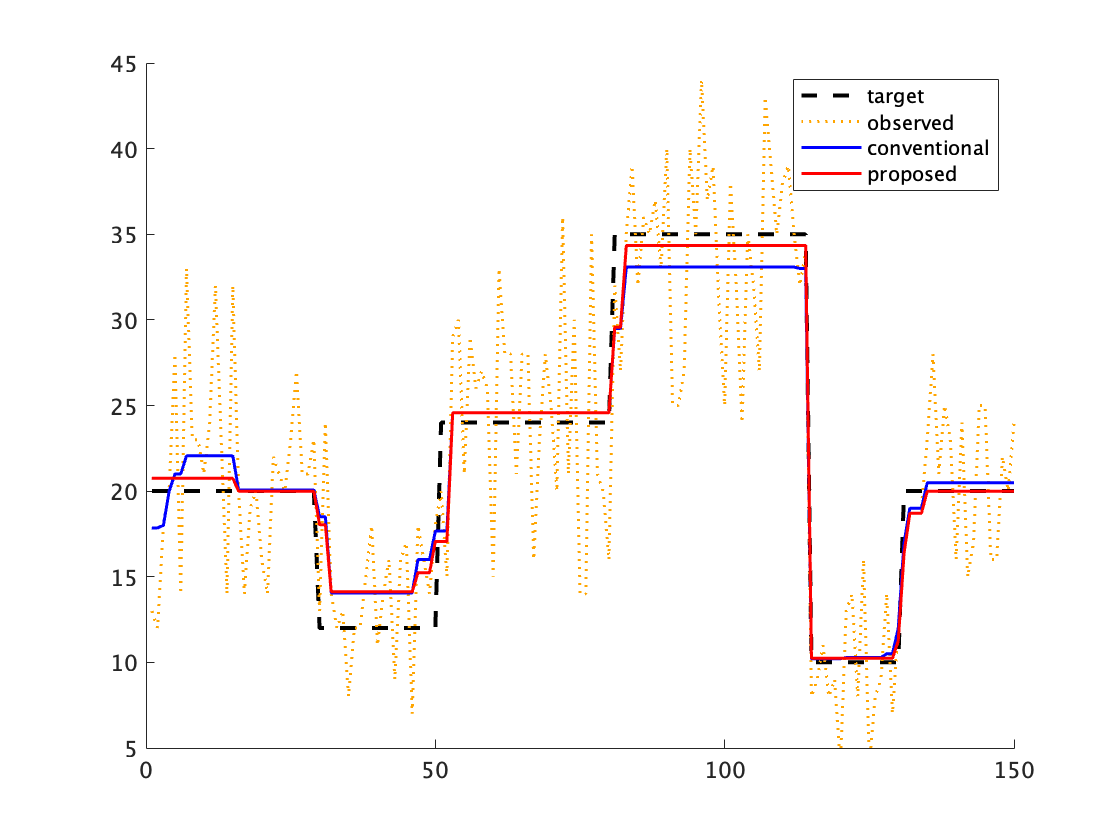}
    \subcaption{}
  \end{minipage}
  \begin{minipage}[b]{0.45\linewidth}
    \centering
    \includegraphics[keepaspectratio, scale=0.15]
    {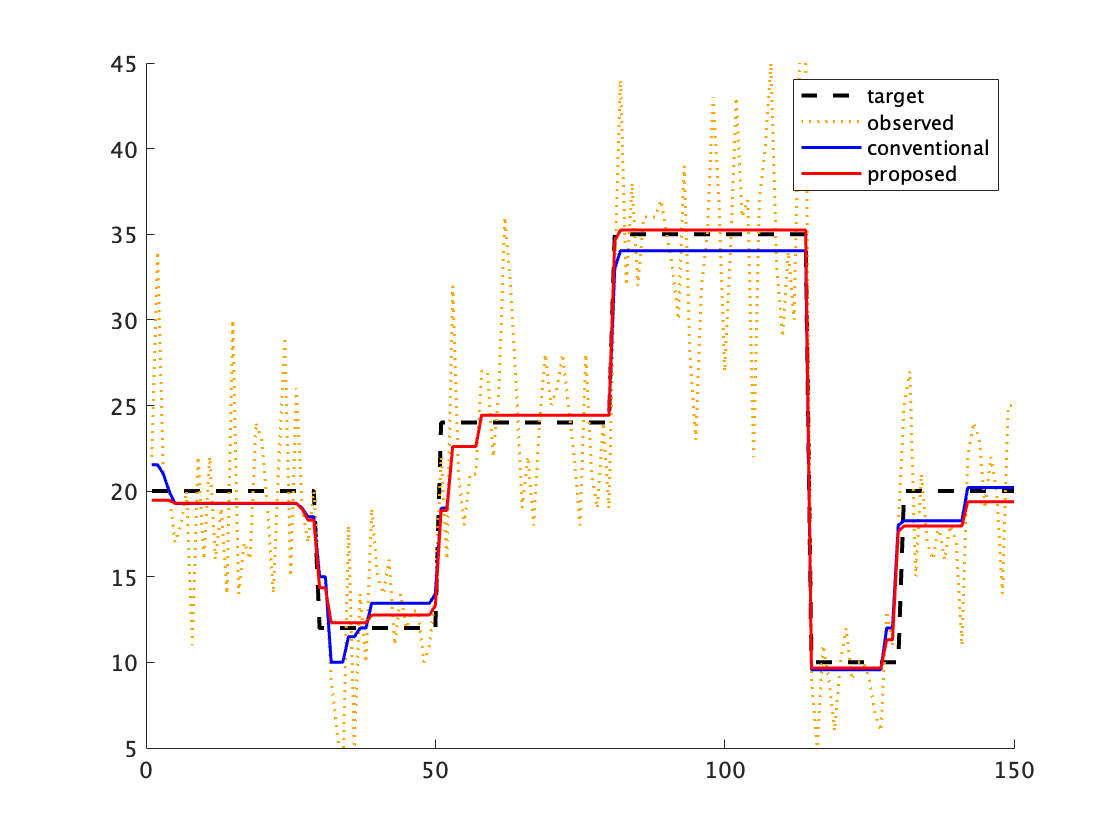}
    \subcaption{}
  \end{minipage}
  \caption{
   (a) target signal $\xstar$, 
   (b-d) reconstruction examples by \eqref{eq:convex-poisson} (blue line) and by the proposed model \eqref{eq:proposed-poisson} (red line) 
   with the target signal (black dotted) and the observed signal (yellow dotted)
  }
  \label{fig:reconst-examples-poisson}
\end{figure}

\subsection{Simultaneous declipping and Gaussian denoising}
\label{sec:declipping}
\subsubsection{Conventional convex model for simultaneous declipping and denoising}
In this section, we consider a simultaneous declipping and denoising problem.
The task is to estimate the target signal $\xstar\in\setR^n$ from
\begin{equation}
  \label{eq:clip-observation-numerical}
  y = \operatorname{clip}_{\vartheta}(A \xstar+\varepsilon) \in\setR^m,
\end{equation}
where $A\in\setR^{m\times n }$ is the observation matrix,  $\operatorname{clip}_{\vartheta}:\setR^m \to \setR^m$ is defined with $\vartheta\in\setPR$ as an entrywise operator:
\begin{equation}
  [\operatorname{clip}_{\vartheta} (u)]_i \coloneqq \begin{cases}
    [u]_i, & |[u]_i| < \vartheta\\
    \vartheta \cdot\operatorname{sign}([u]_i), & |[u]_i| \geq\vartheta
  \end{cases}
\end{equation}
for every $u\in\setR^m$ and $i\in\{1,2,\ldots,m\}$, and each noise $[\varepsilon]_i$ follows Gaussian distribution with zero mean and known variance $s^2\ (s>0)$. 
For estimation of $\xstar$, we have a priori knowledge that $\opL \xstar$ is sparse with a certain linear operator $\opL$.
Such an estimation problem arises, e.g., in image processing \cite{foi2009} and audio processing~\cite{zaviska2021}.

Recently, \cite{banerjee2024} introduced the data fidelity function $\fBeforeExtension^{\lrangle{{\mathrm{clip}}}}\circ A$ by taking the negative logarithm of the following wide-sense likelihood function \cite[(15)]{banerjee2024}:
\begin{equation}
  \label{eq:clipping-likelihood}
  \left(\prod_{i\in \mathbb{S}_{uc}} \frac{1}{s\sqrt{2\pi}} e^{-\frac{([y]_i - [Ax]_i)^2}{2s^2}}\right)
  \left(
    \prod_{i\in \mathbb{S}_{+}} 
    P([Ax]_i+\varepsilon_i \geq \vartheta)
  \right)
  \left(
    \prod_{i\in \mathbb{S}_{-}} 
    P([Ax]_i+\varepsilon_i \leq -\vartheta)
  \right),
\end{equation}
where $\mathbb{S}_{uc}\coloneqq \{i \in\{1,2,\ldots,m\}| |y_i|<\vartheta\}, \mathbb{S}_{+}\coloneqq \{i \in\{1,2,\ldots,m\}| y_i=\vartheta\}$ and $\mathbb{S}_{-}\coloneqq \{i \in\{1,2,\ldots,m\}| y_i=-\vartheta\}$ are index sets respectively of unclipped measurements, positively clipped measurements and negatively clipped measurements.
From \eqref{eq:clipping-likelihood}, the observation loss function $\fBeforeExtension^{\lrangle{{\mathrm{clip}}}}$ can be expressed as $\fBeforeExtension^{\lrangle{{\mathrm{clip}}}}:\setR^m\to\setR: u\mapsto \sum_{i=1}^m \fBeforeExtension^{\lrangle{{\mathrm{clip}}}}_i([u]_i)$ with
    \begin{equation}
      \label{eq:def-clip-i}
      \fBeforeExtension^{\lrangle{{\mathrm{clip}}}}_i(t)\coloneqq
      \begin{cases}
        \frac{1}{2}\left(\frac{[y]_i - t}{s}\right)^2, & i \in\mathbb{S}_{uc}\\
        -\log \left(\int_{\vartheta - t}^{\infty}\exp\left(
          - \frac{\alpha^2}{2s^2}
        \right) \mathrm{d} \alpha \right), &i\in\mathbb{S}_+\\
        -\log \left(\int_{-\infty}^{-\vartheta- t}\exp\left(
          - \frac{\alpha^2}{2s^2}
        \right) \mathrm{d}\alpha \right), &i\in\mathbb{S}_-.
      \end{cases}
    \end{equation}
    With a regularizer $\norm{\cdot}_1\circ\opL$, \cite{banerjee2024} formulated\footnote{The reference \cite{banerjee2024} introduced originally the model (5.10) with $(\opC, \Cz)\coloneqq (\Id, \setR^m)$.} an optimization model for the inverse problem \eqref{eq:clip-observation-numerical} as 
    \begin{equation}
      \label{eq:numerical-conventional}
      \minimize_{\opC x\in\Cz} \fBeforeExtension^{\lrangle{{\mathrm{clip}}}} \circ A(x) + \mu \norm{\cdot}_{1}\circ \opL (x).
    \end{equation}
\subsubsection{Nonconvex enhancement of convex model for simultaneous declipping and denoising and reformulation into Problem \ref{prob:cLiGME-w-general-fidelity}}
As a nonconvex enhancement of the model \eqref{eq:numerical-conventional}, we propose 
\begin{equation}
  \label{eq:numerical-proposed}
  \minimize_{\opC x\in\Cz} \fBeforeExtension^{\lrangle{{\mathrm{clip}}}} \circ A(x) + \mu (\norm{\cdot}_{1})_B\circ \opL (x).
\end{equation}
By setting\footnote{
  A similar constraint is utilized in saturation consistency signal recovery \cite{laska2011}. 
} $\varpi \in\setNNR$, $\opC\coloneqq A$ and $\Cz \coloneqq \Pi \coloneqq \bigtimes_{i=1}^m\Pi_i$ with
\begin{equation}
  \label{eq:Pi-i-clipping}
  \Pi_i \coloneqq \begin{cases}
    \setR, & i\in\mathbb{S}_{uc}\\
    [ \vartheta - \varpi , \infty),  &i \in \mathbb{S}_{+}\\
    (-\infty, -\vartheta + \varpi], & i\in \mathbb{S}_{-},
  \end{cases}
\end{equation}
the following lemma verifies that
the model \eqref{eq:numerical-proposed} is an instance of Problem \ref{prob:relaxed-fidelity}.
\begin{lemma}
  \label{lemma:monotonicity-declipping}
  Consider the model \eqref{eq:numerical-proposed}.
  Let
  $p(\cdot):\setR\to\setR$ be the probability density function of Gaussian distribution with zero mean and a standard derivation $s>0$, and $\operatorname{Pr}(\cdot):\setR\to\setR$ be its cumulative distribution function, i.e.,
  \begin{equation}
     p(t) \coloneqq \frac{1}{\sqrt{2\pi s^2}} \exp\left(-\frac{t^2}{2s^2}\right), \ \operatorname{Pr}(t) \coloneqq\int_{-\infty}^{t} p(\alpha) \mathrm{d}\alpha.
    \end{equation}
  Then, for each $i\in \{1,2,\ldots,m\}$, $\fBeforeExtension^{\lrangle{\mathrm{clip}}}_i$ in \eqref{eq:def-clip-i} is twice continuously differentiable over $\setR$ with the first- and second-order derivative
    \begin{equation}
      \label{eq:clip-gradient}
      {\fBeforeExtension^{\lrangle{{\mathrm{clip}}}}_i}'(t)\coloneqq
      \begin{cases}
        \frac{1}{s^2}\left(t-[y]_i\right),  & i\in \mathbb{S}_{uc}\\
         - \frac{p(t-\vartheta)}{\operatorname{Pr}(t-\vartheta)}, &i\in \mathbb{S}_{+}\\
         \frac{p(-\vartheta-t)}{\operatorname{Pr}(-\vartheta-t)}, &i\in \mathbb{S}_{-},
      \end{cases}
      \quad \quad
      {\fBeforeExtension^{\lrangle{{\mathrm{clip}}}}_i}''(t)\coloneqq
       \begin{cases}
        \frac{1}{s^2},  &i\in \mathbb{S}_{uc}\\
         \left(-\frac{p(\cdot)}{\operatorname{Pr}(\cdot)}\right)'(t - \vartheta), &i\in \mathbb{S}_{+}\\
         \left(-\frac{p(\cdot)}{\operatorname{Pr}(\cdot)}\right)'(-t - \vartheta), &i\in \mathbb{S}_{-},
       \end{cases}
    \end{equation}
    where 
    \begin{equation}
    \label{eq:hessian-clipping-positive}
    \left(-\frac{p(\cdot)}{\operatorname{Pr}(\cdot)}\right)'(t)=
    \frac{p(t)}{\operatorname{Pr}(t)^2}\left( p(t) + \frac{t}{s^2} \operatorname{Pr}(t)\right).
  \end{equation}
  Moreover, by letting $\varpi \in\setNNR$, $\opC\coloneqq A$ and $\Cz \coloneqq \Pi \coloneqq \bigtimes_{i=1}^m\Pi_i$ with $\Pi_i$ in \eqref{eq:Pi-i-clipping}, the model \eqref{eq:numerical-proposed} is an instance of Problem \ref{prob:relaxed-fidelity}, where
  \begin{equation}
    \label{eq:inf-hessian-clip-lemma}
    \inf_{t\in\Pi_i}{ \fBeforeExtension^{\lrangle{\mathrm{clip}}}}''(t) 
    =\begin{cases}
      \frac{1}{s^2}, & i\in \mathbb{S}_{uc}\\
      0, & i\in \mathbb{S}_{+}\cup \mathbb{S}_{-}
    \end{cases}
  \end{equation}
  and
  \begin{equation}
    \label{eq:sup-hessian-clip-lemma}
    \sup_{t\in\Pi_i}{ \fBeforeExtension^{\lrangle{\mathrm{clip}}}}''(t) 
    = \begin{cases}
      \frac{1}{s^2}, & i\in \mathbb{S}_{+}\\
      \left(-\frac{p(\cdot)}{\operatorname{Pr}(\cdot)}\right)'(-\varpi), & i\in \mathbb{S}_{+}\cup \mathbb{S}_{-}.
    \end{cases}
  \end{equation}
\end{lemma}
\begin{proof}
  See Appendix \ref{appendix:proof-clipping-hessian}
\end{proof} 

From Corollary \ref{corollary:existence-minimizer-extension}, we can reformulate the model \eqref{eq:numerical-proposed} into the following as an instance of Problem \ref{prob:cLiGME-w-general-fidelity}:
\begin{equation}
  \label{eq:numerical-proposed-reformulated}
  \minimize_{\opC x\in\Cz} f^{\lrangle{{\mathrm{clip}}}} \circ A(x) + \mu (\norm{\cdot}_{1})_B\circ \opL (x).
\end{equation}
Since $f^{\lrangle{{\mathrm{clip}}}}\circ A$ is $1$-strongly convex relative to $q_{A^*\Lambda A}$ with a diagonal matrix $\Lambda$ defined by \eqref{eq:def-lambda} with \eqref{eq:inf-hessian-clip-lemma}, the cost function in \eqref{eq:numerical-proposed-reformulated} with a GME matrix satisfying 
\begin{equation}
  \label{eq:oc-cond-clip}
  A^*\Lambda A - \mu\opL^*B^*B\opL\succeq \zeroMatrix_{\spX}
\end{equation} 
becomes convex by Corollary \ref{corollary:existence-minimizer-extension}(c). 
Moreover, if $\nullsp \opL=\{0_{\spX}\}$, the existence of a minimizer of \eqref{eq:numerical-proposed-reformulated} is guaranteed by Corollary \ref{corollary:existence-minimizer-extension}(d)(iii) because $\fBeforeExtension^{\lrangle{\mathrm{clip}}}$ is bounded below.

\subsubsection{Numerical experiments of simultaneous declipping and denoising}
Following \cite{banerjee2024}, we conducted numerical experiments on estimation of $\xstar \in \setR^n \ (n\coloneqq m\coloneqq256)$ from its noisy observation $y\in\setR^m$ in~\eqref{eq:clip-observation-numerical}, where $A\coloneqq I_n$ and $\xstar$ was given by the inverse Discrete Cosine Transform (DCT) of a randomly chosen sparse coefficient vector. The target signal $\xstar$ was normalized to satisfy $\norm{\xstar}_{\infty}=0.8$.
We compared the conventional model \eqref{eq:numerical-conventional} and the proposed model \eqref{eq:numerical-proposed-reformulated} for every $(\vartheta, s)\in\{0.4,0.6\}\times \{s_{5},s_{10},s_{15}\}$, where $s_5, s_{10}$ and $s_{15}$ are standard deviations of Gaussian noise achieving respectively $5$dB, $10$dB and $15$dB of SNR: $20\log_{10}\frac{\norm{x^\star}_{\mathbb{R}^m}}{\mathbb{E}[\norm{\varepsilon}_{\mathbb{R}^m}]}$.
For both models \eqref{eq:numerical-conventional} and \eqref{eq:numerical-proposed-reformulated}, we employed $(\opC, \Cz)$ as \eqref{eq:Pi-i-clipping} with $\varpi\coloneqq 10s$ and $\opL\coloneqq \opL_{\mathrm{DCT}}$ (DCT matrix \cite{rao2007}). We also employed a simple GME matrix $B\coloneqq \sqrt{\frac{0.99}{\mu}} \sqrt{\Lambda} \opL_{\mathrm{DCT}}^{-1}$ for the proposed model \eqref{eq:numerical-proposed-reformulated} to achieve the condition \eqref{eq:oc-cond-clip}.
For minimizations of \eqref{eq:numerical-conventional} and \eqref{eq:numerical-proposed-reformulated}, we applied
the proposed algorithm \eqref{eq:K-M}, where $(\sigma, \tau)$ is set as in Remark \ref{remark:choice-sigma-tau}. 
We stopped the proposed algorithm \eqref{eq:K-M} after the residual achieved $\norm{h_k-h_{k-1}}_{\spH}<10^{-4}$ in every experiment.

\begin{figure}[ht]
  \centering
  \begin{minipage}[b]{0.46\linewidth}
    \centering
    \includegraphics[keepaspectratio, scale=0.11]
    {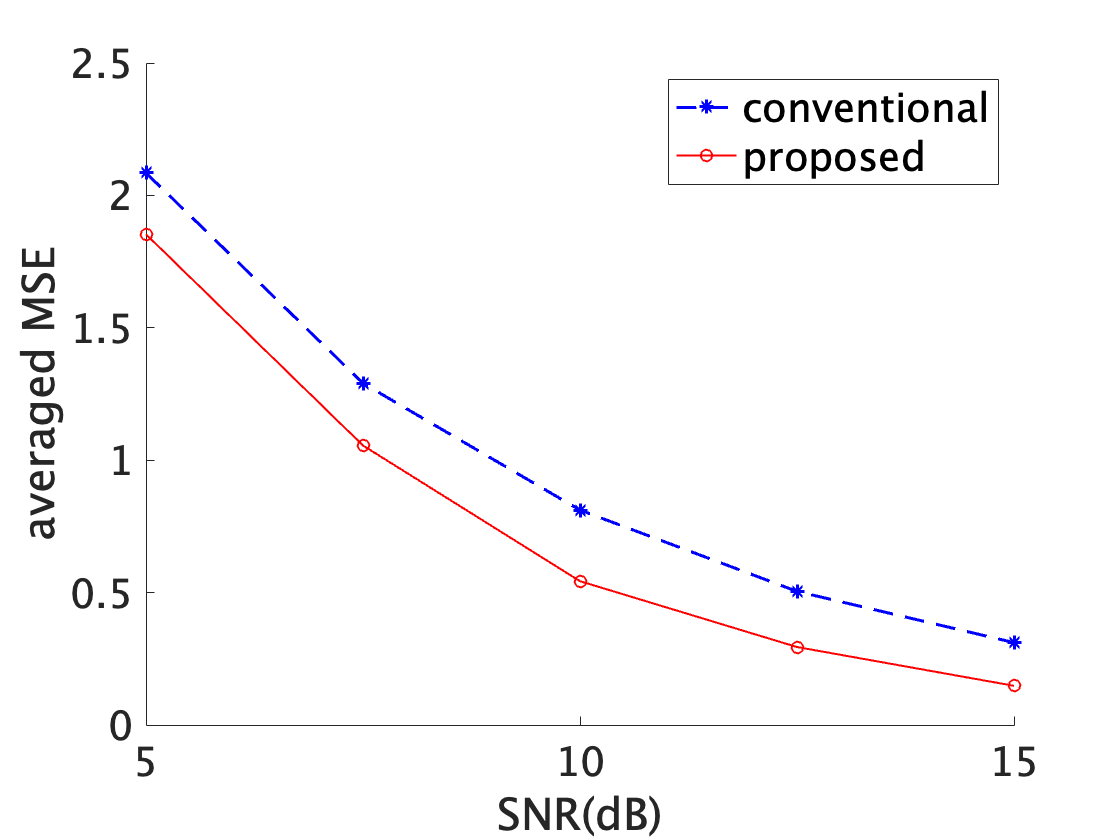}
    \subcaption{$\vartheta=0.4$}
  \end{minipage}
  \begin{minipage}[b]{0.46\linewidth}
    \centering
    \includegraphics[keepaspectratio, scale=0.11]
    {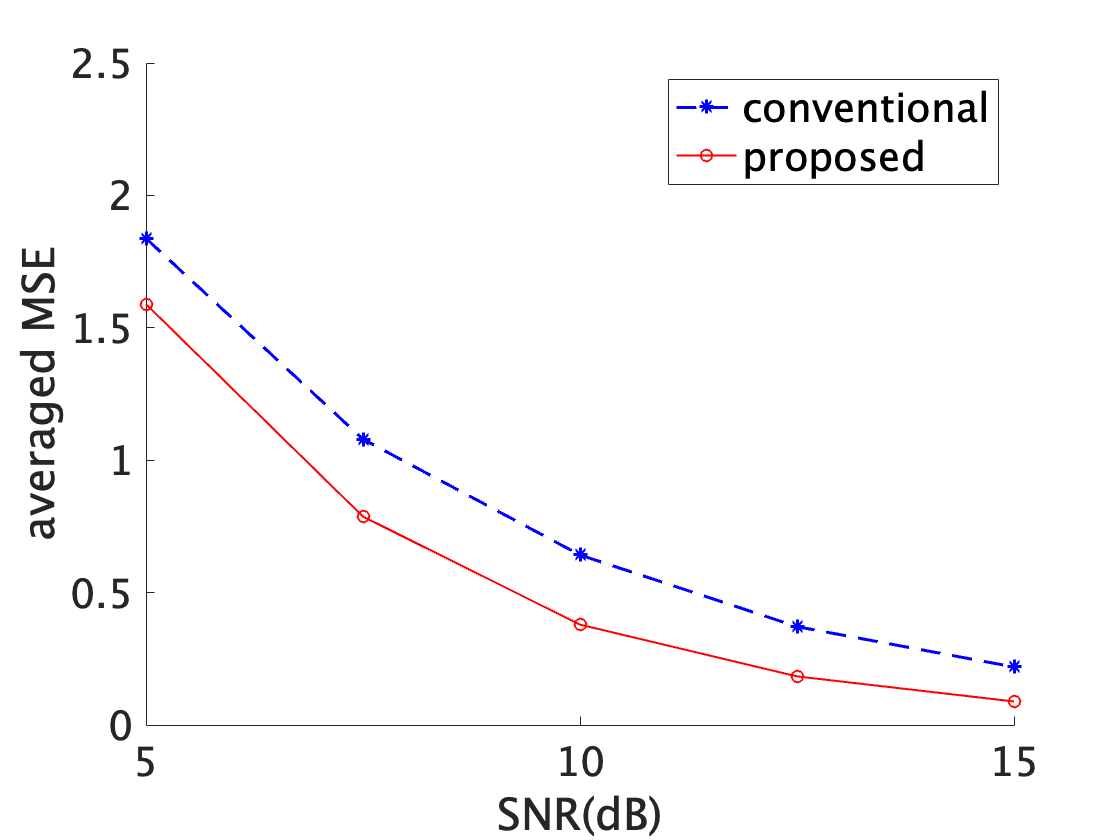}
    \subcaption{$\vartheta=0.6$}
  \end{minipage}
  \caption{
    SNR versus Averaged MSE
  }
  \label{fig:denoise-example}
\end{figure}

Figure \ref{fig:denoise-example} shows the average of MSE: $\norm{x^\star - \bar{x}}_2^2$ of estimates $\bar{x}$ by the conventional model \eqref{eq:numerical-conventional} and by the proposed model~\eqref{eq:numerical-proposed} over 100 realizations of Gaussian noise $\varepsilon$.
For each $(\vartheta, s^2)$, we choose the best regularization parameter $\mu$ from $\{j\in\setN_0\mid 1\leq j \leq 100\}$.
From this figure, we see that the proposed model~\eqref{eq:numerical-proposed} outperforms the model \eqref{eq:numerical-conventional} in all cases.

\section{Conclusion}
In this paper, we proposed a nonconvexly regularized convex model made with a smooth data fidelity function and the GME regularizer. For the proposed model under the overall convexity condition, we presented sufficient conditions for the existence of a minimizer and a proximal splitting type algorithm.
To verify the effectiveness of the proposed methods, we conducted numerical experiments in scenarios of \emph{Poisson denoising} and \emph{simultaneous denoising and declipping}.

\newcounter{appnum}
\setcounter{appnum}{1}
\setcounter{theorem}{0}

\appendix
\setcounter{lemma}{0}
\renewcommand{\thelemma}{\Alph{appnum}.\arabic{lemma}}
\setcounter{equation}{0}
\setcounter{fact}{0}
\renewcommand{\thefact}{\Alph{appnum}.\arabic{fact}}
\setcounter{equation}{0}

\renewcommand{\theequation}{\Alph{appnum}.\arabic{equation}}
\section{Additional facts}
\label{appendix:known-facts}

\subsection{Linear algebra}
\begin{fact}[Properties of partitioned matrix]
  \label{fact:properties-partitioned}
  Let $\spH_1,\spH_2$ be finite-dimensional real Hilbert spaces and set $\spH\coloneqq \spH_1\times \spH_2$.
  Let $P\in\setLO{\spH}{\spH}$ can be partitioned as
  $
    P=
    \begin{bmatrix}
      P_1 & P_2 \\
      P_2^* & P_3
    \end{bmatrix}
  $
  with $P_1\in\setLO{\spH_1}{\spH_1}$, $P_2 \in \setLO{\spH_2}{\spH_1}$ and $P_3\in\setLO{\spH_2}{\spH_2}$.
  Then the following hold.
  \begin{enumerate}[(a)]
     \item (\cite[Proposition 16.1(1)]{gallier2011}). Under the invertibility of $P_3$, $P\succ\zeroMatrix_{\spH}$ holds if and only if $P_3\succ \zeroMatrix_{\spH_2}$ and $P_1 - P_2 P_3^{-1}P_2^* \succ \zeroMatrix_{\spH_1}$ hold.
     \item (\cite[Proposition 16.2(1)]{gallier2011}). Under the invertibility of $P_1$, $P\succ\zeroMatrix_{\spH}$ holds if and only if  $P_1\succ \zeroMatrix_{\spH_1}$ and $P_3 - P_2^* P_1^{-1}P_2 \succ \zeroMatrix_{\spH_2}$ hold.
     \item (\cite[(0.7.3.1)]{horn2012}). If $P\succ\zeroMatrix_{\spH}$, $P_1\succ \zeroMatrix_{\spH_1}$ and $P_3\succ\zeroMatrix_{\spH_2}$ hold, then $P^{-1}$ can be expressed as 
    \begin{equation}
      \label{eq:partitioned-inverse}
      P^{-1}= \begin{bmatrix}
        (P_1 - P_2 P_3^{-1} P_2^*)^{-1} 
        & P_1^{-1} P_2 (P_2^* P_1^{-1} P_2 -P_3)^{-1}\\
        P_3^{-1} P_2^*(P_2P_3^{-1} P_2^* - P_1)^{-1}
        & (P_3 - P_2^* P_1^{-1} P_2)^{-1}
      \end{bmatrix},
    \end{equation}
    where $P_3- P_2^* P_1^{-1} P_2$ and $P_1 - P_2 P_3^{-1} P_2^*$ are invertible by the positive definiteness of $P\succ \zeroMatrix_{\spH}$ $P_1 \succ \zeroMatrix_{\spH_1}$ and $P_3\succ \zeroMatrix_{\spH_2}$ (see (a) and (b)).
    \item (\cite[Proposition 8.3]{halko2011}).  If $P\succeq \zeroMatrix_{\spH}$ holds, then we have $\norm{P}_{\mathrm{op}}\leq \norm{P_1}_{\mathrm{op}}+\norm{P_3}_{\mathrm{op}}$.
  \end{enumerate}
\end{fact}

\begin{fact}
  \label{fact:op-norm-inequality}
  Let $P\succeq \zeroMatrix_{\spH}$ and $\alpha \in\setPR$. If $\alpha\Id - P \succ \zeroMatrix_{\spH}$, then $\norm{(\alpha \Id - P)^{-1}}_{\mathrm{op}} \leq (\alpha - \norm{P}_{\mathrm{op}})^{-1}$ holds.
\end{fact}
\begin{proof}
  From $\alpha \Id - P\succeq (\alpha - \norm{P}_{\mathrm{op}})\Id$ by $\alpha\Id \succeq P$, $(\alpha \Id - P)^{-1}\preceq (\alpha - \norm{P}_{\mathrm{op}})^{-1}\Id$ holds. Hence, we have $\norm{(\alpha \Id - P)^{-1}}_{\mathrm{op}} \leq (\alpha - \norm{P}_{\mathrm{op}})^{-1}$.
\end{proof}

\subsection{Convex analysis, monotone operator theory and fixed point theory of nonexpansive operators}
\begin{fact}[Composition of averaged nonexpansive operators {\cite[Theorem 4(b)]{ogura2002}, \cite[Proposition 2.4]{combettes2015}}] 
  \label{fact:composition-averaged}
  Suppose that each $T_i :\spH\to \spH(i = 1,2)$ is 
  $\alpha_i$-averaged nonexpansive with $\alpha_i \in (0,1)$.
  Then $T_1\circ T_2$ is $\alpha$-averaged nonexpansive with
  $\alpha\coloneqq \frac{\alpha_1+\alpha_2 - 2\alpha_1\alpha_2}{1-\alpha_1\alpha_2}
  \in (0,1)$.
\end{fact}

\begin{fact}[Selected properties of the subdifferential]
  \label{fact:subdifferential}
  \quad
  \begin{enumerate}[(a)]
    \item (Fermat's rule \cite[Theorem 16.3]{CAaMOTiH}). For $\varphi\in\Gamma_0(\spH)$, the following relation holds:
    \begin{equation}
      \label{eq:fermat-rule}
      \bar{x}\in\argmin_{x\in\spH} \varphi(x) \iff 0\in\partial \varphi(x).
    \end{equation}
    \item (Sum rule \cite[Corollary 16.48]{CAaMOTiH}). For $\varphi_1\in\Gamma_0(\spH)$ and $\varphi_2\in\Gamma_0(\spH)$, 
    if $0_{\spH}\in\ri (\dom \varphi_1 - \dom \varphi_2)$, then 
    \begin{equation}
      \label{eq:sum-rule}
      \partial (\varphi_1 +\varphi_2) = \partial \varphi_1 + \partial \varphi_2.
    \end{equation}
    \item (Chain rule \cite[Corollary 16.53]{CAaMOTiH}). For $\psi\in\Gamma_0(\spK)$ and $L\in\setLO{\spH}{\spK}$,
    if $0_{\spK}\in \ri (\dom \psi - \ran L)$, then 
    \begin{equation}
      \label{eq:chain-rule}
      \partial (\psi\circ L) = L^* \circ (\partial \psi) \circ L.
    \end{equation}
    \item (Relation between subdifferential and conjugate \cite[Corollary 16.30]{CAaMOTiH}). For $\varphi\in\Gamma_0(\spH)$, the following relation holds:
    \begin{equation}
      \label{eq:subdifferential-and-conjugate}
      (\forall (x,u)\in\spH\times\spH)\quad
      u\in\partial \varphi(x)\iff x\in\partial \varphi^*(u).
    \end{equation}
    \item (Relation between subdifferential and proximity operator \cite[Example 23.3]{CAaMOTiH}).
    For $\varphi\in\Gamma_0(\spH)$, the following holds: 
    \begin{equation}
      \label{eq:subdifferential-and-prox}
      \prox_{\varphi} = (\Id +\partial\varphi)^{-1}.
    \end{equation}
  \end{enumerate}
\end{fact}

\begin{fact}[{\cite[Lemma 1.2.2]{nesterov2013}}]
  \label{fact:hessian-and-lipconst}
  Let $\varphi\in\Gamma_0(\spH)$ be twice continuously differentiable over $\spH$.
  Then $\nabla \varphi$ is $\lipconst{\nabla \varphi}$-Lipschitz continuous over $\spH$ with $\lipconst{\nabla \varphi}\in\setPR$ if and only if $\sup_{x\in\spH}\norm{\nabla^2\varphi(x)}_{\mathrm{op}} \leq \lipconst{\nabla \varphi}$.
\end{fact}

\begin{fact}[Properties of generalized Hessian]
  \label{fact:generalized-hessian}
  Let $\varphi:\spH\to\setR$ be differentiable over $\spH$ with Lipschitz continuous gradient $\nabla \varphi$. Let $\mathfrak{D}$ be the set of all points where $\varphi$ is twice differentiable, which is a dense subset of $\spH$ (see \cite[Theorem 9.60]{rockafellar2009variational}). 
  Define the \textit{generalized Hessian} of $\varphi$ at $x\in\spH$ in the sense of \cite{aros2021} by
\begin{equation}
  \label{eq:def-g-hess}
  \overline{\nabla}^2 \varphi(x) \coloneqq \{M\in\setLO{\spH}{\spH} \mid \exists (x_n)_{n\in\setN_0}\subset \mathfrak{D} \mbox{ s.t. } \lim_{n\to +\infty} x_n = x \mbox{ and } \lim_{n\to +\infty} \nabla^2 \varphi (x_n) = M \}.
\end{equation}
  Then the following hold.
  \quad
  \begin{enumerate}[(a)]
    \item (\cite[Thorem 13.52]{rockafellar2009variational}). For every $x\in\spH$, $\overline{\nabla}^2 \varphi(x)$ is a nonempty compact set of symmetric matrices. 
    \item (\cite[Proposition 2.2(i)]{aros2021}). $\varphi$ is convex if and only if, for every $x\in\spH$ and $M\in\overline{\nabla}^2\varphi(x)$, $M\succeq \zeroMatrix_{\spH}$ holds.
    \item For $\alpha\in\setR$ and $x\in\spH$, 
    \begin{equation}
      \label{eq:gene-hessian-scalar}
      \overline{\nabla}^2 (\alpha\varphi)(x)
      = \alpha \overline{\nabla}^2 (\varphi)(x).
    \end{equation}
    \item Let $\psi:\spH\to\setR$ be twice continuously differentiable. Then, for $x\in\spH$,
    \begin{equation}
      \label{eq:g-hess-sum-rule}
      \overline{\nabla}^2(\varphi+\psi)
      (x) = \overline{\nabla}^2\varphi(x) + \nabla^2\psi(x)
      (=\{M+\nabla^2\psi(x)\in\setLO{\spH}{\spH} \mid M\in  \overline{\nabla}^2\varphi(x)\}).
    \end{equation}
    \item $\nabla\varphi$ is $\lipconst{\nabla\varphi}$-Lipschitz continuous if and only if the following inequality holds:
    \begin{equation}
      (\forall x\in\spH)(\forall M\in\overline{\nabla}^2\varphi(x)) \quad \norm{M}_{\mathrm{op}}\leq \lipconst{\nabla\varphi}.
    \end{equation}
  \end{enumerate}
\end{fact}
\begin{proof}
  We provide a simple proof of (c-e)
  below for self-containedness. Fix $x\in\spH$ arbitrarily.

  \noindent
  \textbf{(Proof of (c))} If $ \alpha =0$, \eqref{eq:gene-hessian-scalar} holds clearly.
  Assume $\alpha\neq0$. Then \eqref{eq:gene-hessian-scalar} is verified by
  \begin{equation}
    M\in \alpha \overline{\nabla}^2(\varphi)(x)
    \iff "(\exists (x_n)_{n\in\setN_0}\subset \mathfrak{D})\ \lim_{n\to +\infty} x_n =x \mbox{ and } \lim_{n\to \infty}\nabla^2(\alpha\varphi)(x_n)= M"
    \iff M \in\overline{\nabla}^2(\alpha\varphi)(x).
  \end{equation}

  \noindent
  \textbf{(Proof of (d))}
  Since $\mathfrak{D}$ is also the set of all points where
  $\varphi + \psi$ is twice differentiable, the equality \eqref{eq:g-hess-sum-rule} follows from
  \begin{align}
    \label{eq:proof-g-hess-sum-1}
    \hspace{-1em} %
    M \in \overline{\nabla}^2(\varphi+\psi)(x)
    \iff  &\exists (x_n)_{n\in\setN_0}\subset \mathfrak{D} \mbox{ s.t. } \begin{cases}
    \lim_{n\to +\infty} x_n = x\\
    \lim_{n\to +\infty}\nabla^2 (\varphi+\psi) (x_n) = M
    \end{cases}\\
     \label{eq:proof-g-hess-sum-2}
     \iff &\exists (x_n)_{n\in\setN_0}\subset \mathfrak{D} \mbox{ s.t. } \begin{cases}
    \lim_{n\to +\infty} x_n = x\\
    \lim_{n\to +\infty}\nabla^2\varphi (x_n) + \nabla^2\psi (x) = M
    \end{cases}\\
    \label{eq:proof-g-hess-sum-3}
    \iff &  M -\nabla^2\psi(x) \in \overline{\nabla}^2(\varphi)(x) 
    \iff  M  \in \overline{\nabla}^2(\varphi)(x)+\nabla^2\psi(x),
  \end{align}
  where we used the twice continuous differentiability of $\psi$ to derive \eqref{eq:proof-g-hess-sum-2}.

  \noindent
  \textbf{(Proof of (e))} Verified by combining (c) and \cite[Proposition 2.2(ii)]{aros2021}.
\end{proof}

\setcounter{appnum}{2}
\setcounter{lemma}{0}
\setcounter{fact}{0}
\section{Proof of Proposition \ref{prop:overall-convexity}}
\label{sec:a1}

\noindent
\textbf{(Proof of (a))} 
From the definition of $\Psi_B$ in \eqref{eq:GME}, we observe, for every $x\in\spX$,
\begin{align}
  \Psi_B\circ \opL(x) &= \Psi\circ \opL(x) - \min_{v\in\spZ} \left[
    \Psi(v) + \frac{1}{2}\norm{B(\opL x - v)}^2_{\sptildeZ}
  \right]\\
  &=\Psi\circ \opL(x) - \min_{v\in\spZ} \left[
    \Psi(v) +  \frac{1}{2}\norm{B\opL x}^2_{\sptildeZ}-\ip{B^*B\opL x}{v}_{\spZ} + \frac{1}{2}\norm{B v}^2_{\sptildeZ}
  \right]\\
  &=\Psi\circ \opL (x)-  \frac{1}{2}\norm{B\opL x}^2_{\sptildeZ} - \min_{v\in\spZ} \left[
    \Psi(v) - \ip{B^*B\opL x}{v}_{\spZ} + \frac{1}{2}\norm{B v}^2_{\sptildeZ}
  \right]\\
  &= \Psi\circ \opL (x) -  \frac{1}{2}\norm{B\opL x}^2_{\sptildeZ} + \max_{v\in\spZ} \left[ \ip{B^*B\opL x}{v}_{\spZ}
  - \Psi(v) - \frac{1}{2}\norm{B v}^2_{\sptildeZ}
\right]\\
&=   \Psi\circ\opL (x) - \frac{1}{2}\norm{B\opL x}^2_{\sptildeZ}
+\left(\Psi + \frac{1}{2}\norm{B \cdot}^2_{\sptildeZ}\right)^*\circ B^*B\opL (x),
\end{align}
which verifies the expression in \eqref{eq:LiGME-reg-decompose}. 
For the third term in \eqref{eq:LiGME-reg-decompose},
$\left(\Psi + \frac{1}{2}\norm{B \cdot}^2_{\sptildeZ}\right)^*\in\Gamma_0(\spZ)$ is verified by \cite[Corollary 13.38]{CAaMOTiH} with $\Psi + \frac{1}{2}\norm{B \cdot}^2_{\sptildeZ}\in\Gamma_0(\spZ)$.
Moreover, for every $x\in\spX$, we have
\begin{align}
  &\left(\Psi + \frac{1}{2}\norm{B \cdot}^2_{\sptildeZ}\right)^*\circ B^*B\opL(x)
  = \sup_{v\in\spZ}\left[
  \ip{B^*B\opL(x)}{v}_{\spZ}- \Psi(v) - \frac{1}{2}\norm{Bv}^2_{\sptildeZ}\right]\\
  &\leq \sup_{v_1\in\spZ} [-\Psi(v_1)] + \sup_{v_2\in\spZ}\left[
  \ip{B\opL(x)}{B v_2}_{\spZ} - \frac{1}{2}\norm{Bv_2}^2_{\sptildeZ}\right]\\
  &= -\inf_{v_1\in\spZ} [ \Psi(v_1)] - \inf_{\widetilde{v}_2\in\ran B }\left[
  -\ip{B\opL(x)}{\widetilde{v}_2}_{\sptildeZ} + \frac{1}{2}\norm{\widetilde{v}_2}^2_{\sptildeZ}\right] <\infty,
\end{align}
where the last inequality follows from Fact \ref{fact:existence-minimizer}(a) with the coercivities of $\Psi\in\Gamma_0(\spZ)$ and $\left(-\ip{B\opL(x)}{\cdot}_{\sptildeZ} + \frac{1}{2}\norm{\cdot}^2_{\sptildeZ}\right)\in\Gamma_0(\sptildeZ)$.
Therefore, \eqref{eq:dom-of-conjugate} holds, 
and $\left(\Psi + \frac{1}{2}\norm{B \cdot}^2_{\sptildeZ}\right)^*\circ B^*B\opL \in\Gamma_0(\spX)$ is verified by \cite[Proposition 9.5]{CAaMOTiH} with $\left(\Psi + \frac{1}{2}\norm{B \cdot}^2_{\sptildeZ}\right)^*\in\Gamma_0(\spZ)$ and \eqref{eq:dom-of-conjugate}.
Combining \eqref{eq:LiGME-reg-decompose}, \eqref{eq:dom-of-conjugate} and $\dom (-\frac{1}{2}\norm{B\opL (\cdot)}^2_{\sptildeZ})=\spX$, we have \eqref{eq:dom-LiGME-reg}. Since every function in RHS of \eqref{eq:LiGME-reg-decompose} is lower semicontinuous, $\Psi_B\circ \opL$ is also lower semicontinuous by \cite[Lemma 1.27]{CAaMOTiH}.
Moreover, $\Psi_B\circ \opL$ is $1$-weakly convex relative to $q_{\opL^*B^*B\opL}$ because \eqref{eq:LiGME-reg-decompose} implies that $
  \Psi_B\circ \opL + q_{\opL^*B^*B\opL}
  = \Psi\circ\opL
  +\left(\Psi + \frac{1}{2}\norm{B \cdot}^2_{\sptildeZ}\right)^*\circ B^*B\opL
  \in\Gamma_0(\spX)
$ is convex.

{
\sloppy \noindent
\textbf{(Proof of (b))} 
From \eqref{eq:cLiGME-w-general-fidelity}, \eqref{eq:dom-LiGME-reg} and $\dom f = \spY$, we have 
\begin{equation}
  \dom (J_{\Psi_B\circ \opL}) = \dom(f\circ A)\cap\dom(\Psi_B\circ \opL)=\dom(\Psi_B\circ \opL)= \dom(\Psi\circ \opL).
\end{equation}
Moreover, since every term in RHS of \eqref{eq:cLiGME-w-general-fidelity} is lower semicontinuous, the lower semicontinuity of $J_{\Psi_B \circ \opL}$ is verified by \cite[Lemma 1.27]{CAaMOTiH}.

\noindent
\textbf{(Proof of (c))} 
$(C_0)\implies(C_1)$: Under the condition $(C_0)$, $f\circ A$ is $1$-strongly convex relative to $q_{A^*\Lambda A}$ by Proposition \ref{prop:properties-strong-convexity}(b). Therefore,
$(C_1)$ is verified by Proposition \ref{prop:properties-strong-convexity}(a).
$(C_1)\implies(C_2)$: $f\circ A$ is $\mu$-strongly convex relative to $q_{\opL^*B^*B\opL}$ from $(C_1)$, and $\mu\Psi_B\circ \opL$ is $\mu$-weakly convex relative to $q_{\opL^*B^*B\opL}$ by Proposition \ref{prop:overall-convexity}(a).
Therefore, $J_{\Psi_B\circ \opL}\in\Gamma_0(\spX)$ is verified by Proposition \ref{prop:properties-strong-convexity}(a) with $\mu\opL^*B^*B\opL - \mu\opL^*B^*B\opL = \zeroMatrix_{\spX}$.

\noindent
\textbf{(Proof of (d))} 
The equality \eqref{eq:LiGME-reg-decompose} yields the expression \eqref{eq:J-sum-of-convex}, where $\mathfrak{d}\in\Gamma_0(\spX)$ is assumed in $(C_1)$.
From $\Psi\in\Gamma_0(\spZ)$, $\mu\Psi\circ\opL\in\Gamma_0(\spX)$ is verified by \cite[Proposition 9.5]{CAaMOTiH} and $\dom(\Psi\circ\opL)\neq \emptyset$ in the assumption~\eqref{eq:assumption-proper}.
Finally, $\mu\left(\Psi + \frac{1}{2}\norm{B \cdot}^2_{\sptildeZ}\right)^*\circ B^*B\opL\in\Gamma_0(\spX)$ is verified by (a).
}

\setcounter{appnum}{3}
\setcounter{lemma}{0}
\setcounter{fact}{0}
\section{Proof of Proposition \ref{prop:coercivity-J}}
\label{appendix:coercivity-J}

We use the following lemmas for proof.
\begin{lemma}
  \label{prop:null-BL-cond}
  In Problem \ref{prob:cLiGME-w-general-fidelity}, the  following relation holds:
  \begin{equation}
    \label{eq:BL-null-relation}
    (C_1):\mathfrak{d}\overset{\eqref{eq:def-d}}{\coloneqq} f\circ A - \frac{\mu}{2}\norm{B\opL \cdot}_{\sptildeZ}^2\in\Gamma_0(\spX)\implies
    \dom(\rec(f\circ A)) \subset \nullsp(B\opL).
  \end{equation}
\end{lemma}
\begin{proof}
The proof is by contradiction.
Assume that $(C_1)$ in \eqref{eq:convexity-cond} holds and there exists $x\notin\nullsp(B\opL)$ such that $\rec(f\circ A)(x)<\infty$, i.e., $\rec(f\circ A)(x)\in \setR$.
By \cite[Proposition 9.30(i)]{CAaMOTiH} with $\mathfrak{d}\in\Gamma_0(\spX)$, we have $\rec(\mathfrak{d})\in\Gamma_0(\spX)$.
By $\rec(f\circ A)(x)\in \setR$, we also have for every $u\in\dom(\mathfrak{d})=\dom(f\circ A)= \spX$
\begin{align}
  \label{eq:rec-dx-contradiction}
  \rec (\mathfrak{d})(x)&\overset{\eqref{eq:properties-recession}}{=}
  \lim_{\alpha\to +\infty}\frac{\mathfrak{d}(u + \alpha x)}{\alpha}
    \\
    &\overset{\hphantom{\eqref{eq:properties-recession}}}{=} \lim_{\alpha\to +\infty}
    \frac{\left(f\circ A -\frac{\mu}{2}\norm{B\opL\cdot}^2_{\sptildeZ}\right)(u + \alpha x)}{\alpha}\\
    &\overset{\hphantom{\eqref{eq:properties-recession}}}{=} \lim_{\alpha\to +\infty}
    \left(
      \frac{f\circ A (u+\alpha x)}{\alpha}
      -\frac{\mu}{2\alpha}\norm{B\opL(u + \alpha x)}^2_{\sptildeZ}
    \right)\\
    \label{eq:proof-rec-dx-cont}
    &\overset{\eqref{eq:properties-recession}}{=}\rec (f\circ A)(x) - \lim_{\alpha\to +\infty}\frac{\mu}{2\alpha}\norm{B\opL(u + \alpha x)}^2_{\sptildeZ}\\
    &\overset{\hphantom{\eqref{eq:properties-recession}}}{=} \rec (f\circ A)(x) - \lim_{\alpha\to +\infty}\left[\frac{\mu}{2 \alpha} \norm{B\opL u}^2_{\sptildeZ} + \mu\ip{B\opL u}{B\opL x}_{\sptildeZ}
    + \frac{\mu \alpha}{2} \norm{B\opL x}_{\sptildeZ}^2
    \right]\\
    \label{eq:contradiction-neg-inf}
    &\overset{\hphantom{\eqref{eq:properties-recession}}}{=} -\infty,
\end{align}
where we used $B\opL x\neq 0_{\sptildeZ}$ for \eqref{eq:contradiction-neg-inf}.
However, \eqref{eq:contradiction-neg-inf} contradicts $\rec(\mathfrak{d})\in\Gamma_0(\spX)$.
Thus, we have the implication in \eqref{eq:BL-null-relation} by contradiction.
\end{proof}

\begin{lemma}
  Consider Problem \ref{prob:cLiGME-w-general-fidelity} under $(C_1)$ in \eqref{eq:convexity-cond} and $\inf_{x\in\spX} f\circ A (x) >-\infty$.
  Then 
  \begin{equation}
    \label{eq:rec-J-geq-fA}
    (\forall x\in\spX)\quad  
    \rec (J_{\Psi_B\circ \opL})(x)\geq \rec (f\circ A)(x) \geq 0
  \end{equation}
  holds.
  In particular, if $\rec (f\circ A)(x)=0$, then we have
  \begin{equation}
    \label{eq:rec-J-eq-PsiL}
    \rec (J_{\Psi_B\circ \opL})(x) = \mu \rec (\Psi\circ \opL)(x).
  \end{equation}
\end{lemma}
\begin{proof}
  Since $\dom (\Psi\circ \opL) = \dom(\Psi_B\circ \opL)\neq \emptyset$ by \eqref{eq:dom-LiGME-reg} and
  \begin{align}
    \Psi_B \circ \opL(x)
    &=\Psi \circ \opL(x)
    - \min_{v\in\spZ} \left[
      \Psi(v) + \frac{1}{2}\norm{B(\opL x - v)}^2_{\sptildeZ}
    \right]
    \geq\Psi \circ \opL(x)
    - \left[
      \Psi(\opL x) + \frac{1}{2}\norm{B(\opL x - \opL x)}^2_{\sptildeZ}
    \right]\\
    &= \Psi \circ \opL(x) - \Psi \circ \opL(x) =0
  \end{align}
  holds for every $ x\in\dom (\Psi\circ \opL) $, we have 
$
  (\forall x\in\spX)\ J_{\Psi_B\circ\opL}(x)= f\circ A(x) + \mu \Psi_B\circ \opL (x) \geq f\circ A(x)
$.
Thus, with $u\in\dom (J_{\Psi_B\circ \opL})$, \eqref{eq:rec-J-geq-fA} is verified by
\begin{equation}
  \rec (J_{\Psi_B\circ\opL})(x) \overset{\eqref{eq:properties-recession}}{=}\lim_{\alpha\to +\infty} \frac{J_{\Psi_B\circ\opL}(u+\alpha x)}{\alpha}
  \geq \lim_{\alpha\to +\infty} \frac{f\circ A(u+\alpha x)}{\alpha}
  \label{eq:rec-f-A-geq-0}
  \overset{\eqref{eq:properties-recession}}{=}\rec (f\circ A)(x) \geq 0,
\end{equation}
where the last inequality is verified by \cite[Proposition 9.30(v)]{CAaMOTiH} with the assumption $\inf_{x\in\spX} f\circ A(x) >-\infty$.

Next, assume that $\rec (f\circ A)(x) =0$. Then, from \eqref{eq:BL-null-relation}, we obtain $x\in\nullsp(B\opL)$. With $u\in \dom(J_{\Psi_B\circ\opL})$, we observe
\begin{align}
  \operatorname{rec} (J_{\Psi_B\circ\opL})(x)
  \overset{\eqref{eq:properties-recession}}{=}& \lim_{\alpha \to +\infty}\frac{
  J_{\Psi_B\circ\opL}(u+\alpha x)}{\alpha}
  = \lim_{\alpha \to +\infty}\left[
    \frac{f\circ A (u+\alpha x)}{\alpha}
    + \mu\frac{\Psi_B\circ\opL(u+\alpha x)}{\alpha}
  \right]\\
  \label{eq:by-rec-properties-3}
  \overset{\hphantom{\eqref{eq:properties-recession}}}{=}&  0+ \mu\lim_{\alpha \to +\infty}
    \frac{\Psi_B\circ\opL(u+\alpha x)}{\alpha}\\
  \overset{\hphantom{\eqref{eq:properties-recession}}}{=}&\mu\lim_{\alpha \to +\infty} \left[\frac{\Psi\circ\opL(u+\alpha x)}{\alpha}
  - \frac{\min\limits_{v\in\spZ} \left[\Psi(v)
  +\frac{1}{2}\norm{B(\opL (u+\alpha x)- v)}^2_{\sptildeZ}\right]}{\alpha}\right]\\
  \label{eq:by-nullBL-relation}
  \overset{\hphantom{\eqref{eq:properties-recession}}}{=}&\mu\lim_{\alpha \to +\infty} \left[\frac{\Psi\circ\opL(u+\alpha x)}{\alpha}
  - \frac{\min\limits_{v\in\spZ} \left[\Psi(v)
  +\frac{1}{2}\norm{B(\opL u- v)}^2_{\sptildeZ}\right]}{\alpha}\right]\\
  \overset{\hphantom{\eqref{eq:properties-recession}}}{=} &\mu\lim_{\alpha \to +\infty}\frac{\Psi\circ\opL(u+\alpha x)}{\alpha}
  - 0
  =\mu\lim_{\alpha \to +\infty}\frac{\Psi(\opL u+\alpha \opL x)}{\alpha}
  \overset{\eqref{eq:properties-recession}}{=}\mu(\operatorname{rec} \Psi)(\opL x),
\end{align}
where the equality in \eqref{eq:by-rec-properties-3} is verified by \eqref{eq:properties-recession} as  
$
  \lim_{\alpha \to +\infty} \frac{f\circ A (u+\alpha x)}{\alpha}
  =\rec(f\circ A)(x) = 0,
$
and the equality in \eqref{eq:by-nullBL-relation} is verified by $x\in\nullsp (B\opL)$.
\end{proof}

Here, we show Proposition \ref{prop:coercivity-J} below.

\noindent
\textbf{(Proof of (a))} 
\eqref{eq:rec-J-geq-f-A-geq-0} has already been proven in \eqref{eq:rec-J-geq-fA}.
From \eqref{eq:rec-J-geq-fA}, $(\impliedby)$ in \eqref{eq:implication-rec-positive} is verified
by contraposition of ($\implies$) in \eqref{eq:implication-rec-0}, and $(\impliedby)$ in \eqref{eq:implication-rec-0} is verified by contraposition of ($\implies$) in \eqref{eq:implication-rec-positive}.
Hence, it suffices to show ($\implies$) in \eqref{eq:implication-rec-positive} and \eqref{eq:implication-rec-0}.

First, we prove $(\implies)$ in \eqref{eq:implication-rec-positive}
by showing $\rec (J_{\Psi_B\circ\opL})(x) >0$ under the following two cases: (I) $x\in\spX$ satisfies $\rec(f\circ A)(x)>0$ and (II) $x\in\spX$ satisfies $\rec(f\circ A)(x) = 0$ and $x\notin\nullsp\opL$.
If $\rec(f\circ A)(x)>0$ holds, we obtain $\rec (J_{\Psi_B\circ\opL})(x)\geq \rec(f\circ A)(x)>0$ from \eqref{eq:rec-J-geq-fA}.
On the other hand, if $\rec(f\circ A)(x) = 0$ and $x\notin\nullsp\opL$,
from \eqref{eq:rec-J-eq-PsiL}, we obtain $ \rec (J_{\Psi_B\circ\opL})(x) = \mu \rec (\Psi\circ \opL)(x) = \mu \rec (\Psi)(\opL x)>0$, where the last inequality is verified by the coercivity of $\Psi$ and $\opL x\neq 0_{\spZ}$ (see \eqref{eq:def-coercive} for the definition of coercivity).

Next, we prove $(\implies)$ in \eqref{eq:implication-rec-0}. 
Assume that $\rec(f\circ A)(x) = 0 \mbox{ and } x\in\nullsp \opL $.
From \eqref{eq:rec-J-eq-PsiL}, we obtain 
\begin{equation}
  \rec (J_{\Psi_B\circ\opL}) (x) = \mu\rec (\Psi)(\opL x) \overset{\opL x = 0_{\spZ}}{=}\mu\rec (\Psi)(0_{\spZ}) = 0.
\end{equation}

\noindent
\textbf{(Proof of (b))} 
Assume that $f$ is coercive. Then $f$ is bounded below, and thus
$
(\forall x\in\spX)\ \rec (J_{\Psi_B\circ\opL}) (x)\geq 0
$ holds from (a).
Moreover, by $\rec (f\circ A) = \rec(f)\circ A$ and the definition of the coercivity, we have
\begin{equation}
  \label{eq:rec-0-and-null-A}
  \rec(f)(A x) = 0 \iff x\in\nullsp A.
\end{equation}
Combining \eqref{eq:implication-rec-0} and \eqref{eq:rec-0-and-null-A},
we obtain 
\begin{equation}
  x\in\nullsp A \cap\nullsp \opL
  \overset{\eqref{eq:rec-0-and-null-A}}{\iff}
  "\rec(f\circ A)(x) = 0 \mbox{ and } x\in\nullsp \opL"
  \overset{\eqref{eq:implication-rec-0}}{\iff} \rec (J_{\Psi_B\circ\opL})(x) = 0.
\end{equation}

\noindent
\textbf{(Proof of (c))}
Assume that $f$ is coercive.
If $\rec (J_{\Psi_B\circ\opL}) (x)= 0$, \eqref{eq:rec-J-0-and-null} yields
\begin{align}
  (\forall u \in( \dom J_{\Psi_B\circ\opL}))(\forall \alpha\in\setR)\quad J_{\Psi_B\circ\opL}(u+\alpha x)
  = &f\circ A(u+\alpha x)+ \mu\Psi_B\circ\opL (u+\alpha x)\\
  =&f\circ A(u) + \mu\Psi_B\circ\opL(u)\\
  \label{eq:J-const-on-recession-dir}
  =&J_{\Psi_B\circ\opL}(u).
\end{align}
By combining \eqref{eq:rec-J-positive-under-coercivity-of-f} from (b) with \eqref{eq:J-const-on-recession-dir}, $J_{\Psi_B\circ \opL}$ is weakly coercive.

\noindent
\textbf{(Proof of (d))} 
First, assume the condition (i).
Since $f$ is coercive, \eqref{eq:rec-J-positive-under-coercivity-of-f} and \eqref{eq:rec-J-0-and-null} hold from (b).
Therefore,  the coercivity of $J_{\Psi_B\circ \opL}$ is verified by 
\begin{equation}
  x \notin (\nullsp A\cap \nullsp \opL)=\{0_{\spX}\}
  \iff \rec (J_{\Psi_B\circ\opL}) (x) >0.
\end{equation}

Next, assume the condition (ii). By \eqref{eq:implication-rec-positive}, we have the implication
$
  x\notin \nullsp\opL \implies \rec (J_{\Psi_B\circ\opL})(x)>0$.
Then, by combining this implication and $\nullsp \opL =\{0_{\spX}\}$, we obtain $(x\in\spX\setminus \{0_{\spX}\})\ \rec (J_{\Psi_B\circ\opL})(x)>0$.%

\setcounter{appnum}{4}
\setcounter{lemma}{0}
\setcounter{fact}{0}
\section{Proof of Lemma \ref{lemma:coercive-constant}}
\label{appendix:a5}
\noindent
\textbf{(Proof of (a))} 
Since $f\circ A $ is differentiable and $\nabla(f\circ A)$ is Lipschitz continuous over $\spY$, $\mathfrak{d}= f\circ A - \frac{\mu}{2}\norm{B\opL\cdot}^2_{\sptildeZ}$ is also continuously differentiable with Lipschitz continuous gradient.
From Fact~\ref{fact:generalized-hessian}(e), it suffices to show  
\begin{equation}
  (\forall x \in\spX) (\forall M \in \overline{\nabla}^2\mathfrak{d}(x))\quad \norm{M}_{\mathrm{op}} \leq \lipconst{\nabla(f\circ A)}.
\end{equation}
Choose $x\in\spX$ and $M \in \overline{\nabla}^2\mathfrak{d}(x)$ arbitrarily.
Since $-\frac{\mu}{2}\norm{B\opL (\cdot)}^2_{\sptildeZ}$ is twice continuously differentiable over $\spX$, Fact \ref{fact:generalized-hessian}(d) yields 
\begin{equation}
  \label{eq:hess-d-expression}
  \overline{\nabla}^2\mathfrak{d}(x) = \overline{\nabla}^2\left(f\circ A -\frac{\mu}{2}\norm{B\opL (\cdot)}^2_{\sptildeZ}\right)(x)
  = \overline{\nabla}^2(f\circ A)(x) - \mu\opL^*B^*B\opL,
\end{equation} 
which implies
$M + \mu\opL^*B^*B\opL \in \overline{\nabla}^2(f\circ A)(x)$.
From Fact \ref{fact:generalized-hessian}(e), $\norm{M + \mu\opL^*B^*B\opL}_{\mathrm{op}}\leq \lipconst{\nabla(f\circ A)}$ holds. This inequality yields $M + \mu\opL^*B^*B\opL \preceq \lipconst{\nabla(f\circ A)}\Id$. 
Moreover, Fact \ref{fact:generalized-hessian}(b) with $M\in \overline{\nabla}^2\mathfrak{d}(x)$ and $\mathfrak{d}\in\Gamma_0(\spX)$ ensures
$M\succeq \zeroMatrix_{\spX}$.
Combining the last two inequalities and $\mu\opL^*B^*B\opL\succeq \zeroMatrix_{\spX}$, we have 
\begin{equation}
  \lipconst{\nabla(f\circ A)}\Id \succeq M + \mu\opL^*B^*B\opL\succeq M \succeq \zeroMatrix_{\spX},
\end{equation}
which implies $\norm{M}_{\mathrm{op}} \leq \lipconst{\nabla(f\circ A)}$.

\noindent
\textbf{(Proof of (b))}
Since $\mathfrak{d}$ is convex and $\nabla{\mathfrak{d}}$ is $\lipconst{\nabla \mathfrak{d}}$-Lipschitz continuous, Baillon-Haddad theorem (see, e.g., \cite[Corollary 18.17]{CAaMOTiH}) ensures that
  $\nabla \mathfrak{d}$ is $\frac{1}{\lipconst{\nabla \mathfrak{d}}}$-cocoercive, i.e.,
  \begin{equation}
    (\forall (x_1,x_2)\in\spX\times\spX)\ \ip{x_1-x_2}{\nabla\mathfrak{d}(x_1) - \nabla\mathfrak{d}(x_2)}_{\spX}\geq \frac{1}{\lipconst{\nabla \mathfrak{d}}}\norm{\nabla\mathfrak{d}(x_1) - \nabla\mathfrak{d}(x_2)}^2_{\spX}.
  \end{equation}
  On the other hand, we have
  $
    (\forall v\in\spZ)\ \norm{\mu B^*B v}_{\spZ}^2\leq \mu^2 \norm{B}^2_{\mathrm{op}}\norm{Bv}^2_{\sptildeZ}.
  $
  Therefore, for every $(x_1,v_1)\in\spX\times \spZ$ and  $(x_2,v_2)\in\spX\times \spZ$, we obtain
  \begin{align}
    & \quad \norm{
    \begin{pmatrix}
        \nabla \mathfrak{d}(x_1)\\
        \mu B^*B v_1
    \end{pmatrix}
    - \begin{pmatrix}
      \nabla \mathfrak{d}(x_2)\\
      \mu B^*B v_2
  \end{pmatrix}
    }^2_{\spX\times\spZ}\\
    &= \norm{\nabla  \mathfrak{d} (x_1) - \nabla  \mathfrak{d}(x_2)}^2_{\spX}
    + \norm{\mu B^*B (v_1 - v_2)}^2_{\spZ}\\
    &\leq \lipconst{\nabla \mathfrak{d}} \ip{x_1 - x_2}{\nabla  \mathfrak{d}(x_1) - \nabla  \mathfrak{d}(x_2)}_{\spX} + \mu^2 \norm{B}^2_{\mathrm{op}}\norm{B(v_1 - v_2)}^2_{\sptildeZ}\\
    &= \lipconst{\nabla \mathfrak{d}} \ip{x_1 - x_2}{\nabla  \mathfrak{d} (x_1) - \nabla  \mathfrak{d}(x_2)}_{\spX} 
    + \mu  \norm{B}^2_{\mathrm{op}} \ip{v_1-v_2}{\mu B^*B(v_1-v_2)}_{\spZ}\\
    &\leq \max\{ \lipconst{\nabla \mathfrak{d}}, \mu  \norm{B}^2_{\mathrm{op}}\} ( \ip{x_1 - x_2}{\nabla  \mathfrak{d} (x_1) - \nabla  \mathfrak{d}(x_2)}_{\spX}  +  \ip{v_1-v_2}{\mu B^*B(v_1-v_2)}_{\spZ})\\
    &= \frac{1}{\rho} \ip{
      \begin{pmatrix} x_1  \\ 
        v_1
      \end{pmatrix}
      -
      \begin{pmatrix} x_2  \\ 
        v_2
      \end{pmatrix}}
      {
      \begin{pmatrix}
        \nabla  \mathfrak{d} (x_1) \\
        \mu B^*B v_1
      \end{pmatrix}
      -
      \begin{pmatrix}
        \nabla  \mathfrak{d} (x_2) \\
        \mu B^*B v_2
      \end{pmatrix}}_{\spX\times \spZ},
  \end{align}
  which implies that a mapping $\spX\times\spZ\to\spX\times\spZ:(x,v)\mapsto(\nabla\mathfrak{d}(x), \mu B^*B v)$ is $\rho$-cocoercive over $(\spX\times \spZ, \ip{\cdot}{\cdot}_{\spX\times \spZ},\norm{\cdot}_{\spX\times \spZ})$.

\setcounter{appnum}{5}
\setcounter{lemma}{0}
\setcounter{fact}{0}
\section{Proof of Theorem \ref{thm:convergence-analysis}}
\label{appendix:a6}
\subsection{Preliminary lemmas}
We use the following inclusion for proof of convergence analysis.
\begin{lemma}
  \label{lemma:relative-interior}
  In Problem \ref{prob:cLiGME-w-general-fidelity}, the following  inclusion holds:
  \begin{align}
    \label{eq:ri-of-conjugate}
    0_\spZ\in \ri\left[
      \dom \left(\left(
        \Psi + \frac{1}{2}\norm{B\cdot}^2_{\sptildeZ}\right)^*
      \right) - \ran (B^* B \opL) \right].
  \end{align}
\end{lemma}
\begin{proof}
  The proof is similar to the proof of \cite[Lemma 2]{shoji2025} while $\dom \Psi=\spZ$ is assumed in \cite[Lemma 2]{shoji2025}.
  For completeness, we give a proof. 
  Since $\Psi$ is coercive and $\frac{1}{2}\norm{B\cdot}_{\widetilde{\spZ}}^2$ is bounded below with $\dom\left(\frac{1}{2}\norm{B\cdot}_{\widetilde{\spZ}}^2\right)=\spZ$, $\Psi + \frac{1}{2}\norm{B\cdot}_{\widetilde{\spZ}}^2$ is coercive by \cite[Corollary 11.16]{CAaMOTiH}.
Then \cite[Proposition 14.16]{CAaMOTiH} yields $0_{\spZ}\in\operatorname{int} \left(\dom \left( \Psi+ \frac{1}{2}\norm{B\cdot}_{\widetilde{\spZ}}^2\right)^*\right)$.
By the relation (see, e.g., \cite[(6.11)]{CAaMOTiH})
$
  \ri \left(\dom \left( \Psi + \frac{1}{2}\norm{B\cdot}_{\widetilde{\spZ}}^2\right)^*\right)\supset \operatorname{int} \left(\dom \left( \Psi + \frac{1}{2}\norm{B\cdot}_{\widetilde{\spZ}}^2\right)^*\right)
  \ni 0_{\spZ}
$
and $\ri \left(\ran (B^* B\opL)\right) =\ran (B^*B\opL) \ni 0_{\spZ}$,
we have 
\begin{align}
  0_{\spZ}\in  \ri \left(\dom \left( \Psi + \frac{1}{2}\norm{B\cdot}_{\widetilde{\spZ}}^2\right)^* \right) \cap \ri(\ran (B^* B\opL)),
\end{align}
which further yields \eqref{eq:ri-of-conjugate} by \cite[Proposition 6.19(viii)]{CAaMOTiH}.
\end{proof}

We also use the following inequalities which can be verified essentially by facts of Schur complement (see also Fact \ref{fact:properties-partitioned}).
\begin{lemma}
    \label{lemma:matrix-inequality}
    In Problem \ref{prob:cLiGME-w-general-fidelity},
    let
    $(\sigma, \tau)\in(0,\infty)\times (\frac{1}{2\rho},\infty)$ satisfy \eqref{eq:cond-sigma-tau}.
    Then the following hold.
    \begin{enumerate}[\upshape (a)]
    \item We have the inequality
    \begin{equation}
      \label{eq:positiveness-of-Schur-P-by-norm}
      \sigma - \frac{\mu^2}{\tau}\norm{B^*B\opL}^2_{\mathrm{op}} -\mu\norm{\opL^*\opL +\opC^*\opC}_{\mathrm{op}}>0,
    \end{equation}
    and thus the following holds:
    \begin{equation}
      \label{eq:inequality-for-Schur}
      \sigma \Id 
      - \frac{\mu^2}{\tau}\opL^*(B^*B)^2\opL
      - \mu \opL^*\opL -\mu \opC^*\opC \succ \zeroMatrix_{\spX}.
    \end{equation}
    \item $\opP\succ \zeroMatrix_{\spH}$
    holds, where $\opP$ is defined in \eqref{eq:def-P}.
    \item Let $\opP$ be partitioned as 
    $
      \opP = \begin{bmatrix}
        \opP_1 & \opP_2\\
        \opP_2^* & \opP_3
      \end{bmatrix},
    $
    where
    \begin{equation}
      \opP_1 \coloneqq\begin{bmatrix}
        \sigma \Id & -\mu\opL^*B^*B \\
        -\mu B^* B \opL & \tau \Id
      \end{bmatrix}, \ 
      \opP_2 \coloneqq \begin{bmatrix}
        -\mu \opL^*  & -\mu\opC^*\\
        \zeroMatrix_{\spZ, \spZ} & \zeroMatrix_{\spfrakZ, \spZ}
      \end{bmatrix},
      \opP_3 \coloneqq \begin{bmatrix}
        \mu \Id & \zeroMatrix_{\spfrakZ,\spZ}\\
        \zeroMatrix_{\spZ, \spfrakZ} & \mu \Id
      \end{bmatrix}.
    \end{equation}
    Then 
    $\opP^{-1}$ can be expressed as 
    \begin{equation}
      \label{eq:block-exp-P}
    \opP^{-1} = \begin{bmatrix}
      \mathfrak{Q}^{-1} & \opP_1^{-1}\opP_2(\opP_2^* \opP_1^{-1} \opP_2 - \opP_3)^{-1}\\
      \opP_3^{-1} \opP_2^* (\opP_2\opP_3^{-1} \opP_2^* - \opP_1)^{-1} & (\opP_3 - \opP_2^* \opP_1^{-1} \opP_2)^{-1}
    \end{bmatrix}
  \end{equation}
  where 
  \begin{equation}
    \label{eq:def-Q}
    \mathfrak{Q}\coloneqq \opP_1 - \opP_2\opP_3^{-1}\opP_2^* = \begin{bmatrix}
    \sigma \Id - \mu(\opL^* \opL + \opC^*\opC) & -\mu\opL^*B^*B \\
    -\mu B^* B \opL & \tau \Id
  \end{bmatrix} \succ\zeroMatrix_{\spX\times\spZ}.
  \end{equation}
  \item 
  Set $\mathfrak{R}_1\in\setLO{\spX}{\spX}$, $\mathfrak{R}_2\in\setLO{\spZ}{\spX}$ and
  $\mathfrak{R}_3\in\setLO{\spZ}{\spZ}$ to satisfy
  \begin{equation}
    \label{eq:Q-inv-partition}
    \mathfrak{Q}^{-1}\coloneqq\begin{bmatrix}
      \mathfrak{R}_1 & \mathfrak{R}_2\\
      \mathfrak{R}_2^* & \mathfrak{R}_3
      \end{bmatrix}\succ \zeroMatrix_{\spX\times\spZ}.
  \end{equation}
  Then we have the expressions $\mathfrak{R}_1 =  (\sigma \Id - \frac{\mu^2}{\tau}\opL^*(B^*B)^2\opL - \mu(\opL^*\opL +\opC^*\opC) )^{-1}$ and
$\mathfrak{R}_3 = ( \tau \Id - \mu^2 B^* B \opL (\sigma \Id - \mu(\opL^*\opL +\opC^*\opC))^{-1}\opL^*B^*B)^{-1}$, where the following inequalities hold:
\begin{align}
   \label{eq:bound-R1}
  \norm{\mathfrak{R}_1}_{\mathrm{op}}
  &\leq \left( \sigma - \frac{\mu^2}{\tau}\norm{B^*B\opL}_{\mathrm{op}}^2 - \mu\norm{\opL^*\opL +\opC^*\opC}_{\mathrm{op}} \right)^{-1},\\
   \label{eq:bound-R3}
  \norm{\mathfrak{R}_3}_{\mathrm{op}} 
  &\leq \left(
    \tau - \mu^2 \norm{ B^* B \opL}_{\mathrm{op}}^2 (\sigma - \mu\norm{\opL^*\opL +\opC^*\opC}_{\mathrm{op}})^{-1}
  \right)^{-1}.
\end{align}
\end{enumerate}
\end{lemma}
\begin{proof}
\textbf{(Proof of (a))}
\eqref{eq:positiveness-of-Schur-P-by-norm} and \eqref{eq:inequality-for-Schur} are verified by the conditions \eqref{eq:cond-sigma-tau} and $\tau > \frac{1}{2\rho}$ as
\begin{align}
  0 &< \sigma - \mu\norm{\opL^*\opL +\opC^*\opC}_{\mathrm{op}}-\frac{2\rho \mu^2\norm{B^*B\opL}^2_{\mathrm{op}} + \tau}{2\rho\tau - 1}
  <\sigma - \mu\norm{\opL^*\opL +\opC^*\opC}_{\mathrm{op}} -\frac{2\rho \mu^2\norm{B^*B\opL}^2_{\mathrm{op}}}{2\rho\tau}\\
  &= \sigma - \frac{\mu^2}{\tau}\norm{B^*B\opL}^2_{\mathrm{op}} -\mu\norm{\opL^*\opL +\opC^*\opC}_{\mathrm{op}}
  \leq \sigma  -\norm{\frac{\mu^2}{\tau}\opL^*(B^*B)^2\opL +  \mu(\opL^*\opL +\opC^*\opC)}_{\mathrm{op}},
\end{align}
where the last inequality verifies \eqref{eq:inequality-for-Schur}.

\noindent
\textbf{(Proof of (b))} Fact \ref{fact:properties-partitioned}(a) with
$
  \begin{bmatrix}
    \tau \Id & \zeroMatrix_{\spZ, \spZ} & \zeroMatrix_{\spfrakZ,\spZ}\\
       \zeroMatrix_{\spZ,\spZ} & \mu \Id & \zeroMatrix_{\spfrakZ,\spZ}\\
       \zeroMatrix_{\spZ,\spfrakZ} & \zeroMatrix_{\spZ,\spfrakZ} & \mu\Id
  \end{bmatrix}\succ \zeroMatrix_{\spZ\times\spZ\times\spfrakZ}
$ yields
{
\begin{align}
  \opP\succ \zeroMatrix_{\spH}
  \iff &\sigma \Id
  - \begin{bmatrix}
    -\mu\opL^*B^*B & -\mu \opL^* & -\mu\opC^*
  \end{bmatrix}
  \begin{bmatrix}
    \tau \Id & \zeroMatrix_{\spZ, \spZ} & \zeroMatrix_{\spfrakZ,\spZ}\\
       \zeroMatrix_{\spZ,\spZ} & \mu \Id & \zeroMatrix_{\spfrakZ,\spZ}\\
       \zeroMatrix_{\spZ,\spfrakZ} & \zeroMatrix_{\spZ,\spfrakZ} & \mu\Id
  \end{bmatrix}^{-1}
  \begin{bmatrix}
    -\mu B^*B\opL \\ -\mu \opL \\ -\mu\opC
  \end{bmatrix}\succ \zeroMatrix_{\spX}\\
  \iff &\sigma \Id 
  - \frac{\mu^2}{\tau}\opL^*(B^*B)^2\opL
  - \mu \opL^*\opL -\mu \opC^*\opC \succ \zeroMatrix_{\spX}
\end{align}
}%
Therefore, by \eqref{eq:inequality-for-Schur}, we have $\opP\succ\zeroMatrix_{\spH}$.

\noindent
\textbf{(Proof of (c))} 
From \eqref{eq:inequality-for-Schur}, we have 
$\sigma \Id - \frac{\mu^2}{\tau} \opL^*(B^*B)^2\opL\succ\zeroMatrix_{\spX}$.
Then applying Fact \ref{fact:properties-partitioned}(a) to $\opP_1$ with $\tau \Id\succ \zeroMatrix_{\spX}$ and $\sigma \Id - \frac{\mu^2}{\tau} \opL^*(B^*B)^2\opL\succ\zeroMatrix_{\spX}$, we get 
$\opP_1\succ \zeroMatrix_{\spX\times\spZ}$.
By $\opP\succ\zeroMatrix_{\spH}$, $\opP_1\succ\zeroMatrix_{\spX\times\spZ}$ and $\opP_3\succ\zeroMatrix_{\spZ\times\spfrakZ}$, Fact~\ref{fact:properties-partitioned}(c) presents an expression of $\opP^{-1}$ as in \eqref{eq:block-exp-P}. 
Moreover, by applying Fact \ref{fact:properties-partitioned}(a) to $\opP$ with $\opP\succ\zeroMatrix_{\spH}$ by (b) and with $\opP_3\succ \zeroMatrix_{\spZ\times\spfrakZ}$, we also get $\mathfrak{Q}=\opP_1 - \opP_2\opP_3^{-1}\opP_2^*\succ\zeroMatrix_{\spX\times\spZ}$.

\noindent 
{ \sloppy
\textbf{(Proof of (d))} By applying Fact \ref{fact:properties-partitioned}(c) to $\mathfrak{Q}$ in \eqref{eq:def-Q} with 
$
  \sigma \Id - \mu(\opL^* \opL + \opC^*\opC)\succ\zeroMatrix_{\spX}
$ from \eqref{eq:inequality-for-Schur} and with $\tau\Id\succ\zeroMatrix_{\spZ}$,
we have the expressions 
\begin{align}
\mathfrak{R}_1 &=  (\sigma \Id - \frac{\mu^2}{\tau}\opL^*(B^*B)^2\opL - \mu(\opL^*\opL +\opC^*\opC) )^{-1},\\
\mathfrak{R}_3 &= ( \tau \Id - \mu^2 B^* B \opL (\sigma \Id - \mu(\opL^*\opL +\opC^*\opC))^{-1}\opL^*B^*B)^{-1}.
\end{align}
From \eqref{eq:inequality-for-Schur}, we have $\mathfrak{R}_1\succ\zeroMatrix_{\spX}$.
Then, by applying Fact \ref{fact:properties-partitioned}(b) to $\mathfrak{Q}^{-1}\succ\zeroMatrix_{\spX\times\spZ}$ with $\mathfrak{R}_1\succ\zeroMatrix_{\spX}$, we obtain $\mathfrak{R}_3 - \mathfrak{R}_2^*\mathfrak{R}_1^{-1}\mathfrak{R}_2\succ\zeroMatrix_{\spZ}$, which further implies $\mathfrak{R}_3\succ\zeroMatrix_{\spZ}$.
With $\mathfrak{R}_1\succ\zeroMatrix_{\spX}$ and $\mathfrak{R}_3\succ\zeroMatrix_{\spZ}$, we observe
\begin{align}
  &\norm{\mathfrak{R}_1}_{\mathrm{op}}
  = \norm{\left(\sigma \Id - \frac{\mu^2}{\tau}\opL^*(B^*B)^2\opL - \mu(\opL^*\opL +\opC^*\opC)\right)^{-1}}_{\mathrm{op}}\\
  &\overset{\mbox{Fact \ref{fact:op-norm-inequality}}}{\leq} \left( \sigma - \norm{\frac{\mu^2}{\tau}\opL^*(B^*B)^2\opL + \mu(\opL^*\opL +\opC^*\opC)}_{\mathrm{op}} \right)^{-1}
  \leq \left( \sigma - \frac{\mu^2}{\tau}\norm{B^*B\opL}_{\mathrm{op}}^2 - \mu\norm{\opL^*\opL +\opC^*\opC}_{\mathrm{op}} \right)^{-1},\\
  &\norm{\mathfrak{R}_3}_{\mathrm{op}} \overset{\hphantom{\mbox{Fact \ref{fact:op-norm-inequality}}}}{=}\norm{( \tau \Id - \mu^2 B^* B \opL (\sigma \Id - \mu(\opL^*\opL +\opC^*\opC))^{-1}\opL^*B^*B)^{-1}}_{\mathrm{op}}
  \\
  &\hphantom{\norm{\mathfrak{R}_3}_{\mathrm{op}}} \overset{\mbox{Fact \ref{fact:op-norm-inequality}}}{\leq} \left(
    \tau - \norm{\mu^2 B^* B \opL (\sigma \Id - \mu(\opL^*\opL + \opC^*\opC) )^{-1}\opL^*B^*B}_{\mathrm{op}}
  \right)^{-1}
  \\
  &\hphantom{\norm{\mathfrak{R}_3}_{\mathrm{op}}}  \overset{\hphantom{\mbox{Fact \ref{fact:op-norm-inequality}}}}{\leq}\left(
    \tau - \mu^2 \norm{ B^* B \opL}_{\mathrm{op}}^2 \norm{(\sigma \Id - \mu(\opL^*\opL + \opC^*\opC) )^{-1}}_{\mathrm{op}}
  \right)^{-1}
  \\
  \label{eq:ineq-R3}
  &\hphantom{\norm{\mathfrak{R}_3}_{\mathrm{op}}}  \overset{\hphantom{\mbox{Fact \ref{fact:op-norm-inequality}}}}{\leq} \left(
    \tau - \mu^2 \norm{ B^* B \opL}_{\mathrm{op}}^2 (\sigma - \mu\norm{\opL^*\opL +\opC^*\opC}_{\mathrm{op}})^{-1}
  \right)^{-1},
\end{align}
where the inequality on \eqref{eq:ineq-R3} follows from 
\begin{equation}
  \tau - \mu^2 \norm{ B^* B \opL}_{\mathrm{op}}^2 \norm{(\sigma \Id - \mu(\opL^*\opL + \opC^*\opC) )^{-1}}_{\mathrm{op}}
  \overset{}{\geq}
  \tau - \mu^2 \norm{ B^* B \opL}_{\mathrm{op}}^2 (\sigma - \mu\norm{\opL^*\opL + \opC^*\opC }_{\mathrm{op}})^{-1} \label{eq:inequality-for-bound-R3}
  \overset{\eqref{eq:positiveness-of-Schur-P-by-norm}}{>}0
\end{equation}
which is
verified by Fact \ref{fact:op-norm-inequality} with $\sigma \Id - \mu(\opL^*\opL + \opC^*\opC)\succ \zeroMatrix_{\spX}$ from \eqref{eq:inequality-for-Schur}.
}
\end{proof}

\subsection{Proof of Theorem \ref{thm:convergence-analysis}}
The proof is inspired by Condat's primal dual splitting algorithm \cite{condat2013}.

\noindent
\textbf{(Proof of (a))} Note that the solution set $\setS$ of \eqref{eq:cLiGME-w-general-fidelity} is also the solution set of the unconstrained optimization problem: 
\begin{equation}
  \label{eq:cLiGME-unconst}
  \minimize_{x\in\spX}
  J_{\Psi_B\circ\opL}(x) + \mu\iota_{\Cz}\circ\opC(x),
\end{equation}
which is nonempty by Assumption \ref{assumption:alg}(iii).
Applying Fermat's rule \eqref{eq:fermat-rule} to the convex function $J_{\Psi_B\circ\opL} + \mu\iota_{\Cz}\circ\opC$ in \eqref{eq:cLiGME-unconst}, we have 
\begin{equation}
  \label{eq:fermat-rule-J}
  x^\diamond \in \setS\iff 
  0_{\spX}\in \partial (J_{\Psi_B\circ\opL} + \mu\iota_{\Cz}\circ\opC) (x^\diamond)
   \left(= \partial J_{\Psi_B\circ\opL}(\xdiamond) + \mu \opC^*\circ \partial \iota_{\Cz}\circ\opC(\xdiamond)\right),
\end{equation}
where the equality follows from \eqref{eq:sum-rule-J-and-constraint} and \eqref{eq:chain-rule-of-indicator}.
In the following, we calculate the subdifferential of the RHS of \eqref{eq:J-sum-of-convex}.
Since $\mathfrak{d}$ is differentiable over $\spX$, we have 
\begin{equation}
  \label{eq:subdifferential-of-J}
  \partial J_{\Psi_B\circ\opL} = \nabla\mathfrak{d}+ \partial \left(\mu\Psi \circ\opL + \mu \left(\Psi + \frac{1}{2}\norm{B \cdot}^2_{\sptildeZ}\right)^*\circ B^*B\opL\right).
\end{equation}
The sum rule \eqref{eq:sum-rule} with \eqref{eq:dom-of-conjugate} yields
{
\begin{align}
  \label{eq:sum-rule-of-Psi-and-GME}
  \partial \left(\mu\Psi \circ\opL + \mu \left(\Psi + \frac{1}{2}\norm{B \cdot}^2_{\sptildeZ}\right)^*\circ B^*B\opL\right)
  = \mu\partial (\Psi \circ\opL) 
  + \mu \partial \left( \left(\Psi + \frac{1}{2}\norm{B \cdot}^2_{\sptildeZ}\right)^*\circ B^*B\opL\right).
\end{align}
}%
Then, by the chain rule \eqref{eq:chain-rule} with \eqref{eq:ri-of-conjugate}, we have 
\begin{align}
  \label{eq:subdifferential-of-conjugate}
  \partial \left( \left(\Psi + \frac{1}{2}\norm{B \cdot}^2_{\sptildeZ}\right)^*\circ B^*B\opL\right)
  = (\opL^* B^* B) \circ \partial \left(\Psi + \frac{1}{2}\norm{B \cdot}^2_{\sptildeZ}\right)^* \circ (B^*B\opL).
\end{align}
Combining \eqref{eq:chain-rule-of-regularizer}, \eqref{eq:subdifferential-of-J}, \eqref{eq:sum-rule-of-Psi-and-GME} and \eqref{eq:subdifferential-of-conjugate}, we obtain
{
\begin{equation}
  \label{eq:subdifferential-J-decomposed}
  \partial J_{\Psi_B\circ\opL} = \nabla\mathfrak{d} + \mu \opL^* \circ \partial\Psi \circ \opL + \mu(\opL^* B^* B) \circ \partial \left(\Psi + \frac{1}{2}\norm{B \cdot}^2_{\sptildeZ}\right)^* \circ (B^*B\opL)
\end{equation}
}%
Furthermore, by the relation in \eqref{eq:subdifferential-and-conjugate} between subdifferential and conjugate, we have
{
\thickmuskip=0\thickmuskip
\medmuskip=0\medmuskip
\thinmuskip=0.0\thinmuskip
\arraycolsep=0.5\arraycolsep
\begin{equation}
  \label{eq:subdifferential-and-conjugate-J}
  \left.
    \begin{aligned}
    w^\diamond \in \partial \Psi (\opL x^\diamond) \iff
  &\opL x^\diamond \in \partial \Psi^* (w^\diamond),\\
  v^\diamond \in
  \partial \left(\Psi + \frac{1}{2} \norm{B \cdot}^2_{\sptildeZ}\right)^* (B^*B\opL x^\diamond)\iff & B^*B\opL x^\diamond \in \partial \left(\Psi + \frac{1}{2} \norm{B \cdot}^2_{\sptildeZ}\right) (v^\diamond)    
  \ \left(\overset{(\clubsuit)}{=} \ \partial \Psi (v^\diamond) +B^* B v^\diamond\right),\\
  z^\diamond \in\partial \iota_{\Cz} (\opC x^\diamond) \iff &
  \opC x^\diamond \in \partial \iota_{\Cz}^*(z^{\diamond}).
    \end{aligned}
    \right\}
\end{equation}
where the equality in $(\clubsuit)$ follows from the sum rule \eqref{eq:sum-rule} with $\dom\left(\frac{1}{2}\norm{B\cdot}^2_{\sptildeZ}\right)= \spZ$.
}%
Therefore, we have
\begin{align}
  x^\diamond \in\setS
  \overset{\eqref{eq:fermat-rule-J}}{\iff}& 0_{\spX} \in\partial  J_{\Psi_B\circ\opL}(x^\diamond) + \mu \opC^*\circ \partial \iota_{\Cz}\circ \opC (x^\diamond)\\
  \overset{\eqref{eq:subdifferential-J-decomposed}}{\iff} &0_{\spX} \in \nabla\mathfrak{d}(x^\diamond) + \mu(\opL^* B^* B) \circ \partial \left(\Psi + \frac{1}{2}\norm{B \cdot}^2_{\sptildeZ}\right)^* \circ (B^*B\opL) (x^\diamond)\\ 
  &\quad \quad \quad + \mu \opL^* \circ \partial\Psi \circ \opL (x^\diamond) + \mu \opC^*\circ \partial \iota_{\Cz}\circ \opC (x^\diamond)\\
  \iff & \exists (\vdiamond, \wdiamond, \zdiamond)\in\spZ\times\spZ\times\spfrakZ \mbox{ s.t. }\\
  & \begin{cases}
    0_{\spX} \in \nabla\mathfrak{d}(x^\diamond) +
    \mu \opL^* B^* B v^\diamond
    + \mu \opL^* w^\diamond 
    + \mu \opC^* z^\diamond\\
    v^\diamond \in \partial \left(\Psi + \frac{1}{2} \norm{B \cdot}^2_{\sptildeZ}\right)^* (B^*B\opL x^\diamond)\\
     w^\diamond \in \partial \Psi (\opL x^\diamond)
   \\
    z^\diamond \in\partial \iota_{\Cz} (\opC x^\diamond)
  \end{cases}\\
  \overset{\eqref{eq:subdifferential-and-conjugate-J}}{\iff} & \exists (\vdiamond, \wdiamond, \zdiamond)\in\spZ\times\spZ\times\spfrakZ \mbox{ s.t. }\\
  & \begin{cases}
    0_{\spX} \in \nabla\mathfrak{d}(x^\diamond) +
    \mu \opL^* B^* B v^\diamond
    + \mu \opL^* w^\diamond 
    + \mu \opC^* z^\diamond\\
    0_{\spZ} \in - \mu B^*B\opL x^\diamond + \mu B^*B v^{\diamond} + \mu\partial \Psi (v^\diamond)\\
    0_{\spZ}\in -\mu\opL x^\diamond + \mu\partial \Psi^*(w^\diamond)\\
    0_{\spfrakZ}\in -\mu\opC x^\diamond + \mu\partial \iota_{\Cz}^*(z^{\diamond})
  \end{cases}\\
  \iff & \exists (\vdiamond, \wdiamond, \zdiamond)\in\spZ\times\spZ\times\spfrakZ \mbox{ s.t. } (0_{\spX},0_{\spZ},0_{\spZ},0_{\spfrakZ})\in (F+G)(x^\diamond,v^\diamond,w^\diamond,z^\diamond).
\end{align}

\noindent \textbf{(Proof of (b))} 
The following relation holds:
\begin{align}
  (\xi,\zeta,\eta,\varsigma) = T(x, v, w, z)
  \label{eq:characterize-T-by-inclusion}
  \iff 
  & \begin{cases}
  [\sigma \Id - \nabla \mathfrak{d}] (x) - \mu\opL^*B^*B v - \mu\opL^* w-\mu\opC^* z = \sigma \xi\\
  2\mu B^*B\opL \xi -\mu B^*B\opL x +(\tau \Id - \mu B^*B)(v) \in (\tau\Id +\mu\partial \Psi)(\zeta)\\
  2\mu \opL \xi - \mu\opL x +\mu w \in \mu(\Id + \partial \Psi^*) (\eta)\\
  2\mu \opC\xi - \mu \opC x + \mu z \in \mu (\Id + \partial \iota_{\Cz}^*)(\varsigma)
\end{cases}\\
\label{eq:characterize-T-by-F-G}
\iff & (\opP - F)(x,v,w,z) \in (\opP+G)(\xi, \zeta, \eta, \varsigma),
\end{align}
where we used the relation in \eqref{eq:subdifferential-and-prox} to derive \eqref{eq:characterize-T-by-inclusion} and $\opP$ in \eqref{eq:characterize-T-by-F-G} is defined in \eqref{eq:def-P} (Note: $(\Id +\partial \iota_{\Cz}*)^{-1}=\prox_{\iota_{\Cz}^*} = \Id - \prox_{\iota_{\Cz}} = \Id - P_{\Cz}$ follows from \cite[Example 12.25, Proposition 24.8(ix)]{CAaMOTiH}).
Hence, \eqref{eq:fixed-point-encode} is verified by
\begin{align}
  \xdiamond\in\setS\overset{\mbox{ Theorem \ref{thm:convergence-analysis}(a) }}{\iff} &\exists (\vdiamond, \wdiamond, \zdiamond)\in\spZ\times\spZ\times\spfrakZ \mbox{ s.t. } \\
  &(0_{\spX}, 0_{\spZ}, 0_{\spZ}, 0_{\spfrakZ}) \in (F+G)(x^\diamond, v^\diamond, w^\diamond, z^\diamond) \\
  \overset{\hphantom{\mbox{ Theorem \ref{thm:convergence-analysis}(a) }}}{\iff}
  &\exists (\vdiamond, \wdiamond, \zdiamond)\in\spZ\times\spZ\times\spfrakZ \mbox{ s.t. }\\
  &(\opP - F)(x^\diamond, v^\diamond, w^\diamond, z^\diamond) \in (\opP+G)(x^\diamond, v^\diamond, w^\diamond, z^\diamond)\\
  \overset{\hphantom{\mbox{ Theorem \ref{thm:convergence-analysis}(a) }}}{\overset{\eqref{eq:characterize-T-by-F-G}}{\iff}}
  &\exists (\vdiamond, \wdiamond, \zdiamond)\in\spZ\times\spZ\times\spfrakZ \mbox{ s.t. }
  (x^\diamond, v^\diamond, w^\diamond, z^\diamond) \in \fix T\\
  \overset{\hphantom{\mbox{ Theorem \ref{thm:convergence-analysis}(a) }}}{\overset{}{\iff}}
  & \xdiamond \in \Xi(\fix T).
\end{align}

\noindent
\textbf{(Proof of (c))} 
$\opP\succ \zeroMatrix_{\spH}$ is verified by Lemma \ref{lemma:matrix-inequality}(b).
Firstly, we show that $\opP^{-1}\circ F$ is $\frac{1}{\theta}$-cocoercive over $(\spH, \ip{\cdot}{\cdot}_{\opP}, \norm{\cdot}_{\opP})$.
Choose $h_1 = (x_1,v_1,w_1,z_1)\in\spH$ and $h_2 = (x_2,v_2,w_2,z_2)\in\spH$ arbitrary.
Noting that $F (h_1) - F (h_2) = (\nabla \mathfrak{d}(x_1) - \nabla \mathfrak{d}(x_2), \mu B^*B(v_1-v_2), 0_{\spZ}, 0_{\spfrakZ})$,
with $\mathfrak{Q}$ in \eqref{eq:def-Q} and the expression of $\opP^{-1}$ in \eqref{eq:block-exp-P}, we observe
{
\begin{align}
  &\norm{\opP^{-1}\circ F (h_1)- \opP^{-1}\circ F (h_2)}^2_{\opP}
  =\ip{\opP^{-1}(F (h_1)-  F (h_2))}{ F (h_1)-  F (h_2)}_{\spH}\\
  =& \ip{\opP^{-1}
  \begin{bmatrix}
    \nabla \mathfrak{d}(x_1) - \nabla \mathfrak{d}(x_2)\\ \mu B^*B(v_1-v_2)\\ 0_{\spZ}\\ 0_{\spfrakZ}
  \end{bmatrix} }{\begin{bmatrix}
    \nabla \mathfrak{d}(x_1) - \nabla \mathfrak{d}(x_2)\\ \mu B^*B(v_1-v_2)\\ 0_{\spZ}\\ 0_{\spfrakZ}
  \end{bmatrix}}_{\spH}\\
  =&\ip{\mathfrak{Q}^{-1} (\nabla \mathfrak{d}(x_1) - \nabla \mathfrak{d}(x_2), \mu B^*B(v_1-v_2))}{(\nabla \mathfrak{d}(x_1) - \nabla \mathfrak{d}(x_2), \mu B^*B(v_1-v_2))}_{\spX\times \spZ}\\
  \leq &\norm{\mathfrak{Q}^{-1}}_{\mathrm{op}} \norm{ (\nabla \mathfrak{d}(x_1) - \nabla \mathfrak{d}(x_2), \mu B^*B(v_1-v_2))}_{\spX\times \spZ}^2\\
  \leq &   \frac{\norm{\mathfrak{Q}^{-1}}_{\mathrm{op}}}{\rho} \cdot \ip{(x_1 - x_2, v_1-v_2)}{(\nabla \mathfrak{d}(x_1) - \nabla \mathfrak{d}(x_2), \mu B^*B(v_1-v_2))}_{\spX\times \spZ},
  \label{eq:cocoercivity-with-opnorm}
\end{align}
where we used Lemma \ref{lemma:coercive-constant}(b) on the last inequality.
}%
From Fact \ref{fact:properties-partitioned}(d) with \eqref{eq:Q-inv-partition}, \eqref{eq:bound-R1} and \eqref{eq:bound-R3}, $\norm{\mathfrak{Q}^{-1}}_{\mathrm{op}}$ can be bounded above by
\begin{align}
  &\norm{\mathfrak{Q}^{-1}}_{\mathrm{op}}
  =
   \norm{\begin{bmatrix}
    \mathfrak{R}_1 & \mathfrak{R}_2\\
    \mathfrak{R}_2^* & \mathfrak{R}_3
    \end{bmatrix}}_{\mathrm{op}}
  \leq 
  \norm{\mathfrak{R}_1}_{\mathrm{op}} + \norm{\mathfrak{R}_3}_{\mathrm{op}}\\
  \leq
  & \left(\sigma  - \frac{\mu^2}{\tau}\norm{B^* B\opL}^2_{\mathrm{op}} - \mu\norm{\opL^*\opL + \opC^*\opC }_{\mathrm{op}}\right)^{-1}
  + \left(\tau - \mu^2 \norm{B^* B \opL}_{\mathrm{op}}^2 (\sigma - \mu\norm{\opL^*\opL + \opC^*\opC }_{\mathrm{op}})^{-1}\right)^{-1}\\
  = & \frac{\tau}{\sigma\tau - \mu^2 \norm{B^* B \opL}_{\mathrm{op}}^2 - \tau\mu\norm{\opL^*\opL + \opC^*\opC }_{\mathrm{op}}}
  + \frac{\sigma - \mu\norm{\opL^*\opL + \opC^*\opC }_{\mathrm{op}}}{\sigma\tau - \mu^2 \norm{B^* B \opL}_{\mathrm{op}}^2 - \tau\mu\norm{\opL^*\opL + \opC^*\opC }_{\mathrm{op}}}
  \\
  = &
  \frac{ \sigma  + \tau- \mu\norm{\opL^*\opL + \opC^*\opC }_{\mathrm{op}}}{\sigma \tau - \tau\mu\norm{\opL^*\opL + \opC^*\opC }_{\mathrm{op}} - \mu^2 \norm{B^* B \opL}_{\mathrm{op}}^2} = \theta\rho.
  \label{eq:opnorm-ineq}
\end{align}
Combining \eqref{eq:cocoercivity-with-opnorm} and \eqref{eq:opnorm-ineq}, we have 
\begin{equation}
  \norm{\opP^{-1}\circ F (h_1)- \opP^{-1}\circ F (h_2)}^2_{\opP}
  \leq  \theta \ip{(x_1 - x_2, v_1-v_2)}{(\nabla \mathfrak{d}(x_1) - \nabla \mathfrak{d}(x_2), \mu B^*B(v_1-v_2))}_{\spX\times \spZ}.
\end{equation}
Furthermore, from the observation
\begin{align}
  &\ip{(x_1 - x_2, v_1-v_2)}{(\nabla \mathfrak{d}(x_1) - \nabla \mathfrak{d}(x_2), \mu B^*B(v_1-v_2))}_{\spX\times \spZ}
  =  \ip{h_1 - h_2}{F (h_1) - F (h_2)}_{\spH}\\
  = & \ip{h_1 - h_2}{\opP(\opP^{-1}\circ F (h_1) - \opP^{-1}\circ F (h_2))}_{\spH}
  = \ip{h_1 - h_2}{\opP^{-1}\circ F (h_1) - \opP^{-1}\circ F (h_2)}_{\opP},
\end{align}
we obtain
\begin{align}
  \norm{\opP^{-1}\circ F (h_1)- \opP^{-1}\circ F (h_2)}^2_{\opP}
  \leq &\theta \ip{(x_1 - x_2, v_1-v_2)}{(\nabla \mathfrak{d}(x_1) - \nabla \mathfrak{d}(x_2), \mu B^*B(v_1-v_2))}_{\spX\times \spZ}\\
  = & \theta\ip{h_1 - h_2}{\opP^{-1}\circ F (h_1) - \opP^{-1}\circ F (h_2)}_{\opP},
\end{align}
which implies that $\opP^{-1}\circ F$ is $\frac{1}{\theta}$-cocoercive over $(\spH,\ip{\cdot}{\cdot}_{\opP},\norm{\cdot}_{\opP})$.

Secondly, we show that $\opP^{-1}\circ G$ is maximally monotone over $(\spH,\ip{\cdot}{\cdot}_{\opP},\norm{\cdot}_{\opP})$.
Set 
\begin{align}
  &G_1:\spH\to2^{\spH}:(x,v,w,z) \mapsto \{0_{\spX}\}\times \mu\partial\Psi(v) \times \mu\partial \Psi^*(w) \times \mu\partial\iota_{\Cz}^*(z),\\
  &G_2:\spH\to\spH:(x,v,w,z)\mapsto  (\mu\opL B^*B v + \mu\opL^*w + \mu \opC^*z, -\mu B^*B \opL x, -\mu \opL x, -\mu\opC x).
\end{align}
Then, from \cite[Corollary 13.38, Proposition 20.23 and Theorem 20.25]{CAaMOTiH}, $G_1$ is maximally monotone over $(\spH,\ip{\cdot}{\cdot}_{\spH}, \norm{\cdot}_{\spH})$. Also, $G_2$ is a bounded linear skew-symmetric operator (i.e., $G_2\in\setLO{\spH}{\spH}$ and $G_2^* = -G_2$) and thus is maximally monotone over $(\spH,\ip{\cdot}{\cdot}_{\spH}, \norm{\cdot}_{\spH})$ \cite[Example 20.35]{CAaMOTiH}.
Moreover, from $\dom G_2 \coloneqq \{u\in\spH\mid G_2u\neq \emptyset\} = \spH$ in the sense of set-valued operator and from \cite[Corollary 25.5 (i)]{CAaMOTiH}, $G= G_1+G_2$ is maximally monotone over $(\spH,\ip{\cdot}{\cdot}_{\spH}, \norm{\cdot}_{\spH})$.
Therefore, the maximal monotonicity of $\opP^{-1}\circ G$ over $(\spH,\ip{\cdot}{\cdot}_{\opP},\norm{\cdot}_{\opP})$ is verified by \cite[Proposition 20.24]{CAaMOTiH}.

Finally, we show that $T$ is $\frac{2}{4-\theta}$-averaged nonexpansive over $(\spH,\ip{\cdot}{\cdot}_{\opP},\norm{\cdot}_{\opP})$.
By applying $\opP^{-1}\succ \zeroMatrix_{\spH}$ to the both sides in \eqref{eq:characterize-T-by-F-G}, we have
\begin{align}
  (\xi,\zeta,\eta,\varsigma) = T(x, v, w, z)
  \overset{\eqref{eq:characterize-T-by-F-G}}{\iff} &(\opP -F)(x, v, w, z) \in (\opP + G)(\xi,\zeta,\eta,\varsigma) \\
  \label{eq:inclusion-id-P-comp-F}
  \iff &(\Id - \opP^{-1}\circ F) (x, v, w, z) \in (\Id + \opP^{-1}\circ G)(\xi,\zeta,\eta,\varsigma).
\end{align}
Since $\opP^{-1}\circ G$ is maximally monotone over $(\spH,\ip{\cdot}{\cdot}_{\opP},\norm{\cdot}_{\opP})$, its resolvent $(\Id + \opP^{-1}\circ G)^{-1}$ is single-valued and $\frac{1}{2}$-averaged nonexpansive over $(\spH,\ip{\cdot}{\cdot}_{\opP},\norm{\cdot}_{\opP})$ \cite[Corollary 23.9]{CAaMOTiH}, and thus $T$ can be expressed as $T= (\Id + \opP^{-1}\circ G)^{-1} \circ (\Id - \opP^{-1}\circ F)$ by \eqref{eq:inclusion-id-P-comp-F}.
With $\tau>\frac{1}{2\rho}$ and \eqref{eq:positiveness-of-Schur-P-by-norm}, we get $\theta<2$ by \eqref{eq:cond-sigma-tau} and 
\begin{align}
  &\sigma > \mu\norm{\opL^*\opL +\opC^*\opC}_{\mathrm{op}}
  + \frac{2\rho\mu^2\norm{B^*B\opL}^2_{\mathrm{op}}+\tau}{2\rho\tau - 1} \\
  \iff & 
    \left(2\rho\tau - 1\right) \sigma > (2\rho\tau - 1)\mu\norm{\opL^*\opL +\opC^*\opC}_{\mathrm{op}}+ \tau +2\rho\mu^2\norm{B^*B\opL}^2_{\mathrm{op}}
  \\
  \iff &\sigma + \tau  - \mu\norm{\opL^*\opL +\opC^*\opC}_{\mathrm{op}} < 2\rho(\sigma \tau - \tau\mu\norm{\opL^*\opL +\opC^*\opC}_{\mathrm{op}} - \mu^2 \norm{B^* B \opL}_{\mathrm{op}}^2)
  \overset{\eqref{eq:positiveness-of-Schur-P-by-norm}}
  {\iff}
  \theta < 2.
  \label{eq:theta-less-than}
\end{align}
Since $\opP^{-1}\circ F$ is $\frac{1}{\theta}$-cocoercive over $(\spH,\ip{\cdot}{\cdot}_{\opP},\norm{\cdot}_{\opP})$ and $ 1<\frac{2}{\theta}$,
$(\Id - \opP^{-1}\circ F)$ is $\frac{\theta}{2}$-averaged nonexpansive over $(\spH,\ip{\cdot}{\cdot}_{\opP},\norm{\cdot}_{\opP})$ \cite[Proposition 4.39]{CAaMOTiH}.
Therefore, by Fact \ref{fact:composition-averaged}, we conclude that 
$T= (\Id + \opP^{-1}\circ G)^{-1} \circ (\Id - \opP^{-1}\circ F)$ is $\frac{2}{4-\theta}$-averaged nonexpansive over $(\spH,\ip{\cdot}{\cdot}_{\opP},\norm{\cdot}_{\opP})$.

\noindent
\textbf{(Proof of (d))} 
Thanks to Fact \ref{fact:KM-itr}, $(h_k)_{k\in\setN}$ is guaranteed to converge to $\bar{h}\in \fix T$ (Note: Since $\spH$ is finite-dimensional, $\norm{\cdot}_{\spH}$ and $\norm{\cdot}_{\opP}$ are equivalent).
Furthermore, from (b), we conclude that
$x_k=\Xi(h_k)$ converges to $\Xi(\bar{h})\in\setS$.

\setcounter{appnum}{6}
\setcounter{lemma}{0}
\setcounter{fact}{0}
\section{Proof of Proposition \ref{prop:separable-extension}}
\label{appendix:proof-separable-extension}
\sloppy
\textbf{(Proof of (a))}
Choose arbitrarily a twice continuously differentiable convex function $r\in\Gamma_0(\setR)$ satisfying \eqref{eq:cond-r}.
From the definition of $\widetilde{\fBeforeExtension}_i(\cdot;r)$ in \eqref{eq:def-ftilde-i-r},
for every $i\in\{1,2,\ldots,m\}$, $\widetilde{\fBeforeExtension}_i$ is twice continuously differentiable over $\setR$ with the gradient
\begin{equation}
  \label{eq:gradient-f-i-r}
  \widetilde{\fBeforeExtension}_i'(t;r) = \begin{cases}
    \fBeforeExtension_i''(l_i)(t - l_i) + \fBeforeExtension'(l_i) - r'(-t+l_i), & t\leq l_i\\
    \fBeforeExtension_i'(t), & l_i< t < h_i\\
    \fBeforeExtension_i''(h_i)(t - h_i) + \fBeforeExtension'(h_i) + r'(t-h_i), & t\geq h_i
  \end{cases}
\end{equation}
and the Hessian 
\begin{equation}
  \label{eq:hessian-f-i-r}
  \widetilde{\fBeforeExtension}_i''(t;r) = \begin{cases}
    \fBeforeExtension_i''(l_i) + r''(-t+l_i), & t\leq l_i\\
    \fBeforeExtension_i''(t), & l_i< t < h_i\\
    \fBeforeExtension_i''(h_i) + r''(t-h_i), & t\geq h_i.
  \end{cases}
\end{equation}
Hence, $\widetilde{\fBeforeExtension}(\cdot;r)$ in \eqref{eq:def-ftilde-r} is also twice continuously differentiable over $\spY$.
From \eqref{eq:hessian-f-i-r}, \eqref{eq:assumption-hessian-sup} and $\sup_{t\in\setR} r''(t)<\infty$ in \eqref{eq:cond-r},
we have 
\begin{equation}
  (\forall i \in\{1,2,\ldots,m\})\quad
  \sup_{t\in\setR} \widetilde{\fBeforeExtension}_i''(t;r) \leq \sup_{t\in\Pi_i} \fBeforeExtension_i''(t) +
  \sup_{t\in\setR} r''(t)  < \infty.
\end{equation}
Therefore, we obtain $ (\forall (u_1, u_2)\in\spY\times \spY)$
\begin{align}
  \norm{\nabla \widetilde{\fBeforeExtension}(u_1;r) - \nabla \widetilde{\fBeforeExtension}(u_2;r)}^2_\spY
  = &\sum_{i=1}^m \left(\widetilde{\fBeforeExtension}_i'([u_1]_i;r) - \widetilde{\fBeforeExtension}_i'([u_2]_i;r)\right)^2
  \label{eq:proof-separable-extension-1}
  \overset{\mbox{Fact \ref{fact:hessian-and-lipconst}}}{\leq} \sum_{i=1}^m \left(\sup_{t\in\setR} \widetilde{\fBeforeExtension}_i''(t;r)\right)^2 ([u_1]_i - [u_2]_i)^2\\
  \leq &\sum_{i=1}^m \left(\sup_{t\in\Pi_i} \fBeforeExtension_i''(t) +
  \sup_{t\in\setR} r''(t)\right)^2 ([u_1]_i - [u_2]_i)^2\\
  \leq &\underbrace{\left(
    \max\left\{
          \sup_{t\in \Pi_1} \fBeforeExtension_1''(t),\sup_{t\in \Pi_2} \fBeforeExtension_2''(t), \ldots , \sup_{t\in \Pi_m} \fBeforeExtension_m''(t)
          \right\}
          +  \sup_{t\in\setR} r''(t)
  \right)^2}_{=\lipconst{\nabla \widetilde{\fBeforeExtension}(\cdot;r)}^2}
          \sum_{i=1}^m ([u_1]_i - [u_2]_i)^2\\
  = &\lipconst{\nabla \widetilde{\fBeforeExtension}(\cdot;r)}^2\norm{u_1-u_2}^2_{\spY}.
\end{align}

\noindent
\textbf{(Proof of (b))}
From \eqref{eq:hessian-f-i-r} and \eqref{eq:cond-r}, we have
$
  (\forall i \in \{1,2,\ldots,m\})\
  \inf_{t\in\setR} \widetilde{\fBeforeExtension}''_i (t;r) = \inf_{t\in\Pi_i} \fBeforeExtension''_i (t),
$
which yields with \eqref{eq:def-lambda} 
\begin{equation}
  \label{eq:proof-separable-extension-2}
  (\forall 1\leq i \leq m) (\forall t\in\setR) \quad 
  \widetilde{\fBeforeExtension}_i''(t;r) \geq [\Lambda]_{i,i}\geq 0.
\end{equation}
Noting that $\nabla^2 \widetilde{\fBeforeExtension}(u;r)$ is diagonal and $[\nabla^2 \widetilde{\fBeforeExtension}(u;r)]_{i,i} = \widetilde{\fBeforeExtension}_i''([u]_i)$ holds for every $u\in\spY$, \eqref{eq:proof-separable-extension-2} implies the inequality
$
  (\forall u\in\spY) \ \nabla^2 \widetilde{\fBeforeExtension}(u;r)\succeq\Lambda.
$
Therefore, $\widetilde{\fBeforeExtension}(\cdot;r) - q_{\Lambda}\in\Gamma_0(\spX)$ is verified by 
\begin{align}
  (\forall u\in\spY) \quad \nabla^2\left(\widetilde{\fBeforeExtension}(\cdot;r) - q_{\Lambda}\right)(u) &=\nabla^2 \widetilde{\fBeforeExtension}(u;r) - \Lambda \succeq \zeroMatrix_{\spX}.
\end{align}

\noindent \textbf{(Proof of (c))} From the definition of $\Pi$ in Problem \ref{prob:relaxed-fidelity}, we have $x\in A^{-1}(\Pi)$ for every $x\in\opC^{-1}(\Cz)$. This implies that $Ax \in\Pi$ holds. From the definition of $\widetilde{\fBeforeExtension}_i$ and $\widetilde{\fBeforeExtension}$ in \eqref{eq:def-ftilde-i-r} and \eqref{eq:def-ftilde-r}, $\widetilde{\fBeforeExtension}(Ax;r) = \fBeforeExtension \circ A (x)$ is verified.  

\noindent \textbf{(Proof of (d))}
Assume that $\lim_{t\to +\infty}r(t)/t = \infty$ holds. Choose $i \in\{1,2,\ldots,m\}$ arbitrarily.
If $h_i < \infty$, we have
\begin{align}
  \lim_{t\to +\infty}\widetilde{\fBeforeExtension}_i(t;r) = &\lim_{t\to +\infty}\left[\frac{\fBeforeExtension_i''(h_i)}{2} (t -h_i)^2 + \fBeforeExtension_i'(h_i) (t-h_i) + \fBeforeExtension_i(h_i) +r(t-h_i)\right]\\
  \label{eq:proof-separable-extension-3}
  \geq &\lim_{t\to +\infty}\left[\fBeforeExtension_i'(h_i) (t-h_i) + \fBeforeExtension_i(h_i) +r(t-h_i)\right]\\
  = & \lim_{t\to +\infty}t\left[\fBeforeExtension_i'(h_i) \left(1-\frac{h_i}{t}\right) + \frac{\fBeforeExtension_i(h_i)}{t} +\frac{r(t-h_i)}{t}\right] = \infty,
\end{align}
where we used $\fBeforeExtension_i''(h_i)\geq 0$ (see \eqref{eq:def-lambda}) to derive \eqref{eq:proof-separable-extension-3} and the assumption $\lim_{t\to +\infty}r(t)/t = \infty$ to derive the last equation.
Hence, we have 
\begin{equation}
  \label{eq:proof-seperable-extension-4}
  \lim_{t\to +\infty} \widetilde{\fBeforeExtension}_i(t;r) = 
  \begin{cases}
    \infty, & h_i<\infty\\
    \lim_{t\to +\infty} \fBeforeExtension_i(t), & h_i = \infty.  
  \end{cases}
\end{equation}
With a similar observation, we also obtain
\begin{equation}
  \label{eq:proof-separable-extension-5}
  \lim_{t\to-\infty} \widetilde{\fBeforeExtension}_i(t;r)= 
  \begin{cases}
    \infty, & l_i>-\infty\\
    \lim_{t\to-\infty} \fBeforeExtension_i(t), & l_i = -\infty.  
  \end{cases}
\end{equation}
If $\fBeforeExtension$ is coercive, then $\lim_{|t|\to +\infty} \widetilde{\fBeforeExtension}_i (t;r) = \infty$ follows from \eqref{eq:proof-seperable-extension-4} and \eqref{eq:proof-separable-extension-5}, and thus $\widetilde{\fBeforeExtension}$ is also coercive.
Similarly, if $\fBeforeExtension$ is bounded below, then $\lim_{t\to +\infty} \widetilde{\fBeforeExtension}_i (t;r) > -\infty$ and $\lim_{t\to -\infty} \widetilde{\fBeforeExtension}_i (t;r) > -\infty$ follow from \eqref{eq:proof-seperable-extension-4} and \eqref{eq:proof-separable-extension-5}, and thus $f$ is also bounded below.

\setcounter{appnum}{7}
\setcounter{lemma}{0}
\setcounter{fact}{0}
\section{Proof of Corollary \ref{corollary:existence-minimizer-extension}}
\label{appendix:proof-reformulation}

\noindent \textbf{(Proof of (a))} See Proposition \ref{prop:separable-extension}(a).
\noindent 
\textbf{(Proof of (b))}  The coincidence of solution sets of \eqref{eq:before-reformulated-model} and \eqref{eq:reformulated-model} follows from $(\forall x\in\opC^{-1}(\Cz)) \ \fBeforeExtension \circ A(x) = \widetilde{\fBeforeExtension}\circ A(x)$ by Proposition \ref{prop:separable-extension}(c).
Moreover, from (a), the model \eqref{eq:reformulated-model} is an instance of Problem \ref{prob:cLiGME-w-general-fidelity}.
\noindent \textbf{(Proof of (c))} Since $\widetilde{\fBeforeExtension}$ is $1$-strongly convex relative to $q_{\Lambda}$ with $\Lambda$ in \eqref{eq:def-lambda}, it follows from $(C_0)\implies (C_1)\implies (C_2)$ in Proposition~\ref{prop:overall-convexity}(c).

  \noindent\textbf{(Proof of (d))}
  Let $\widetilde{r}\in\Gamma_0(\setR)$ satisfy the condition (4.5) and $\lim_{t\to\infty}\widetilde{r}(t)/t=\infty$ (see Footnote \ref{footnote:existence-r} in Proposition \ref{prop:separable-extension}(d) for an example of $\widetilde{r}$). It suffices for ensuring the existence of a minimizer of the models \eqref{eq:before-reformulated-model} and \eqref{eq:reformulated-model} to show the existence of a minimizer of another model:
  \begin{equation}
    \label{eq:another-reformulated}
    \minimize_{\opC x\in\Cz} \widetilde{\fBeforeExtension}(\cdot;\widetilde{r})\circ A(x) + \mu \Psi_B\circ\opL (x)
  \end{equation}
  because $(\forall x \in \opC^{-1}(\Cz)) \widetilde{\fBeforeExtension}(Ax;\widetilde{r}) = \fBeforeExtension \circ A (x) = \widetilde{\fBeforeExtension}\circ A(x)$ holds by Proposition \ref{prop:separable-extension}(c). 
  By Proposition \ref{prop:separable-extension}(d),
  $\widetilde{\fBeforeExtension}(\cdot;\widetilde{r})$ is coercive (bounded below) if $\fBeforeExtension$ is coercive (bounded below). 
  Therefore, if one of the conditions (i)-(iv) in Corollary~\ref{corollary:existence-minimizer-extension}(d) holds, Theorem \ref{thm:existence-solution} verifies the existence of a minimizer of the model \eqref{eq:another-reformulated}.

  \noindent\textbf{(Proof of (e))}
  Recall that the model \eqref{eq:reformulated-model} is an instance of Problem \ref{prob:cLiGME-w-general-fidelity}. Moreover, Assumption \ref{assumption:alg}(i) holds by (b), Assumption \ref{assumption:alg}(ii) holds by (c) and the assumption (i) in (e), and Assumption \ref{assumption:alg}(iii)-(iv) hold respectively by the assumptions (ii)-(iii) in (e). Therefore, (e) is verified by Theorem \ref{thm:convergence-analysis}.

\setcounter{appnum}{8}
\setcounter{lemma}{0}
\setcounter{fact}{0}
\section{Proof of Lemma \ref{lemma:monotonicity-declipping}}
\label{appendix:proof-clipping-hessian}
The expressions of first-order and second-order derivatives in \eqref{eq:clip-gradient} and \eqref{eq:hessian-clipping-positive} are given in \cite[Proof of Theorem 1]{banerjee2024}.
Let $\mathcal{R}:\setR\to\setR$ be the inverse of  Mill's ratio 
$
  \mathcal{R}(t)\coloneqq \frac{\exp\left(-\frac{t^2}{2}\right)}{\int_{t}^{\infty} \exp\left(-\frac{\alpha^2}{2}\right)\mathrm{d}\alpha}.
$
Then we have 
\begin{equation}
-\frac{p(t)}{\operatorname{Pr}(t)}
  = - \frac{\exp\left(-\frac{t^2}{2s^2}\right)}{\int_{-\infty}^{t} \exp\left(-\frac{\alpha^2}{2s^2}\right) \mathrm{d}\alpha}
  \overset{\widetilde{\alpha}=\frac{\alpha}{s}}{=} -\frac{\exp\left(-\frac{1}{2}(\frac{-t}{s})^2\right)}{s\int_{\frac{-t}{s}}^{\infty} \exp\left(-\frac{\widetilde{\alpha}^2}{2}\right) \mathrm{d}\widetilde{\alpha}}
  =- \frac{1}{s}\mathcal{R}\left(\frac{-t}{s}\right)
\end{equation}
and
$
\left(-\frac{p(\cdot)}{\operatorname{Pr}(\cdot)}\right)''(t) = -\frac{1}{s^3}\mathcal{R}''\left(\frac{-t}{s}\right).
$
Since $(\forall t \in\setR)\ \mathcal{R}''(t)>0$
holds from \cite[(4)]{sampford1953}, we have the inequality $(\forall t\in\setR)\ \left(-\frac{p(\cdot)}{\operatorname{Pr}(\cdot)}\right)''(t)<0$, which implies that $\left(-\frac{p(\cdot)}{\operatorname{Pr}(\cdot)}\right)'$ in \eqref{eq:hessian-clipping-positive} is monotonically decreasing over $\setR$.
Hence, we obtain 
\begin{align}
  \sup_{t\in\Pi_i}{ \fBeforeExtension_i^{\lrangle{\mathrm{clip}}}}''(t) &=
\begin{cases}
      \sup_{t\in\setR}\frac{1}{s^2},  &i\in \mathbb{S}_{uc}\\
      \sup_{t\in [ \vartheta - \varpi , \infty)}\left(-\frac{p(\cdot)}{\operatorname{Pr}(\cdot)}\right)'(t - \vartheta), &i\in \mathbb{S}_{+}\\
      \sup_{t\in (-\infty, -\vartheta + \varpi]}
      \left(-\frac{p(\cdot)}{\operatorname{Pr}(\cdot)}\right)'(-t - \vartheta), &i\in \mathbb{S}_{-}
      \end{cases}\\
      &=
\begin{cases}
      \frac{1}{s^2},  &i\in \mathbb{S}_{uc}\\
      \left(-\frac{p(\cdot)}{\operatorname{Pr}(\cdot)}\right)'(-\varpi), &i\in \mathbb{S}_{+}\cap\mathbb{S}_{-}.
      \end{cases}
\end{align}
Moreover, from $\lim_{t\to \infty} t p(t)=0$, $\lim_{t\to\infty} \operatorname{Pr}(t)=1$ and $\lim_{t\to\infty} p(t)=0$, we obtain
$
  \lim_{t\to\infty}\left(-\frac{p(\cdot)}{\operatorname{Pr}(\cdot)}\right)'(t)
  = \lim_{t\to\infty} \frac{p(t)}{\operatorname{Pr}(t)^2}\left( p(t) + \frac{t}{s^2} \operatorname{Pr}(t)\right)
  =\lim_{t\to\infty} \frac{t p(t)}{s^2\operatorname{Pr}(t)}
  =0
$.
Therefore, we have
\begin{align}
  \inf_{t\in\Pi_i}{ \fBeforeExtension_i^{\lrangle{\mathrm{clip}}}}''(t) &=
\begin{cases}
      \inf_{t\in\setR}\frac{1}{s^2},  &i\in \mathbb{S}_{uc}\\
      \inf_{t\in [ \vartheta - \varpi , \infty)}\left(-\frac{p(\cdot)}{\operatorname{Pr}(\cdot)}\right)'(t - \vartheta), &i\in \mathbb{S}_{+}\\
      \inf_{t\in (-\infty, -\vartheta + \varpi]}
      \left(-\frac{p(\cdot)}{\operatorname{Pr}(\cdot)}\right)'(-t - \vartheta), &i\in \mathbb{S}_{-}
      \end{cases}\\
      &=
  \begin{cases}
      \frac{1}{s^2},  &i\in \mathbb{S}_{uc}\\
      \lim_{t\to\infty}\left(-\frac{p(\cdot)}{\operatorname{Pr}(\cdot)}\right)'(t-\vartheta), &i\in \mathbb{S}_{+}\\
      \lim_{t\to-\infty}\left(-\frac{p(\cdot)}{\operatorname{Pr}(\cdot)}\right)'(-t-\vartheta), &i\in \mathbb{S}_{-}
    \end{cases}\\
      &=\begin{cases}
      \frac{1}{s^2},  &i\in \mathbb{S}_{uc}\\
      0, &i\in \mathbb{S}_{+}\cup \mathbb{S}_{-}.
    \end{cases}
\end{align}

\bibliographystyle{IEEEbib}
\bibliography{references}

\begin{thebibliography}{10}

\bibitem{nashed2002}
M.~Z. Nashed and O.~Scherzer, Eds.,
\newblock {\em {Inverse Problems, Image Analysis, and Medical Imaging}},
\newblock American Methematical Society, 2002.

\bibitem{theodoridis2020}
S.~Theodoridis,
\newblock {\em Machine Learning: A Bayesian and Optimization Perspective},
\newblock Academic Press, 2 edition, 2020.

\bibitem{byrne1993}
C.~Byrne,
\newblock ``{Iterative image reconstruction algorithms based on cross-entropy
  minimization},''
\newblock {\em IEEE Transactions on Image Processing}, vol. 2, no. 1, pp.
  96--103, 1993.

\bibitem{chouzenoux2015}
E.~Chouzenoux, A.~Jezierska, J.-C. Pesquet, and H.~Talbot,
\newblock ``A convex approach for image restoration with exact
  {Poisson--Gaussian} likelihood,''
\newblock {\em SIAM Journal on Imaging Sciences}, vol. 8, no. 4, pp.
  2662--2682, 2015.

\bibitem{banerjee2024}
S.~Banerjee, S.~Peddabomma, R.~Srivastava, and A.~Rajwade,
\newblock ``{A likelihood based method for compressive signal recovery under
  Gaussian and saturation noise},''
\newblock {\em Signal Processing}, vol. 217, 2024.

\bibitem{tibshirani1996}
R.~Tibshirani,
\newblock ``Regression shrinkage and selection via the lasso,''
\newblock {\em Journal of the Royal Statistical Society. Series B
  (Methodological)}, vol. 58, no. 1, pp. 267--288, 1996.

\bibitem{rudin1992}
L.~I. Rudin, S.~Osher, and E.~Fatemi,
\newblock ``{Nonlinear total variation based noise removal algorithms},''
\newblock {\em Physica D: Nonlinear Phenomena}, vol. 60, no. 1-4, pp. 259--268,
  1992.

\bibitem{candes2008}
E.~J. Cand{\`e}s, M.~B. Wakin, and S.~P. Boyd,
\newblock ``Enhancing sparsity by reweighted $\ell_1$ minimization,''
\newblock {\em Journal of Fourier Analysis and Applications}, vol. 14, no. 5,
  pp. 877--905, 2008.

\bibitem{chartrand2007}
R.~Chartrand,
\newblock ``Exact reconstruction of sparse signals via nonconvex
  minimization,''
\newblock {\em IEEE Signal Processing Letters}, vol. 14, no. 10, pp. 707--710,
  2007.

\bibitem{cLiGME}
W.~Yata, M.~Yamagishi, and I.~Yamada,
\newblock ``{A constrained LiGME model and its proximal splitting algorithm
  under overall convexity condition},''
\newblock {\em Journal of Applied and Numerical Optimization}, vol. 4, no. 2,
  pp. 245--271, 2022.

\bibitem{LiGME}
J.~Abe, M.~Yamagishi, and I.~Yamada,
\newblock ``{Linearly involved generalized {Moreau} enhanced models and their
  proximal splitting algorithm under overall convexity condition},''
\newblock {\em Inverse Problems}, vol. 36, no. 3, pp. 035012, 2020.

\bibitem{selesnick2017}
I.~Selesnick,
\newblock ``Sparse regularization via convex analysis,''
\newblock {\em IEEE Transactions on Signal Processing}, vol. 65, no. 17, pp.
  4481--4494, 2017.

\bibitem{Chen2023}
Y.~Chen, M.~Yamagishi, and I.~Yamada,
\newblock ``A unified design of generalized {Moreau} enhancement matrix for
  sparsity aware {LiGME} models,''
\newblock {\em IEICE Transactions on Fundamentals of Electronics,
  Communications and Computer Sciences}, vol. E106.A, no. 8, pp. 1025--1036,
  2023.

\bibitem{feng2020}
Y.~Feng, B.~Ding, H.~Graber, and I.~Selesnick,
\newblock ``Transient artifacts suppression in time series via convex
  analysis,''
\newblock in {\em {Signal Processing in Medicine and Biology: Emerging Trends
  in Research and Applications}}, I.~Obeid, I.~Selesnick, and J.~Picone, Eds.,
  pp. 107--138. Springer International Publishing, 2020.

\bibitem{hou2021}
F.~Hou, I.~Selesnick, J.~Chen, and G.~Dong,
\newblock ``Fault diagnosis for rolling bearings under unknown time-varying
  speed conditions with sparse representation,''
\newblock {\em Journal of Sound and Vibration}, vol. 494, pp. 115854, 2021.

\bibitem{sun2022}
S.~Sun, T.~Wang, H.~Yang, and F.~Chu,
\newblock ``{Damage identification of wind turbine blades using an adaptive
  method for compressive beamforming based on the generalized minimax-concave
  penalty function},''
\newblock {\em Renewable Energy}, vol. 181, pp. 59--70, 2022.

\bibitem{kitahara2021}
D.~Kitahara, R.~Kato, H.~Kuroda, and A.~Hirabayashi,
\newblock ``Multi-contrast {CSMRI} using common edge structures with {LiGME}
  model,''
\newblock in {\em 2021 29th European Signal Processing Conference (EUSIPCO)},
  2021, pp. 2119--2123.

\bibitem{liu2023}
X.~Liu, A.~J. Molstad, and E.~C. Chi,
\newblock ``{A convex-nonconvex strategy for grouped variable selection},''
\newblock {\em Electronic Journal of Statistics}, vol. 17, no. 2, pp. 2912 --
  2961, 2023.

\bibitem{shoji2025}
S.~Shoji, W.~Yata, K.~Kume, and I.~Yamada,
\newblock ``An {LiGME} regularizer of designated isolated minimizers - {An}
  application to discrete-valued signal estimation,''
\newblock {\em IEICE Transactions on Fundamentals of Electronics,
  Communications and Computer Sciences}, 2025.

\bibitem{katsuma2025}
A.~Katsuma, S.~Kyochi, S.~Ono, and I.~Selesnick,
\newblock ``Sparsity-enhanced multilayered non-convex regularization with
  epigraphical relaxation for debiased signal recovery,''
\newblock {\em IEEE Transactions on Signal Processing}, pp. 1--16, 2025.

\bibitem{shabili2021}
A.~H. Al-Shabili, Y.~Feng, and I.~Selesnick,
\newblock ``Sharpening sparse regularizers via smoothing,''
\newblock {\em IEEE Open Journal of Signal Processing}, vol. 2, pp. 396--409,
  2021.

\bibitem{kuroda2024}
H.~Kuroda,
\newblock ``A convex-nonconvex framework for enhancing minimization induced
  penalties,''
\newblock {\em arXiv preprint (2407.14819v3)}, 2024.

\bibitem{heng2025}
Q.~Heng, X.~Liu, and E.~C. Chi,
\newblock ``Anderson accelerated operator splitting methods for
  convex-nonconvex regularized problems,''
\newblock {\em arXiv preprint (2502.14269v1)}, 2025.

\bibitem{zhang2023}
Y.~Zhang and I.~Yamada,
\newblock ``A unified framework for solving a general class of nonconvexly
  regularized convex models,''
\newblock {\em IEEE Transactions on Signal Processing}, vol. 23, 2023.

\bibitem{zhang2025}
Y.~Zhang and I.~Yamada,
\newblock ``Piecewise linearity of min-norm solution map of a nonconvexly
  regularized convex sparse model,''
\newblock {\em IEEE Transactions on Information Theory (accepted for
  publication [August 2025])},
\newblock See also arXiv:2311.18438.

\bibitem{efron2004}
B.~Efron, T.~Hastie, I.~Johnstone, and R.~Tibshirani,
\newblock ``Least angle regression,''
\newblock {\em Annals of Statistics}, vol. 32, no. 2, pp. 407 -- 499, 2004.

\bibitem{tibishirani2013}
R.~J. Tibshirani,
\newblock ``The lasso problem and uniqueness,''
\newblock {\em Electronic Journal of Statistics}, vol. 7, no. none, pp. 1456 --
  1490, 2013.

\bibitem{tao1997}
P.~D. Tao and L.~T.~H. An,
\newblock ``Convex analysis approach to {DC} programming: theory, algorithms
  and applications,''
\newblock {\em Acta mathematica vietnamica}, vol. 22, no. 1, 1997.

\bibitem{thi2018}
H.~A. Le~Thi and T.~Pham~Dinh,
\newblock ``{DC programming and DCA: thirty years of developments},''
\newblock {\em Mathematical Programming}, vol. 169, no. 1, 2018.

\bibitem{yata2025}
W.~Yata, K.~Kume, and I.~Yamada,
\newblock ``A convexity preserving nonconvex regularization for inverse
  problems under non-{Gaussian} noise,''
\newblock in {\em 2025 33rd European Signal Processing Conference (EUSIPCO)},
  2025.

\bibitem{CAaMOTiH}
H.~H. Bauschke and P.~L. Combettes,
\newblock {\em Convex Analysis and Monotone Operator Theory in Hilbert Spaces},
\newblock Springer, 2nd edition, 2017.

\bibitem{Auslender1996}
A.~Auslender,
\newblock ``Noncoercive optimization problems,''
\newblock {\em Mathematics of Operations Research}, vol. 21, no. 4, pp.
  769--782, 1996.

\bibitem{rockafellar2009variational}
R.~T. Rockafellar and R.~J.-B. Wets,
\newblock {\em {Variational Analysis}},
\newblock Springer Berlin Heidelberg, 2009.

\bibitem{auslender2003}
A.~Auslender and M.~Teboulle,
\newblock {\em Asymptotic Cones and Functions in Optimization and Variational
  Inequalities},
\newblock Springer, 2003.

\bibitem{groetsch1972}
C.~Groetsch,
\newblock ``{A note on segmenting Mann iterates},''
\newblock {\em Journal of Mathematical Analysis and Applications}, vol. 40, no.
  2, 1972.

\bibitem{haihao2018}
H.~Lu, R.~M. Freund, and Y.~Nesterov,
\newblock ``Relatively smooth convex optimization by first-order methods, and
  applications,''
\newblock {\em SIAM Journal on Optimization}, vol. 28, no. 1, pp. 333--354,
  2018.

\bibitem{davis2018}
D.~Davis, D.~Drusvyatskiy, and K.~J. MacPhee,
\newblock ``{Stochastic model-based minimization under high-order growth},''
\newblock {\em arXiv preprint (1807.00255v1)}, 2018.

\bibitem{liu2022}
X.~Liu and E.~C. Chi,
\newblock ``{Revisiting convexity-preserving signal recovery with the linearly
  involved GMC penalty},''
\newblock {\em Pattern Recognition Letters}, vol. 156, pp. 60--66, 2022.

\bibitem{Goberna2010}
M.~A. Goberna, E.~González, J.~E. Martínez-Legaz, and M.~I. Todorov,
\newblock ``Motzkin decomposition of closed convex sets,''
\newblock {\em Journal of Mathematical Analysis and Applications}, vol. 364,
  no. 1, 2010.

\bibitem{royset2021}
J.~O. Royset and R.~J.-B. Wets,
\newblock {\em {An Optimization Primer}},
\newblock Springer, 2021.

\bibitem{triet2007}
T.~Le, R.~Chartrand, and T.~J. Asaki,
\newblock ``A variational approach to reconstructing images corrupted by
  poisson noise,''
\newblock {\em Journal of Mathematical Imaging and Vision}, vol. 27, no. 3, pp.
  257--263, 2007.

\bibitem{zanella2009}
R.~Zanella, P.~Boccacci, L.~Zanni, and M.~Bertero,
\newblock ``{Efficient gradient projection methods for edge-preserving removal
  of Poisson noise},''
\newblock {\em Inverse Problems}, vol. 25, no. 4, pp. 045010, 2009.

\bibitem{caroline2009}
C.~Chaux, J.-C. Pesquet, and N.~Pustelnik,
\newblock ``Nested iterative algorithms for convex constrained image recovery
  problems,''
\newblock {\em SIAM Journal on Imaging Sciences}, vol. 2, no. 2, pp. 730--762,
  2009.

\bibitem{dupe2009}
F.-X. Dupe, J.~M. Fadili, and J.-L. Starck,
\newblock ``A proximal iteration for deconvolving {Poisson} noisy images using
  sparse representations,''
\newblock {\em IEEE Transactions on Image Processing}, vol. 18, no. 2, pp.
  310--321, 2009.

\bibitem{foi2009}
A.~Foi,
\newblock ``{Clipped noisy images: Heteroskedastic modeling and practical
  denoising},''
\newblock {\em Signal Processing}, vol. 89, no. 12, pp. 2609--2629, 2009.

\bibitem{zaviska2021}
P.~Záviška, P.~Rajmic, A.~Ozerov, and L.~Rencker,
\newblock ``A survey and an extensive evaluation of popular audio declipping
  methods,''
\newblock {\em IEEE Journal of Selected Topics in Signal Processing}, vol. 15,
  no. 1, pp. 5--24, 2021.

\bibitem{laska2011}
J.~N. Laska, P.~T. Boufounos, M.~A. Davenport, and R.~G. Baraniuk,
\newblock ``Democracy in action: Quantization, saturation, and compressive
  sensing,''
\newblock {\em Applied and Computational Harmonic Analysis}, vol. 31, no. 3,
  pp. 429--443, 2011.

\bibitem{rao2007}
K.~R. Rao and P.~Yip,
\newblock {\em Discrete Cosine Transform: Algorithms, Advantages,
  Applications},
\newblock Academic press, 1990.

\bibitem{gallier2011}
J.~Gallier,
\newblock {\em {Geometric methods and applications: for computer science and
  engineering}},
\newblock Springer-Verlag, 2011.

\bibitem{horn2012}
R.~A. Horn and C.~R. Johnson,
\newblock {\em Matrix Analysis},
\newblock Cambridge Univ. Press, 2 edition, 2012.

\bibitem{halko2011}
N.~Halko, P.~G. Martinsson, and J.~A. Tropp,
\newblock ``Finding structure with randomness: Probabilistic algorithms for
  constructing approximate matrix decompositions,''
\newblock {\em SIAM Review}, vol. 53, no. 2, pp. 217--288, 2011.

\bibitem{ogura2002}
N.~Ogura and I.~Yamada,
\newblock ``{Non-strictly convex minimization over the fixed point set of an
  asymptotically shrinking nonexpansive mapping},''
\newblock {\em Numerical Functional Analysis and Optimization}, vol. 23, pp.
  113–137, 2002.

\bibitem{combettes2015}
P.~L. Combettes and I.~Yamada,
\newblock ``Compositions and convex combinations of averaged nonexpansive
  operators,''
\newblock {\em J. Math. Anal. Appl.}, vol. 425, no. 1, pp. 55--70, 2015.

\bibitem{nesterov2013}
Y.~Nesterov,
\newblock {\em {Introductory Lectures on Convex Optimization: A Basic Course}},
\newblock Springer, 2013.

\bibitem{aros2021}
P.~P{\'e}rez-Aros and E.~Vilches,
\newblock ``An enhanced {Baillon--Haddad} theorem for convex functions defined
  on convex sets,''
\newblock {\em Applied Mathematics \& Optimization}, vol. 83, no. 3, pp.
  2241--2252, 2021.

\bibitem{condat2013}
L.~Condat,
\newblock ``A primal--dual splitting method for convex optimization involving
  {Lipschitzian}, proximable and linear composite terms,''
\newblock {\em Journal of Optimization Theory and Applications}, vol. 158, no.
  2, pp. 460--479, 2013.

\bibitem{sampford1953}
M.~R. Sampford,
\newblock ``{Some inequalities on Mill's ratio and related functions},''
\newblock {\em The Annals of Mathematical Statistics}, vol. 24, no. 1, pp.
  130--132, 1953.

\end{thebibliography}
\end{document}